\RequirePackage{etoolbox}
\csdef{input@path}{{style/}{graphics/}}
\documentclass[seceqn,MSNbibl,number,citesort,dvips]{arxbj}
\makeatletter
   \@ifpackageloaded{graphicx}{}{\usepackage{graphicx}}
\makeatother
\usepackage{mathbh}
\usepackage{graphicx}

\aid{0}
\volume{21}
\issue{2}
\pubyear{2015}
\firstpage{1166}
\lastpage{1199}
\doi{10.3150/14-BEJ601} 

\makeatletter

\newcommand{\rrVert}{\Vert}
\newcommand{\rrvert}{\vert}
\newcommand{\llVert}{\Vert}
\newcommand{\llvert}{\vert}

\newtheorem{theorem}{Theorem}[section]
\newtheorem{proposition}{Proposition}[section]
\newtheorem{lemma}[theorem]{Lemma}
\newtheorem{corollary}{Corollary}[section]

\newremark{remark}{Remark}[section]
\newproclaim{definition}[theorem]{Definition}
\newremark{example}[theorem]{Example}

\newcommand{\xrightarrow}[2]{{}\mathop{\hbox to 1cm{\rightarrowfill}}\limits^{#2}_{#1}{}}
\newcommand{\Dom}{\operatorname{Dom}}
\newcommand{\Tr}{\operatorname{Tr}}
\newcommand{\intt}{\operatorname{int}}
\newcommand{\Bd}{\operatorname{Bd}}

\newcommand{\fraca}[2]{{#1}/{#2}}

\newcommand{\fracb}[2]{(#1)/{#2}}

\makeatother

\begin{document}
\begin{frontmatter}

\title{Backward stochastic variational inequalities on random interval}
\runtitle{Backward SVIs on random interval}

\begin{aug}
\author[1,2]{\inits{L.}\fnms{Lucian}~\snm{Maticiuc}\corref{}\thanksref{1,2,e1}\ead[label=e1,mark]{lucian.maticiuc@uaic.ro}} \and
\author[1,3]{\inits{A.}\fnms{Aurel}~\snm{R\u{a}\c{s}canu}\thanksref{1,3,e2}\ead[label=e2,mark]{aurel.rascanu@uaic.ro}}
\address[1]{Faculty of Mathematics, ``Alexandru Ioan Cuza'' University, Carol I Blvd., no. 11, Iasi, 700506, Romania.
\printead{e1,e2}}
\address[2]{Department of Mathematics, ``Gheorghe Asachi'' Technical University, Carol I Blvd., no. 11, Iasi, 700506, Romania}
\address[3]{``Octav Mayer'' Mathematics Institute of the Romanian Academy, Iasi branch, Carol I Blvd., no. 8, Iasi, 700506, Romania}
\end{aug}

\received{\smonth{6} \syear{2012}}
\revised{\smonth{1} \syear{2014}}

\begin{abstract}
The aim of this paper is to study, in the infinite dimensional framework, the
existence and uniqueness for the solution of the following multivalued
generalized backward stochastic differential equation, considered on a random,
possibly infinite, time interval:
\[
\cases{\displaystyle -\mathrm{d}Y_{t}+\partial_{y}\Psi (
t,Y_{t} )\,\mathrm{d}Q_{t}\ni\Phi ( t,Y_{t},Z_{t}
)\,\mathrm{d}Q_{t}-Z_{t}\,\mathrm{d}W_{t},\qquad
0\leq t<\tau,
\cr
\displaystyle{Y_{\tau}=\eta,} }
\]
where $\tau$ is a stopping time, $Q$ is a progressively measurable increasing
continuous stochastic process and $\partial_{y}\Psi$ is the subdifferential of
the convex lower semicontinuous function $y\longmapsto\Psi (  t,y )$.

As applications, we obtain from our main results applied for suitable convex
functions, the existence for some backward stochastic partial differential
equations with Dirichlet or Neumann boundary conditions.
\end{abstract}


\end{frontmatter}

\section{Introduction}
In this paper, we are interested to prove the existence and uniqueness of a
triple $ (  Y,Z,K )  $ which is the solution for the following
generalized backward stochastic variational inequality (BSVI for short)
considered in the Hilbert space framework:
%
\begin{equation}
\label{GBSVI 1}
\hspace*{-10pt}\cases{\displaystyle Y_{t}+{\int_{t\wedge\tau}^{\tau}}
\,\mathrm{d}K_{s}= \eta+{\int_{t\wedge\tau}^{\tau}}
\bigl[ F ( s,Y_{s},Z_{s} )\,\mathrm{d}s+G (
s,Y_{s} )\,\mathrm{d}A_{s} \bigr] -{\int
_{t\wedge\tau}^{\tau}} Z_{s}\,\mathrm{d}W_{s},
\qquad\mbox{a.s.,}\vspace*{2pt}
\cr
\displaystyle \mathrm{d}K_{t}\in \partial
\varphi ( Y_{t} )\,\mathrm{d}t+\partial\psi ( Y_{t} )\,
\mathrm{d}A_{t}, \qquad\forall t\geq0, }
\end{equation}
where $ \{  W_{t}\dvtx t\geq0 \}  $ is a cylindrical Wiener process,
$\partial\varphi$, $\partial\psi$ are the subdifferentials of a convex lower
semicontinuous functions $\varphi$, $\psi$, $ \{  A_{t}\dvtx t\geq0 \}  $
is a progressively measurable increasing continuous stochastic process, and
$\tau$ is a stopping time.

In fact, we will define and prove the existence of the solution for an
equivalent form of (\ref{GBSVI 1}):
%
\begin{equation}
\label{GBSVI 2} \cases{\displaystyle Y_{t}+{\int_{t}^{\infty}}
\,\mathrm{d}K_{s}= \eta+{\int_{t}^{\infty}}
\Phi ( s,Y_{s},Z_{s} )\,\mathrm{d}Q_{s}-{\int
_{t}^{\infty}}Z_{s}\,\mathrm{d}W_{s},
\qquad \mbox{a.s., }t\geq0,\vspace*{2pt}
\cr
\displaystyle \mathrm{d}K_{t}\in
\partial_{y}\Psi ( t,Y_{t} )\,\mathrm{d}Q_{t},
\qquad\mbox{on }[0,\infty), }
\end{equation}
with $Q$, $\Phi$ and $\Psi$ adequately defined. The notation $\mathrm{d}K_{t}
\in\partial_{y}\Psi (  t,Y_{t} )\,\mathrm{d}Q_{t}$ means that $Y$ is a
continuous stochastic process and for any continuous stochastic process
$X\dvtx \mathbb{R}_{+}\rightarrow H$ and any $0\leq s_{1}\leq s_{2}$, the bounded
variation of $K$ on $ [  s_{1},s_{2} ]  $ is finite and the following
inequality holds:
\[
\int_{s_{1}}^{s_{2}} \langle X_{r}-Y_{r},\mathrm{d}K_{r}
\rangle +\int_{s_{1}%
}^{s_{2}}\Psi ( r,Y_{r} )
\,\mathrm{d}Q_{r}\leq\int_{s_{1}}^{s_{2}}\Psi
( r,Y_{r} )\,\mathrm{d}Q_{r} , \qquad\mbox{a.s.}
\]
The study of the backward stochastic differential equations (BSDEs for short)
in the finite dimensional case (equation of type (\ref{GBSVI 1}) with $A$ and
$\varphi$ equal to $0$) was initiated by Pardoux and Peng \cite{pa-pe/90} (see also Pardoux and Peng \cite{pa-pe/92}). The authors have proved the
existence and the uniqueness of the solution for the BSDE on fixed time
interval, under the assumption of Lipschitz continuity of $F$ with respect to
$y$ and $z$ and square integrability of $\eta$ and $F (  t,0,0 )  $.
The case of BSDEs on random time interval (possibly infinite), under weaker
assumptions on the data, have been treated by Darling and Pardoux \cite{da-pa/97}, where it is obtained, as application, the existence of a
continuous viscosity solution to the elliptic partial differential equations
(PDEs) with Dirichlet boundary conditions.

The more general case of scalar BSDEs with one-sided reflection and
associated optimal control problems was considered by El Karoui, Kapoudjian, Pardoux, Peng and Quenez \cite{ka-ka-pa/97} and with
two-sided reflection associated with stochastic game problem by Cvitani\'c and Karatzas \cite{cv-ka/96}.

When the obstacles are fixed, the reflected BSDE become a particular case of
BSVI of type (\ref{GBSVI 1}), by taking $\Psi$ as convex indicator of the
interval defined by obstacles. We must mention that the solution of a BSVI
belongs to the domain of the operator $\partial\Psi$ and it is reflected at
the boundary of this.

The standard work on BSVI in the finite dimensional case is that of Pardoux and R\u{a}\c{s}canu \cite{pa-ra/98}, where it is proved the existence and
uniqueness of the solution $ (  Y,Z,K )  $ for BSVI (\ref{GBSVI 1})
with $A\equiv0$, under the following assumptions on $F$: monotonicity with
respect to $y$ (in the sense that $\langle y^{\prime}-y,F(t,y^{\prime
},z)-F(t,y,z)\rangle\leq\alpha|y^{\prime}-y|^{2}$), Lipschitzianity with
respect to $z$ and a sublinear growth for $F (  t,y,0 )  $. Moreover,
it is shown that, unlike the forward case, the process $K$ is absolute
continuous with respect to $\mathrm{d}t$. In Pardoux and R\u{a}\c{s}canu \cite{pa-ra/99}, the same authors extend
these results to the Hilbert spaces framework. Afterwards, various particular cases of BSVI (\ref{GBSVI 1}) were the subject of many articles: Maticiuc and R\u{a}\c{s}canu \cite{ma-ra/10}, Maticiuc, R\u{a}\c{s}canu and Z\u{a}linescu \cite{ma-ra-za/14}, Maticiuc and Rotenstein \cite{ma-ro/12}, Maticiuc and Nie \cite{ma-ni/13} (where the backward equations are studied in the frame of fractional stochastic calculus) and Diomande and Maticiuc \cite{di-ma/14} (where the generator $F$ at the moment $t$ is allowed to depend on the past values on $[0,t]$ of the solution $(Y,Z)$).

Our paper generalizes the existence and uniqueness results from
Pardoux and R\u{a}\c{s}canu \cite{pa-ra/99} by considering random time interval $ [  0,\tau ]  $
and the Lebesgue--Stieltjes integral terms, and by assuming a weaker
boundedness condition for the generator $\Phi$ (instead of the sublinear
growth), that is,
%
\begin{equation}
\label{local bound} \mathbb{E} \biggl(\int_{0}^{T}
\Phi_{\rho}^{\#}(s)\,\mathrm{d}s \biggr)^{p}<\infty,
\qquad \mbox{where } \Phi_{\rho}^{\#} ( t ) :=\sup
_{\llvert  y\rrvert
\leq\rho} \bigl\llvert \Phi(t,y,0) \bigr\rrvert . %
\end{equation}

We mention that, since $\tau$ is a stopping time, the presence of the process
$A$ is justified by the possible applications of equation (\ref{GBSVI 1}) in
proving probabilistic interpretation for the solution of elliptic multivalued
partial differential equations with Neumann boundary conditions on a domain
from $\mathbb{R}^{d}$. The stochastic approach of the existence problem for
finite dimensional multivalued parabolic PDEs, was considered by Maticiuc and R\u{a}\c{s}canu \cite{ma-ra/10}.

Concerning assumption (\ref{local bound}), we recall that, in the case of
finite dimensional BSDE, Pardoux \cite{pa/99} has used a similar
condition, in order to prove the existence of a solution in $L^{2}$. His
result was generalized by Briand, Delyon, Hu, Pardoux and Stoica \cite{br-de-hu-pa/03}, where it is proved the existence in $L^{p}$ of the
solution for BSDEs considered both with fixed and random terminal time. We
mention that the assumptions from our paper are, broadly speaking, similar to
those of Briand, Delyon, Hu, Pardoux and Stoica \cite{br-de-hu-pa/03}.

The article is organized as follows: in the next section a brief summary of
infinite dimensional stochastic integral and the assumptions are given.
Section \ref{sec3} is devoted to the proof of the existence and uniqueness of a strong
solution for (\ref{GBSVI 2}). In the  Section \ref{weak form},  is a new
type of solution (called variational weak solution) and it is also proves the
existence and uniqueness result. In Section \ref{weak form} are obtained, as applications,
the existence of the solution for various type of backward stochastic partial
differential equations with boundary conditions. The \hyperref[app]{Appendix} contains,
following Pardoux and R\u{a}\c{s}canu \cite{pa-ra/13}, some results useful throughout the paper.

\section{Preliminaries}

\subsection{Infinite dimensional framework}

In the beginning of this subsection, we give a brief exposition of the
stochastic integral with respect to a Wiener process defined on a Hilbert
space. For a deeper discussion concerning the notion of cylindrical Wiener
process and the construction of the stochastic integral, we refer reader to
Da Prato and Zabczyk \cite{da-za/92}.

We consider a complete probability space $(\Omega,\mathcal{F},\mathbb{P})$,
the set $\mathcal{N}_{\mathbb{P}}=\{A\in\mathcal{F}\dvtx \mathbb{P} (
A )  =0\}$, a right continuous and complete filtration $ \{
\mathcal{F}_{t} \}  _{t\geq0} $, and two real separable Hilbert spaces
$H,H_{1}$.

Let us denote by $\mathcal{S}_{H}^{p} [  0,T ]  $, $p\geq0$, the
complete metric space of continuous progressively measurable stochastic
process (p.m.s.p.) $X\dvtx \Omega\times [  0,T ]  \rightarrow H$ with the
metric given by
\[
\rho_{p}(X,\tilde{X})=\cases{\displaystyle \Bigl(\mathbb{E}\sup
_{t\in [  0,T ]  }%
\llvert X_{t}-\tilde{X}_{t}
\rrvert ^{p} \Bigr)^{1\wedge1/p}<\infty, &\quad if $p>0$,
\vspace*{2pt}
\cr
\displaystyle \mathbb{E} \Bigl(1\wedge \sup_{t\in [  0,T ]
}
\llvert X_{t}-\tilde{X}_{t}\rrvert \Bigr)< \infty, &\quad if
$p=0$, }
\]
and by $\mathcal{S}_{H}^{p}$ the space of p.m.s.p. $X\dvtx \Omega\times
[0,\infty)\rightarrow H$ such that, for all $T>0$, the restriction
$X\rrvert  _{ [  0,T ]  }\in\mathcal{S}_{H}^{p}[0,T]$. To
shorten notation, we continue to write $\mathcal{S}^{p}$ for $\mathcal{S}%
_{H}^{p}$. Remark that $\mathcal{S}_{H}^{p} [  0,T ]  $ is a Banach
space for $p\geq1$.

By $\mathrm{M}^{p} (  \Omega\times [  0,T ]  ;H )  $,
$p\geq1$, we denote the Banach space of the continuous stochastic processes
$M$ such that $\mathbb{E} (  \llvert  M (  t )  \rrvert
^{p} )  <{\infty}$, $\forall t\in [  0,T ]  $, $M (
0 )  =0$ a.s., and $\mathbb{E}^{\mathcal{F}_{s}} (  M_{t} )
=M_{s}$, a.s. for all $0\leq s\leq t\leq T$. The norm is defined by
$\llVert  M\rrVert  _{\mathrm{M}^{p}}= [  \mathbb{E} (
\llvert  M (  T )  \rrvert  ^{p} )   ]  ^{1/p}$. If
$p>1$, then $\mathrm{M}^{p} (  \Omega\times [  0,T ]  ;H )
$ is a closed linear subspace of $\mathcal{S}_{H}^{p} [  0,T ]  $.

Let $W=\{W_{t}(a)\dvtx t\geq0, a\in H_{1}\}\subset L^{0} (  \Omega
,\mathcal{F},\mathbb{P} )  $ be a Gaussian family of real-valued random
variables with zero mean and the covariance function given by $\mathbb{E}%
[  W_{t}(a)W_{s}(b) ]  = (  t\wedge s )    \langle
a,b \rangle _{H_{1}}$, $s,t\geq0$, $a,b\in H_{1}$. We call $W$ a $H_{1}%
$-Wiener process if, for all $t\geq0$,
\begin{enumerate}[(ii)]
\item[(i)]  $\mathcal{F}_{t}^{W}:=\sigma\{W_{s}(a)\dvtx s\in[ 0,t], a\in H_{1}\}
\vee\mathcal{N}_{\mathbb{P}}\subset\mathcal{F}_{t}$,

\item[(ii)] $W_{t+h}(a)-W_{t}(a)$ is independent of
$\mathcal{F}_{t}$, for all $h>0$, $a\in H_{1}$.
\end{enumerate}

Let $ \{  e_{i} \}  _{i\in\mathbb{N}^{\ast}}$ be an orthonormal and
complete basis in $H_{1}$. We introduce the separable Hilbert space
$L_{2}(H_{1};H)$ of Hilbert--Schmidt operators from $H_{1}$ to $H$, that is, the
space of linear operators $Z\dvtx H_{1}\rightarrow H$ such that $\llvert
Z\rrvert  _{L_{2}(H_{1};H)}^{2}=\sum_{i=1}^{\infty}\llvert
Ze_{i}\rrvert  _{H}^{2}=\Tr (  Z^{\ast}Z )  <\infty$. It
will cause no confusion if we use $\llvert  Z\rrvert  $ to designate the
norm in $L_{2}(H_{1};H)$.

The sequence $W^{i}= \{  W_{t}^{i}:=W_{t} (  e_{i} )\dvtx t\in[0,T]\}, i\in\mathbb{N}^{\ast}$, defines a family of
real-valued Wiener processes mutually independent on $(\Omega,\mathcal{F}%
,\mathbb{P})$.

If $H_{1}$ is finite dimensional space then we have the representation
$W_{t}=\sum_{i}W_{t}^{i}, t\geq0$, but, in general case, this series
does not converge in $H_{1}$, but rather in a larger space $H_{2}$ such that
$H_{1}\subset H_{2}$ with the injection $J\dvtx  H_{1}\rightarrow H_{2}$
being a Hilbert--Schmidt operator. Moreover, $W\in\mathrm{M}^{2} (
\Omega\times [  0,T ]  ;H_{2} )  $.

For $0<T\leq\infty$, we will denote by $\Lambda_{L_{2} (
H_{1},H )  }^{p} (  0,T )  $, $p\geq0$, the space $L_{ad}%
^{p}(\Omega\times (  0,T );\allowbreak L_{2} (  H_{1}, H )  )$, that is,
the complete metric space of progressively measurable stochastic processes
$Z\dvtx \Omega\times (  0,T )  \rightarrow L_{2} (  H_{1},H )  $
with metric of convergence
\[
d_{p}(Z,\tilde{Z})=\cases{ \displaystyle \biggl(\mathbb{E} \biggl(
\int_{0}^{T}\llvert Z_{s}-
\tilde{Z}%
_{s}\rrvert ^{2}\,\mathrm{d}s
\biggr)^{p/2} \biggr)^{1\wedge1/p}< \infty, &\quad if $p>0$,
\vspace*{2pt}
\cr
\displaystyle \mathbb{E} \biggl(1\wedge \biggl(\int
_{0}^{T}\llvert Z_{s}-\tilde
{Z}_{s}\rrvert ^{2}\,\mathrm{d}s \biggr)^{1/2}
\biggr)<\infty,& \quad if $p=0$. }
\]
The space $\Lambda_{L_{2} (  H_{1},H )  }^{p} (  0,T )  $ is
a Banach space for $p\geq1$ with norm $\llVert  Z\rrVert  _{\Lambda^{p}%
}=d_{p}(Z,0)$. From now on, for simplicity of notation, we write $\Lambda
^{p} (  0,T )  $ instead of $\Lambda_{L_{2} (  H_{1},H )
}^{p} (  0,T )  $ (when no confusion can arise).

Let us denote by $\Lambda^{p}$ the space of measurable stochastic processes
$X\dvtx \Omega\times[0,\infty)\rightarrow H$ such that, for all $T>0$, the
restriction $X\rrvert  _{ [  0,T ]  }\in\Lambda^{p} (
0,T )  $.

For any $Z\in\Lambda^{2}$ let the stochastic integral $\mathrm{I}(Z)(t)=\int_{0}^{t}Z_{s}\,\mathrm{d}W_{s}:=\sum_{i=1}^{\infty}\int_{0}^{t}Z_{s}(e_{i})\,\mathrm{d}W_{s}(e_{i})$, $t\in[0,T]$, where $ \{
e_{i} \}_{i}$ is an orthonormal basis in $H_{1}$. Note that the
introduced stochastic integral does not depend on the choice of the orthonormal
basis on $H_{1}$. By the standard localization procedure, we can extend this
integral as a linear continuous operator $\mathrm{I}\dvtx \Lambda^{p} (
0,T )  \rightarrow\mathcal{S}^{p} [  0,T ]  , p\geq0$, and it
has the following properties:

\begin{proposition}
Let $Z\in\Lambda^{p} (  0,T )  $. Then
\begin{enumerate}[(iii)]
\item[(i)]  $\mathbb{E} \mathrm{I} (  Z )   (  t )
=0$, $\forall t\in [  0,T ]$,  if $p\geq1$,
\item[(ii)]   $\mathbb{E}\llvert  \mathrm{I} (  Z )   (
T )  \rrvert  ^{2}=\llVert  Z\rrVert  _{\Lambda^{2}}^{2}$,
if $p\geq2$,
\item[(iii)] $\frac{1}{c_{p}}\llVert  Z\rrVert  _{\Lambda^{p}}^{p}\leq\mathbb{E}\sup_{t\in [  0,T ]
}\llvert  \mathrm{I} (  Z )   (  t )  \rrvert  ^{p}\leq
c_{p}\llVert  Z\rrVert  _{\Lambda^{p}}^{p}$,  if $p>0$
(Burkholder--Davis--Gundy inequality),
\item[(iv)]  $\mathrm{I} (  Z )  \in\mathrm{M}^{p} (
\Omega\times [  0,T ]  ;H )$,  if $p\geq1$.
\end{enumerate}
\end{proposition}

From now on, we shall consider that the original filtration $ \{
\mathcal{F}_{t} \}  _{t\geq0}$ is replaced by the filtration
$\{\mathcal{F}_{t}^{W}\}_{t\geq0}$ generated by the Wiener process. The
following Hilbert space version of the martingale representation theorem,
extended to a random interval, holds the following proposition.

\begin{proposition}\label{mart repr th}
Let $\tau\dvtx \Omega\rightarrow [  0,\infty ]
$ be a stopping time, $p>1$ and $\eta\dvtx \Omega\rightarrow
H$ be a $\mathcal{F}_{\tau}$-measurable random variable
such that $\mathbb{E}\llvert  \eta\rrvert  ^{p}<\infty$. Then
\begin{enumerate}[3.]
\item[1.] there exists a unique stochastic process $\zeta\in
\Lambda^{p} (  0,{\infty} )  $ such that $\eta=\mathbb{E}\eta + \int_{0}^{\tau}\zeta_{s}\,\mathrm{d}W_{s}$ and $\zeta
_{t}=\mathbh{1}_{ [  0,\tau ]  } (  t )  \zeta_{t}$,
$\forall t\geq0$, or equivalently,

\item[2.] there exists a unique pair $ (  \xi,\zeta )
\in\mathcal{S}^{p}\times\Lambda^{p} (  0,{\infty} )  $ such that
%
\begin{equation}
\label{mart repr th 1} \xi_{t}=\eta-\int_{t\wedge\tau}^{\tau}
\zeta_{s}\,\mathrm{d}W_{s},\qquad\mbox{a.s., } t\geq0,
\end{equation}
or equivalently,

\item[3.] there exists a unique pair $ (  \xi,\zeta )
\in\mathcal{S}^{p}\times\Lambda^{p} (  0,{\infty} )  $
such that $\xi_{t}=\eta-{\int_{t}^{\infty}}\zeta_{s}\,\mathrm{d}W_{s}$,
a.s., $t\geq0$ and $\xi_{t}=\mathbb{E}^{\mathcal{F}%
_{t}}\eta=\mathbb{E}^{\mathcal{F}_{t\wedge\tau}}\eta$ and $\zeta
_{t}=\mathbh{1}_{ [  0,\tau ]  } (  t )  \zeta_{t} $,
$t\geq0$.
\end{enumerate}
\end{proposition}

\subsection{Assumptions and definitions}\label{sec2.2}

In order to study equation (\ref{GBSVI 1}), or the equivalent form
(\ref{GBSVI 2}), we introduce the next assumptions:

\begin{enumerate}[(A$_{5}$)]
\item[(A$_{1}$)] \textit{The parameter} $p\geq2$;

\item[(A$_{2}$)] \textit{The random variable} $\tau
\dvtx \Omega\rightarrow [  0,\infty ] $ \textit{is a stopping time};

\item[(A$_{3}$)] \textit{The random variable} $\eta
\dvtx \Omega\rightarrow H$ \textit{is} $\mathcal{F}_{\tau}$\textit{-measurable
such that} $\mathbb{E}\llvert  \eta\rrvert  ^{p}<\infty$ \textit{and the
stochastic process} $ (  \xi,\zeta )  \in\mathcal{S}^{p}\times
\Lambda^{p} (  0,{\infty} )  $ \textit{is the unique pair
associated to} $\eta$ \textit{such that we have the martingale representation
formula \emph{(\ref{mart repr th 1})}};

\item[(A$_{4}$)] \textit{The process} $ \{  A_{t}\dvtx
t\geq0 \}  $ \textit{is a progressively measurable increasing
continuous stochastic process such that} $A_{0}=0$;

\item[(A$_{5}$)] \textit{The functions} $F\dvtx \Omega
\times[0,\infty)\times H\times L_{2} (  H_{1},H )  \rightarrow
H$ \textit{and} $G\dvtx \Omega\times[0,\infty)\times H\rightarrow
H$ \textit{are such that}
\[
\cases{\displaystyle F ( \cdot,\cdot,y,z ) , G ( \cdot,\cdot,y ) \mbox{\textit{
are p.m.s.p.\textup{,} for all}} ( y,z ) \in H\times L_{2} ( H_{1},H),
\cr
\displaystyle F ( \omega,t,\cdot,\cdot ) , G ( \omega,t,\cdot )
\mbox{ \textit{are continuous functions a.e.}}, }
\]
\textit{and} $\int_{0}^{T}F_{\rho}^{\#} (  s )\,\mathrm{d}s+\int_{0}%
^{T}G_{\rho}^{\#} (  s )\,\mathrm{d}A_{s}<\infty,$ $\forall \rho$, $T\geq0$,
$\mathbb{P}$-\textit{a.s.}, \textit{where }$F_{\rho}^{\#} (  s )
=\sup_{\llvert  y\rrvert  \leq\rho}\llvert  F (
s,\allowbreak y,0 )  \rrvert  $ \textit{and} $G_{\rho}^{\#} (  s )  =\sup_{\llvert  y\rrvert  \leq\rho}\llvert  G (  s,y )
\rrvert  $.

\textit{Moreover, there exist two p.m.s.p.} $\mu,\nu\dvtx \Omega\times
[0,\infty)\rightarrow\mathbb{R}$ \textit{such that} $\int_{0}^{T}%
\llvert \mu_{t}\rrvert ^{2}\,\mathrm{d}t<\infty$ \textit{and} $\int_{0}^{T}\llvert \nu_{t}\rrvert ^{2}\,\mathrm{d}A_{t}<\infty,$
\textit{for all} $T>0, \mathbb{P}$\textit{-a.s.}, \textit{and there exists}
$\ell\geq0,$\textit{ such that}, \textit{for all} $y,y^{\prime}\in H$, $z,z^{\prime
}\in L_{2} (  H_{1},H ),$
%
\begin{eqnarray}
\label{F, G assumpt 2} \bigl\langle y^{\prime}-y,F \bigl(t,y^{\prime},z
\bigr)-F(t,y,z) \bigr\rangle &\leq&\mathbh{1}_{ [  0,\tau ]  } ( t )
\mu_{t} \bigl\llvert y^{\prime}-y \bigr\rrvert ^{2},
\nonumber
\\
\bigl\langle y^{\prime}-y,G \bigl(t,y^{\prime} \bigr)-G(t,y) \bigr
\rangle &\leq& \mathbh{1}_{ [  0,\tau ]  } ( t ) \nu_{t} \bigl\llvert
y^{\prime}-y \bigr\rrvert ^{2},
\\
\bigl\llvert F \bigl(t,y,z^{\prime} \bigr)-F(t,y,z) \bigr\rrvert &\leq&
\mathbh{1}_{ [
0,\tau ]  } ( t ) \ell \bigl\llvert z^{\prime}-z \bigr\rrvert
.
\nonumber
\end{eqnarray}
\end{enumerate}

Let us introduce the function
\[
Q_{t} ( \omega ) :=t+A_{t} ( \omega )
\]
and let $ \{  \alpha_{t}\dvtx t\geq0 \}  $ be the a real positive
p.m.s.p. (given by Radon--Nikodym's representation theorem) such that
$\alpha\in [  0,1 ]  $ and $\mathrm{d}t=\alpha_{t}\,\mathrm{d}Q_{t}$ and $\mathrm{d}A_{t}= (
1-\alpha_{t} )\,\mathrm{d}Q_{t}$.

Let
\[
\Phi ( \omega,t,y,z ) :=\mathbh{1}_{ [  0,\tau (
\omega )   ]  } ( t ) \bigl[
\alpha_{t} ( \omega ) F ( \omega,t,y,z ) + \bigl( 1-
\alpha_{t} ( \omega ) \bigr) G ( \omega,t,y ) \bigr] ,
\]
in which case (\ref{F, G assumpt 2}) yields
\begin{eqnarray*}
\bigl\langle y^{\prime}-y,\Phi \bigl(t,y^{\prime},z \bigr)-\Phi(t,y,z)
\bigr\rangle &\leq&\mathbh{1}_{ [  0,\tau ]  } ( t ) \bigl[ \mu _{t}
\alpha_{t}+\nu_{t} ( 1-\alpha_{t} ) \bigr] \bigl
\llvert y^{\prime}-y \bigr\rrvert ^{2},
\\
\bigl\llvert \Phi \bigl(t,y,z^{\prime} \bigr)-\Phi(t,y,z) \bigr\rrvert &
\leq& \mathbh{1}_{ [
0,\tau ]  } ( t ) \ell\alpha_{t} \bigl\llvert
z^{\prime
}-z \bigr\rrvert .
\end{eqnarray*}
For $a>1$, let
%
\begin{eqnarray}
\label{def V} V_{t}&=&{\int_{0}^{t}}
\mathbh{1}_{ [  0,\tau ]  } ( s ) \biggl[ \biggl(\mu_{s}+
\frac{a}{2} \ell^{2} \biggr)\alpha_{s}+
\nu_{s} ( 1-\alpha_{s} ) \biggr]\,\mathrm{d}Q_{s}
\nonumber
\\[-8pt]
\\[-8pt]
&=&{\int_{0}^{t}}\mathbh{1}_{ [
0,\tau ]  } ( s
) \biggl[ \biggl(\mu_{s}+\frac{a}{2} \ell ^{2} \biggr)
\,\mathrm{d}s+\nu_{s}\,\mathrm{d}A_{s} \biggr] .
\nonumber
\end{eqnarray}
We can give now some a priori estimates concerning the solution of
(\ref{GBSVI 1}).

\begin{lemma}
Let $ (  Y,Z )  , (\tilde{Y},\tilde{Z})\in\mathcal{S}^{0} [
0,T ]  \times\Lambda^{0} (  0,T )  $. Under assumption
\emph{(A$_{5}$)} the following inequalities hold, in the sense of signed
measures on $[0,\infty)$,
%
\begin{equation}
\label{a priori estim 1} \bigl\langle Y_{s},\Phi ( s,Y_{s},Z_{s}
)\,\mathrm{d}Q_{s} \bigr\rangle \leq\llvert Y_{s}\rrvert
\bigl\llvert \Phi ( s,0,0 ) \bigr\rrvert\,\mathrm{d}Q_{s}+\llvert
Y_{s}%
\rrvert ^{2}\,\mathrm{d}V_{s}+
\frac{1}{2a}\llvert Z_{s}\rrvert ^{2}\,\mathrm{d}s
\end{equation}
and
%
\begin{equation}
\label{a priori estim 2} \bigl\langle Y_{s}-\tilde{Y}_{s},\Phi (
s,Y_{s},Z_{s} ) -\Phi (s,\tilde{Y}_{s},
\tilde{Z}_{s}) \bigr\rangle\,\mathrm{d}Q_{s}\leq\llvert
Y_{s}- \tilde{Y}_{s}%
\rrvert ^{2}\,
\mathrm{d}V_{s}+ \frac{1}{2a}\llvert Z_{s}-
\tilde{Z}_{s}\rrvert ^{2}\,\mathrm{d}s. %
\end{equation}
\end{lemma}

\begin{pf}
The inequalities can be obtained by standard calculus (applying the
monotonicity and Lipschitz property of function $\Phi$).\hfill
\end{pf}

\begin{enumerate}[(A$_{6}$)]
\item[(A$_{6}$)] $\varphi,\psi\dvtx H\rightarrow [
0,+\infty ]  $ \textit{are proper convex lower semicontinuous
$($l.s.c.$)$
functions such that} $\varphi (  0 )  =\psi (  0 )
=0$ (\textit{consequently} $0\in\partial\varphi (0)  \cap
\partial\psi (0)$).
\end{enumerate}

Let us define
\[
\Psi ( \omega,t,y ) :=\mathbh{1}_{ [  0,\tau (
\omega )   ]  } ( t ) \bigl[
\alpha_{t} ( \omega ) \varphi ( y ) + \bigl( 1-\alpha_{t} (
\omega ) \bigr) \psi ( y ) \bigr] .
\]
We recall now that the multivalued subdifferential operator $\partial\varphi$
is the maximal monotone operator
\[
\partial\varphi ( y ) := \bigl\{ \hat{y}\in H\dvtx \langle \hat {y},v-y \rangle
+\varphi ( y ) \leq\varphi ( v ) , \forall v\in H \bigr\} .
\]
We define $\Dom  (  \varphi )  = \{  y\in H\dvtx \varphi (
y )  <\infty \}  $ and $\Dom  (  \partial\varphi )
= \{  y\in H\dvtx \partial\varphi (  y )  \neq\varnothing \}
\subset\Dom  (  \varphi )  $ and by $ (  y,\hat{y} )
\in\partial\varphi$ we understand that $y\in\Dom  (  \partial
\varphi )  $ and $\hat{y}\in\partial\varphi (  y )  $. We know
that $\intt (  \Dom  (  \varphi )   )
=\intt (  \Dom  (  \partial\varphi )   )  $
and $\overline{\Dom  (  \varphi )  }=\overline{\Dom %
(  \partial\varphi )  } $.

\begin{definition}
If $k\dvtx [0,\infty)\rightarrow H$ is a locally bounded variation function,
$a\dvtx [0,\infty)\rightarrow\mathbb{R}$ is a real increasing function,
$y\dvtx [0,\infty)\rightarrow H$ is a continuous function and $\varphi$ is like in
\emph{(A$_{6}$)}, then notation $\mathrm{d}k_{t}\in\partial\varphi (
y_{t} )\,\mathrm{d}a_{t}$ means that for any continuous function $x\dvtx [0,\infty
)\rightarrow H$, it holds
%
\begin{equation}
\label{in the subdiff} {\int_{t}^{s}} \langle
x_{r}-y_{r},\mathrm{d}k_{r} \rangle +{\int
_{t}^{s}%
}\varphi ( y_{r} )
\,\mathrm{d}a_{r}\leq{\int_{t}^{s}}
\varphi ( x_{r} )\,\mathrm{d}a_{r} ,\qquad 0\leq t\leq s
.%
\end{equation}
\end{definition}

Now we are able to introduce the rigorous definition of a solution for
equation (\ref{GBSVI 1}). First, using definitions of $Q$, $\Phi$ and $\Psi
$, respectively, we can rewrite (\ref{GBSVI 1}) in the form
%
\begin{equation}
\label{GBSVI 4} \cases{\displaystyle Y_{t}+{\int_{t}^{\infty}}
\,\mathrm{d}K_{s}= \eta+{\int_{t}^{\infty}}
\Phi ( s,Y_{s},Z_{s} )\,\mathrm{d}Q_{s}-{\int
_{t}^{\infty}}Z_{s}\,\mathrm{d}W_{s},
\qquad \mbox{a.s., }t\geq0,\vspace*{2pt}
\cr
\displaystyle \mathrm{d}K_{t}\in
\partial_{y}\Psi ( t,Y_{t} )\,\mathrm{d}Q_{t}=
\partial\varphi ( Y_{t} )\,\mathrm{d}t+ \partial\psi ( Y_{t}
)\,\mathrm{d}A_{t}, \qquad\mbox{on }[0,\infty). }%
\end{equation}

\begin{definition}\label{definition sol GBSVI 1}
We call $ (  Y_{t},Z_{t},K_{t} )
_{t\geq0}$ a solution of (\ref{GBSVI 4}) if $K$ has locally bounded variation
and $ (  Y,Z )  \in S^{0}\times\Lambda^{0}$ with $ (  Y_{t}%
,Z_{t} )  = (  \xi_{t},\zeta_{t} )  = (  \eta,0 )  $
for $t>\tau$ such that
\begin{enumerate}[(iii)]
\item[(i)]    ${\int_{0}^{T}}\llvert  \Phi (
s,Y_{s},Z_{s} )  \rrvert\,\mathrm{d}Q_{s}<\infty$, $\mathbb{P}$-a.s., for
all $T\geq0$,
\item[(ii)]    $\mathrm{d}K_{t}\in\partial_{y}\Psi (  t,Y_{t} )\,\mathrm{d}Q_{t} , \mathrm{d}\mathbb{P}\otimes \mathrm{d}Q_{t}$-a.e.,
\item[(iii)]    $\mathrm{e}^{2V_{T}}\llvert  Y_{T}-\xi_{T}\rrvert
^{2}+{\int_{T}^{\infty}}\mathrm{e}^{2V_{s}}\llvert  Z_{s}-\zeta
_{s}\rrvert  ^{2}\,\mathrm{d}s  \xrightarrow{}{prob.}0$, as
$T\rightarrow\infty$
(where $V$ is given by (\ref{def V})) and
\item[(iv)]   $Y_{t}+{\int_{t}^{T}}\,\mathrm{d}K_{s}=Y_{T}+{\int_{t}^{T}}\Phi (  s,Y_{s},Z_{s} )\,\mathrm{d}Q_{s}-{\int_{t}^{T}}Z_{s}\,\mathrm{d}W_{s}$, a.s., $\forall0\leq t\leq T$.
\end{enumerate}
\end{definition}

Let $\varepsilon>0$ and the Moreau--Yosida regularization of $\varphi$ given
by $\varphi_{\varepsilon} (  y )  =\inf\{\frac{1}{2\varepsilon}\llvert  y-v\rrvert  ^{2}+\varphi (  v )  \dvtx v\in
H\} $, which is a $C^{1}$ convex function. We mention some properties
(see Br\'{e}zis \cite{br/73}, and Pardoux and R\u{a}\c{s}canu \cite{pa-ra/98} for the last one): for all $x,y\in H$
%
\begin{equation}
\label{ineq Yosida} %
\begin{array} [c]{r@{\quad}l}%
\mbox{(a)} &
\displaystyle\varphi_{\varepsilon} ( x ) =\frac{\varepsilon}{2} \bigl\llvert \nabla
\varphi_{\varepsilon}(x) \bigr\rrvert ^{2}+\varphi \bigl( x-\varepsilon
\nabla\varphi_{\varepsilon}(x) \bigr) ,
\\\noalign{\vspace*{4pt}}
\mbox{(b)} & \displaystyle\nabla\varphi_{\varepsilon}(x)=\partial
\varphi_{\varepsilon} ( x ) \in\partial\varphi \bigl( x-\varepsilon\nabla
\varphi_{\varepsilon
}(x) \bigr) ,
\\\noalign{\vspace*{4pt}}
\mbox{(c)} & \displaystyle \bigl\llvert \nabla\varphi_{\varepsilon}(x)-\nabla
\varphi_{\varepsilon
}(y) \bigr\rrvert \leq\frac{1}{\varepsilon}\llvert x-y\rrvert ,
\\\noalign{\vspace*{4pt}}
\mbox{(d)} & \displaystyle \bigl\langle \nabla\varphi_{\varepsilon}(x)-\nabla
\varphi_{\varepsilon
}(y),x-y \bigr\rangle \geq0,
\\\noalign{\vspace*{4pt}}
\mbox{(e)} & \displaystyle \bigl\langle \nabla\varphi_{\varepsilon}(x)-\nabla
\varphi_{\delta}(y),x-y \bigr\rangle \geq-(\varepsilon+\delta) \bigl\langle
\nabla\varphi_{\varepsilon}(x),\nabla\varphi_{\delta}(y) \bigr\rangle .
\end{array} %
\end{equation}
We introduce the \textit{compatibility conditions} between $\varphi,\psi$
(which have previously been used in Maticiuc and R\u{a}\c{s}canu \cite{ma-ra/10}):

\begin{enumerate}[(A$_{7}$)]
\item[(A$_{7}$)] \textit{For all }$\varepsilon>0$, $t\geq0$,
$y\in H$, $z\in L_{2} (  H_{1},H )  $
%
\begin{equation}
\label{compatib. assumption} %
\begin{array}[c]{@{\hspace*{-3pt}}r@{
\quad}l}%
\mbox{(i)} & \bigl\langle \nabla\varphi_{\varepsilon} ( y ) ,
\nabla \psi_{\varepsilon} ( y ) \bigr\rangle \geq0,
\\\noalign{\vspace*{4pt}}
\mbox{(ii)} & \bigl\langle \nabla\varphi_{\varepsilon} ( y ) ,G ( t,y ) +
\nu_{t}^{-}y \bigr\rangle \leq \bigl\llvert \nabla
\psi_{\varepsilon} ( y ) \bigr\rrvert \bigl\llvert G ( t,y ) +
\nu_{t}^{-}y \bigr\rrvert , \qquad\mathbb{P}\mbox{\textit{-a.s.},%
}
\\\noalign{\vspace*{4pt}}
\mbox{(iii)} & \bigl\langle \nabla\psi_{\varepsilon} ( y ) ,F ( t,y,z ) +
\mu_{t}^{-}y \bigr\rangle \leq \bigl\llvert \nabla
\varphi_{\varepsilon} ( y ) \bigr\rrvert \bigl\llvert F ( t,y,z ) +
\mu_{t}^{-}y \bigr\rrvert ,\qquad \mathbb{P}\mbox{\textit{-a.s.}},%
\end{array} %
\end{equation}
\textit{where }$\mu^{-}=-\min \{  \mu,0 \}  $\textit{ and }$\nu
^{-}=-\min \{  \nu,0 \}$\textit{.}
\end{enumerate}

\begin{example}
Let $H=\mathbb{R}$.

\begin{enumerate}[B.]
\item[A.] Clearly, since $\nabla\varphi_{\varepsilon}$ and
$\nabla\psi_{\varepsilon}$ are increasing monotone, we see that, if
$y (G(t,y)+\nu_{t}^{-}y)\leq0$ and $y (F(t,y,z)+\mu_{t}^{-}y)\leq0$, $\forall t,y,z$, then
compatibility assumptions (\ref{compatib. assumption}) are satisfied.

\item[B.] If $\varphi,\psi\dvtx \mathbb{R}\rightarrow(-\infty,+\infty]$
are the convexity indicator functions, that is,
\[
\varphi ( y ) =\cases{\displaystyle 0, &\quad if $y\in [ a_{1},a_{2}
]$,
\cr
\displaystyle+\infty , &\quad if $y\notin [ a_{1},a_{2}
]$, } \quad \mbox{and}\quad \psi ( y ) =\cases{\displaystyle 0, & \quad if $y\in
[ b_{1},b_{2} ]$,
\cr
\displaystyle+\infty , & \quad if $y
\notin [ b_{1},b_{2} ]$, }
\]
where $a_{1},a_{2},b_{1},b_{2}\in\mathbb{R}$ are such that
$0\in [  a_{1},a_{2} ]  \cap [  b_{1},b_{2} ]  $ (see
assumption (A$_{6}$)), then $\nabla\varphi_{\varepsilon
} (  y )  =\frac{1}{\varepsilon} [ (  y-a_{2} )
^{+}- (  a_{1}-y )  ^{+} ]$ and similar for $\nabla\psi_{\varepsilon}$.

Since (A$_{7}$)(i) is fulfilled, the compatibility assumptions become $G (  t,y )  +\nu_{t}^{-}y\geq0$, for $y\leq a_{1}$ and $G (  t,y )  +\nu_{t}^{-}y\leq0$, for $y\geq a_{2}$, and, respectively, $F (  t,y,z )
+\mu_{t}^{-}y\geq0,\mathit{ }$ for $y\leq b_{1}$ and $F (
t,y,z )  +\mu_{t}^{-}y\leq0$, for $y\geq b_{2}$.
\end{enumerate}
\end{example}

The last assumption is the following:

\begin{enumerate}[(A$_{8}$)]
\item[(A$_{8}$)] \textit{There exist the p.m.s.p.}
$\tilde{\mu},\tilde{\nu}\dvtx \Omega\times[0,\infty)\rightarrow\mathbb{R}%
$ \textit{with} $\tilde{\mu}\geq\max \{  \mu,\frac{1}{2}\mu \}
$ \textit{and} $\tilde{\nu}\geq\max \{  \nu,\frac{1}{2}\nu \},
$\textit{ such that} $\int_{0}^{T} (  \llvert \tilde{\mu}_{t}\rrvert ^{2}\,\mathrm{d}t+\llvert \tilde
{\nu}_{t}\rrvert ^{2}\,\mathrm{d}A_{t} )  <\infty, \forall T>0, \mathbb{P}$\textit{-a.s.
and\textup{,} using notation}
%
\begin{equation}
\label{def V tilde} \tilde{V}_{t}:=\int_{0}^{t}
\mathbh{1}_{ [  0,\tau ]  } ( s ) \biggl[ \biggl(\tilde{\mu}_{s}+
\frac{a}{2} \ell^{2} \biggr)\,\mathrm{d}s+\tilde{\nu
}_{s}\,\mathrm{d}A_{s} \biggr], %
\end{equation}
\textit{we suppose}
%
\begin{equation}
\label{Assumption A_8} %
\begin{array}[c]{r@{\quad}l}%
\mbox{(i)} &
\displaystyle\mathbb{E} \bigl[\mathrm{e}^{2\sup_{s\in [  0,\tau ]
}\tilde{V}_{s}} \bigl( \varphi(\eta)+\psi(
\eta) \bigr) \bigr]<\infty ,
\\\noalign{\vspace*{6pt}}
\mbox{(ii)} & \displaystyle\mathbb{E} \bigl(\mathrm{e}^{p\sup_{s\in [  0,\tau ]
}\tilde{V}_{s}} \llvert
\eta\rrvert ^{p} \bigr)+\mathbb{E} \bigl( Q_{T}^{p}
\bigr) < \infty, \qquad\forall T>0,
\\\noalign{\vspace*{6pt}}
\mbox{(iii)} &\displaystyle \mathbb{E} \biggl( {\int_{0}^{\tau}%
}\mathrm{e}^{2\tilde{V}_{s}} \Psi ( s,\xi_{s} )\,\mathrm{d}Q_{s}
\biggr) ^{p/2}+\mathbb{E} \biggl( {\int_{0}^{\tau}}
\mathrm{e}^{\tilde{V}_{s}%
} \bigl\llvert \Phi ( s,\xi_{s},
\zeta_{s} ) \bigr\rrvert\,\mathrm{d}Q_{s} \biggr) ^{p}<
\infty, \end{array} %
\end{equation}
\textit{and the locally boundedness conditions}:
%
\begin{equation}
\label{Assumption A_9} %
\begin{array}[c]{r@{\quad}l}%
\mbox{(iv)} &
\displaystyle \mathbb{E} \biggl({\int_{0}^{T}}
\mathrm{e}^{\tilde
{V}_{s}} \sup_{\llvert  y\rrvert  \leq\rho} \bigl\llvert F \bigl(s,
\mathrm{e}^{-\tilde{V}_{s}%
}y,0 \bigr)-\tilde{ \mu}_{s}y \bigr\rrvert \,
\mathrm{d}s \biggr)^{p}
\\\noalign{\vspace*{6pt}}
&\displaystyle\quad{} +\mathbb{E} \biggl({\int_{0}^{T}}
\mathrm{e}^{\tilde{V}_{s}}%
\sup_{\llvert  y\rrvert  \leq\rho} \bigl\llvert G
\bigl(s,\mathrm{e}^{-\tilde{V}_{s}%
}y \bigr)-\tilde{ \nu}_{s}y \bigr
\rrvert \,\mathrm{d}A_{s} \biggr)^{p}<\infty,\qquad \forall
T,\rho >0,
\\\noalign{\vspace*{6pt}}
\mbox{(v)} &\displaystyle \mathbb{E} {\int_{0}^{\tau}}
\mathrm{e}^{2\tilde{V}%
_{s}} \sup_{\llvert  y\rrvert  \leq\rho} \bigl\llvert F \bigl(s,
\mathrm{e}^{-\tilde{V}_{s}%
}y,0 \bigr) \bigr\rrvert ^{2}\,\mathrm{d}s
\\\noalign{\vspace*{6pt}}
&\displaystyle \quad{}+\mathbb{E} {\int_{0}^{\tau}}
\mathrm{e}^{2\tilde{V}_{s}}%
\sup_{\llvert  y\rrvert  \leq\rho} \bigl\llvert G
\bigl(s,\mathrm{e}^{-\tilde{V}_{s}%
}y \bigr) \bigr\rrvert ^{2}\,
\mathrm{d}A_{s}<\infty,\qquad \forall \rho>0. \end{array} %
\end{equation}
\end{enumerate}

\begin{remark}
We point out that the purpose of defining of the new process $\tilde{V}$ is
due to the computations; see, e.g., inequalities (\ref{approxim eq 3}) and
(\ref{boundedness for approxim 7''}) from the proof of the first main theorem,
where it is necessary to have a new process $\tilde{V}$ such that $\mathrm{d}V_{t}\leq
\mathrm{d}\tilde{V}_{t}$ and $\frac{1}{2}\,\mathrm{d}V_{t}\leq \mathrm{d}\tilde{V}_{t}$ on $[0,\infty)$.
\end{remark}

\begin{remark}
It can be choose in {(A$_{8}$)}, in particular, $\tilde{\mu}$
and $\tilde{\nu}$ such that $\tilde{\mu}=\mu^{+}=\max \{  \mu,0 \}  $
and $\tilde{\nu}=\nu^{+}=\max \{  \nu,0 \}  $. In this case
$\tilde{V}$ defined by (\ref{def V tilde}) will become non-decreasing, hence
$\sup_{s\in [  0,\tau ]  }\tilde{V}_{s}=\tilde{V}_{\tau}$ and
(\ref{Assumption A_8}) and  (\ref{Assumption A_9}) will be simplified.

We prefer to keep inequalities $\tilde{\mu}\geq\max \{  \mu,\frac{1}{2}\mu \}$ and $\tilde{\nu}\geq\max \{\nu,\frac{1}{2}
\nu \}$ in this form because we allow to $\tilde{\mu}$ and $\tilde{\nu
}$ to be negative and therefore to enlarge the class of the generators $F$ and
$G$ who satisfy (\ref{Assumption A_8}) and (\ref{Assumption A_9}) (and also we not
restrict the class of the final data~$\eta$).
\end{remark}

\section{Main result: The existence of the strong solution}\label{sec3}

We present first the definition of a solution in the strong case when the
process $K$ is absolutely continuous with respect to $\mathrm{d}Q$ (i.e., $\mathrm{d}K_{t}%
=U_{t}\,\mathrm{d}Q_{t}$ on $[0,\infty)$).

\begin{definition}\label{definition sol GBSVI}
We call $ (  Y_{t},Z_{t},U_{t} )
_{t\geq0}$ a strong solution of (\ref{GBSVI 4}) if there exist two p.m.s.p.
$U^{1}$, $U^{2}$ and $U_{t}:=\mathbh{1}_{ [  0,\tau ]  } (
t )   [  \alpha_{t}U_{t}^{1}+ (  1-\alpha_{t} )  U_{t}%
^{2} ]  $, such that $ (  Y,Z,K )  $ is a solution of
(\ref{GBSVI 4}) with $K_{t}=\int_{0}^{t}U_{s}\,\mathrm{d}Q_{s}$ and
%
\begin{equation}
\label{definition sol 1} %
\begin{array}[c]{r@{\quad}l}%
\mbox{
\emph{(i)}} & \displaystyle{\int_{0}^{T}}\llvert
U_{s} \rrvert\,\mathrm{d}Q_{s}<\infty, \qquad\mathbb{P}\mbox{-a.s., for all }T \geq0,
\\\noalign{\vspace*{6pt}}
\mbox{\emph{(ii)}} &\displaystyle U_{t}^{1}\in\partial
\varphi ( Y_{t} ) ,\qquad \mathrm{d}\mathbb{P}\otimes \mathrm{d}t
\mbox{-a.e., } U_{t}^{2}\in\partial\psi (
Y_{t} ) ,\ \mathrm{d}\mathbb{P}\otimes \mathrm{d}A_{t}
\mbox{-a.e.},
\\\noalign{\vspace*{6pt}}
\mbox{\emph{(iii)}} &\displaystyle \mathbb{E} \bigl( \mathrm{e}^{2V_{T}}
\llvert Y_{T}-\xi _{T}\rrvert ^{2} \bigr) +{
\mathbb{E}\int_{T}^{\infty}%
}
\mathrm{e}^{2V_{s}}\llvert Z_{s}-\zeta_{s}\rrvert
^{2}\,\mathrm{d}s\rightarrow 0,\qquad \mbox{as }T\rightarrow\infty,
\end{array} %
\end{equation}
where $V$ is given by (\ref{def V}).
\end{definition}

\begin{remark}
If there exists $C>0$ such that $\sup_{s\in [  0,\tau ]  }\llvert
V_{s}\rrvert  \leq C$, $\mathbb{P}$-a.s., then the condition
(\ref{definition sol 1})(iii) is equivalent to $\mathbb{E}\llvert
Y_{T}-\eta\rrvert  ^{2}+\mathbb{E}{\int_{T}^{\infty}%
}\llvert  Z_{s}\rrvert  ^{2}\,\mathrm{d}s\rightarrow0$, as $T\rightarrow\infty$.
\end{remark}

We can now formulate the first main result. In order to obtain the absolute
continuity with respect to $\mathrm{d}Q_{t}$ of the process $K$ (as in Definition
\ref{definition sol GBSVI}) it is necessary to impose a supplementary assumption:

\begin{enumerate}[(A$_{9}$)]
\item[(A$_{9}$)] \textit{There exists} $R_{0}>0$ \textit{such that, for all}
$t\geq0$,
%
\begin{equation}
\label{Assumption A_10} \mathbb{E}^{\mathcal{F}_{t}} \bigl(\mathrm{e}^{2\sup_{s\geq t}\tilde{V}_{s}%
}\llvert
\eta\rrvert ^{2} \bigr)+\mathbb{E}^{\mathcal{F}_{t}}%
\biggl({
\int_{t\wedge\tau}^{\tau}}\mathrm{e}^{\tilde{V}_{s}} \bigl
\llvert \Phi ( s,0,0 ) \bigr\rrvert\,\mathrm{d}Q_{s} \biggr)^{2}\leq
R_{0}, \qquad\mbox{a.s.} %
\end{equation}
\end{enumerate}

\begin{remark}\label{remark_(A9)}
We mention that without this assumption we are not able to
prove, among other, that there exist two processes $U^{1}$ and $U^{2}$ such
that $K_{t}=\int_{0}^{t}\mathbh{1}_{[0,\tau]}(t)[U_{t}^{1}\,\mathrm{d}t+U_{t}^{2}\,\mathrm{d}A_{t}]$ (see step $\mathrm{F}$ from the
proof of the next theorem).
\end{remark}

\begin{theorem}\label{main result 1}
Let assumptions \emph{(A$_{1}$)--(A$_{9}$)} be satisfied.
Then the backward stochastic variational inequality (\ref{GBSVI 4}) has a
unique solution $ (  Y,Z,U )  $ such that for all $T\geq0$,
%
\begin{equation}
\label{ineq of sol 1} \mathbb{E}\sup_{s\in [  0,T ]  }\mathrm{e}^{p\tilde{V}_{s}}
\llvert Y_{s}\rrvert ^{p}<\infty %
\end{equation}
and\vspace*{-2pt}
%
\begin{equation}
\label{ineq of sol 2} Y_{t}+{\int_{t}^{T}}U_{s}
\,\mathrm{d}Q_{s}=Y_{T}+{ \int_{t}^{T}%
}\Phi ( s,Y_{s},Z_{s} )\,\mathrm{d}Q_{s}-{\int
_{t}^{T}}%
Z_{s}\,
\mathrm{d}W_{s}, \qquad \mbox{a.s., }\forall t\in [ 0,T ] .%
\end{equation}
Moreover, for all $2\leq q\leq p$, there exists a constant $C=C (
a,q )  >0$ such that, for all $t\geq0$, $\mathbb{P}$-a.s.
%
\begin{equation}
\label{ineq of sol 3} %
\begin{array}[c]{r@{\quad}l}%
\mbox{
\emph{(a)}} &\displaystyle \mathrm{e}^{q\tilde{V}_{t}}\llvert Y_{t}
\rrvert ^{q}+ \mathbb{E}^{\mathcal{F}%
_{t}} \biggl({\int
_{t}^{\infty}}\mathrm{e}^{2\tilde{V}_{s}}\llvert
Z_{s}%
\rrvert ^{2}\,\mathrm{d}s
\biggr)^{q/2}
\\\noalign{\vspace*{5pt}}
& \displaystyle\quad\leq C \mathbb{E}^{\mathcal{F}_{t}} \biggl[
\mathrm{e}^{q\sup_{s\geq t}\tilde{V}_{s}%
}\llvert \eta\rrvert ^{q}+ \biggl({\int
_{t}^{{\infty}}}\mathrm{e}^{\tilde{V}_{s}%
}\bigl\llvert
\Phi ( s,0,0 ) \bigr\rrvert \,\mathrm{d}Q_{s} \biggr)^{q}
\biggr],
\\\noalign{\vspace*{5pt}}
\mbox{\emph{(b)}} & \displaystyle \mathrm{e}^{q\tilde{V}_{t}}\llvert
Y_{t}- \xi_{t}\rrvert ^{q}+\mathbb{E}^{\mathcal{F}_{t}}
\biggl( {\int_{t}^{\infty}%
}
\mathrm{e}^{2\tilde{V}_{s}}\llvert Z_{s}-\zeta_{s}\rrvert
^{2}\,\mathrm{d}s \biggr) ^{q/2}
\\\noalign{\vspace*{5pt}}
&\displaystyle\quad \leq C \mathbb{E}^{\mathcal{F}_{t}} \biggl[ \biggl({\int
_{t}%
^{\infty}}\mathrm{e}^{2\tilde{V}_{s}}\Psi
( s, \xi_{s} )\,\mathrm{d}Q_{s}%
\biggr)^{q/2}+ \biggl({ \int_{t}^{\infty}}
\mathrm{e}^{\tilde{V}_{s}} \bigl\llvert \Phi ( s, \xi_{s},
\zeta_{s} ) \bigr\rrvert\,\mathrm{d}Q_{s} \biggr)^{q}%
\biggr],
\\\noalign{\vspace*{5pt}}
\mbox{\emph{(c)}} &\displaystyle \mathbb{E} \bigl[\mathrm{e}^{2\tilde{V}_{t}}
\bigl( \varphi (Y_{t})+\psi(Y_{t}) \bigr) \bigr]\leq
\mathbb{E} \bigl[\mathrm{e}^{2\tilde{V}%
_{{\infty}}} \bigl( \varphi(\eta)+\psi(\eta) \bigr)
\bigr],
\\\noalign{\vspace*{5pt}}
\mbox{\emph{(d)}} &\displaystyle \lim_{T\rightarrow\infty}
\mathbb{E}%
\biggl[\mathrm{e}^{p\tilde{V}_{T}}\llvert Y_{T}-
\xi_{T}\rrvert ^{p}%
+ \biggl({\int
_{T}^{\infty}}\mathrm{e}^{2\tilde{V}_{s}}\llvert
Z_{s}%
-\zeta_{s}\rrvert ^{2}\,
\mathrm{d}s \biggr)^{p/2} \biggr]=0 \end{array} %
\end{equation}
and\vspace*{-2pt}
%
\begin{equation}
\label{ineq of sol 4} %
\mbox{\emph{(e)}} \quad \mathbb{E} {\int
_{0}^{\tau}} \bigl[ \mathrm{e}^{2\tilde{V}_{s}} \bigl(
\bigl\llvert U_{s}^{1}\bigr\rrvert ^{2}\,
\mathrm{d}s+\bigl\llvert U_{s}^{2}\bigr\rrvert ^{2}
\,\mathrm{d}A_{s} \bigr) \bigr] <\infty. %
\end{equation}
\end{theorem}

\begin{pf*}{Proof}
If $ (  Y,Z )  $, $(\bar
{Y},\bar{Z})$ are two solutions, in the sense of Definition
\ref{definition sol GBSVI}, that satisfy (\ref{ineq of sol 1}), then
$\mathbb{E}\sup_{s\in [  0,T ]  }\mathrm{e}^{p\tilde{V}_{s}}\llvert
Y_{s}-\bar{Y}_{s}\rrvert  ^{p}<\infty$. From (\ref{a priori estim 2}),
satisfied by the process $Y_{s}-\bar{Y}_{s}$, we conclude that
%
\begin{eqnarray}\label{approxim eq 3}
&& \bigl\langle Y_{s}-\bar{Y}_{s},\Phi (
s,Y_{s},Z_{s} ) -\Phi(s,\bar {Y}_{s},
\bar{Z}_{s})-U_{s}+\bar{U}_{s} \bigr\rangle\,\mathrm{d}Q_{s}
\nonumber\\[-8pt]\\[-8pt]\nonumber
&&\qquad \leq\llvert Y_{s}-\bar{Y}_{s}
\rrvert ^{2}\,\mathrm{d} \tilde{V}_{s}+\frac{1}{2a}
\llvert Z_{s}-\tilde{Z}_{s}\rrvert ^{2}\,
\mathrm{d}s, %
\end{eqnarray}
since $ \langle Y_{s}-\bar{Y}_{s},U_{s}-\bar{U}_{s} \rangle \geq0$,
for $ U_{s}\in\partial_{y}\Psi(s,Y_{s})$ and $\bar{U}_{s}\in\partial_{y}%
\Psi(s,\bar{Y}_{s})$, and $\mathrm{d}V_{s}\leq \mathrm{d}\tilde{V}_{s}$ on $ [
0,\tau ]  $.

Applying Proposition \ref{prop 1 Appendix} from the \hyperref[app]{Appendix}, it follows that
there exists $C=C (  a,p )  >0$ such that
\[
\mathbb{E}\sup_{s\in [  0,T ]  }\mathrm{e}^{p\tilde{V}_{s}}\llvert
Y_{s}- \bar{Y}_{s}\rrvert ^{p}+\mathbb{E}
\biggl( {\int_{0}^{T}}%
\mathrm{e}^{2\tilde{V}_{s}}\llvert Z_{s}- \bar{Z}_{s}\rrvert
^{2}\,\mathrm{d}s \biggr) ^{p/2}\leq C \mathbb{E}%
\bigl( \mathrm{e}^{p\tilde{V}_{T}}\llvert Y_{T}- \bar{Y}_{T}
\rrvert ^{p} \bigr) \xrightarrow{T\rightarrow\infty} {}0,
\]
and the uniqueness is proved.

The proof of the existence will be split into several steps.

A. \textit{Approximating problem.}
Let $n\in\mathbb{N}^{\ast}$ and $\varepsilon=1/n$. We consider the
approximating stochastic equation
%
\begin{eqnarray}
\label{approxim eq 1} %
&& Y_{t}^{n}+{\int
_{t}^{\infty}}\mathbh{1}_{ [  0,n ]
} ( s )
\nabla_{y}\Psi^{n} \bigl(s,Y_{s}^{n}
\bigr)\,\mathrm{d}Q_{s}
\nonumber
\\[-8pt]
\\[-8pt]
&&\quad=\eta +{\int_{t}^{\infty}}\mathbh{1}_{ [  0,n ]  }
( s ) \Phi \bigl(s,Y_{s}^{n},Z_{s}^{n}
\bigr)\,\mathrm{d}Q_{s} -{\int_{t}^{\infty}}Z_{s}^{n}
\,\mathrm{d}W_{s}, \qquad\mathbb{P}\mbox{-a.s., }\forall t\geq0,
\nonumber
\end{eqnarray}
or equivalent, $\mathbb{P}$-a.s.,
%
\begin{equation}
\label{approxim eq 2} \cases{\displaystyle Y_{t}^{n}+{\int
_{t}^{n}} \nabla_{y}
\Psi^{n} \bigl(s,Y_{s}^{n}%
\bigr)\,
\mathrm{d}Q_{s}
\cr
\displaystyle\quad=\mathbb{E}^{\mathcal{F}_{n}}\tilde{
\eta}+{\int_{t}^{n}%
}\Phi
\bigl(s,Y_{s}^{n},Z_{s}^{n} \bigr)\,
\mathrm{d}Q_{s}-{\int_{t}^{n}}Z_{s}^{n}
\,\mathrm{d}W_{s},\qquad \forall t\in [ 0,n ] ,
\cr
\displaystyle{
\bigl(Y_{t}^{n},Z_{t}^{n} \bigr)=(
\xi_{t},\zeta_{t}), \qquad\forall t>n,}%
}
\end{equation}
with $\Psi^{n} (  \omega,s,y )  :=\mathbh{1}_{ [  0,\tau (
\omega )   ]  } (  s )   [  \alpha_{s} (
\omega )  \varphi_{1/n} (  y )  + (  1-\alpha_{s} (
\omega )   )  \psi_{1/n} (  y )   ]  $.

We notice that $\Phi_{n} (  t,y,z )  :=\mathbh{1}_{ [
0,n ]  } (  t )   (\Phi (  t,y,z )  -\nabla_{y}%
\Psi^{n} (  t,y )   )$ satisfies inequalities
\begin{eqnarray*}
\bigl\langle y^{\prime}-y,\Phi_{n} \bigl(t,y^{\prime},z
\bigr)-\Phi_{n}%
(t,y,z) \bigr\rangle &\leq&
\mathbh{1}_{ [  0,n\wedge\tau ]  } ( t ) \bigl[ ( \mu_{t}-n )
\alpha_{t}+ ( \nu _{t}-n ) ( 1-\alpha_{t} )
\bigr] \bigl\llvert y^{\prime
}-y \bigr\rrvert ^{2}
\\
&\leq&\mathbh{1}_{ [  0,n\wedge\tau ]  } ( t ) \bigl[ \tilde{\mu}_{t}
\alpha_{t}+\tilde{\nu}_{t} ( 1-\alpha_{t} )
\bigr] \bigl\llvert y^{\prime}-y \bigr\rrvert ^{2}%
\end{eqnarray*}
and $\llvert  \Phi_{n}(t,y,z^{\prime})-\Phi_{n}(t,y,z)\rrvert
\leq\mathbh{1}_{ [  0,n\wedge\tau ]  } (  t )   \ell
\alpha_{t} \llvert  z^{\prime}-z\rrvert  $, since $\mu_{s}\leq\tilde
{\mu}_{s}$ and $\nu_{s}\leq\tilde{\nu}_{s}$ on $[0,{\infty})$.

The corresponding process $\tilde{V}_{t}^{n}$ (see definitions (\ref{def V})
and (\ref{def V tilde})) is given by
\[
\tilde{V}_{t}^{n}=\int_{0}^{t}
\mathbh{1}_{ [  0,n\wedge\tau ]
} ( s ) \biggl[ \biggl(\tilde{\mu}_{s}+
\frac{a}{2} \ell^{2}%
\biggr)\,\mathrm{d}s+\tilde{
\nu}_{s}\,\mathrm{d}A_{s} \biggr].
\]
Obviously, $\tilde{V}_{t}^{n}=\tilde{V}_{t\wedge n}$, $\forall t\geq0$.

Applying Proposition \ref{prop 3 Appendix} from the \hyperref[app]{Appendix}, with $\Phi$
replaced with $\Phi_{n}$, we deduce that equation (\ref{approxim eq 2}) has a
unique solution $(Y^{n},Z^{n})$ such that
\[
\mathbb{E}\sup_{s\in [  0,n ]  }\mathrm{e}^{p\tilde{V}_{s}} \bigl\llvert
Y_{s}^{n} \bigr\rrvert ^{p}+\mathbb{E} \biggl(
\int_{0}^{n}\mathrm{e}^{2\tilde{V}_{s}%
}\bigl\llvert
Z_{s}^{n}\bigr\rrvert ^{2}\,\mathrm{d}s \biggr)
^{p/2}<\infty,
\]
and, using (\ref{ineq_xi_zeta}), it can be prove that
\[
\mathbb{E}\sup_{s\in [  0,T ]  }\mathrm{e}^{p\tilde{V}_{s}} \bigl\llvert
Y_{s}^{n} \bigr\rrvert ^{p}+\mathbb{E} \biggl(
\int_{0}^{T}\mathrm{e}^{2\tilde{V}_{s}%
}\bigl\llvert
Z_{s}^{n}\bigr\rrvert ^{2}\,\mathrm{d}s \biggr)
^{p/2}<\infty, \qquad\mbox{for all }T\geq0.
\]

B. \textit{Boundedness of }$Y^{n}$ \textit{and} $Z^{n}$.
Since $\varphi^{n},\psi^{n}$ are convex functions and it is assumed that
$\varphi (  0 )  =\psi (  0 )  =0$, we see that
$\langle\nabla_{y}\Psi^{n} (  t,y )  ,y\rangle\geq0$, $\forall y\in
H$, and therefore (\ref{a priori estim 1}) becomes
\[
\bigl\langle Y_{t}^{n},\Phi_{n}
\bigl(t,Y_{t}^{n},Z_{t}^{n} \bigr)\,
\mathrm{d}Q_{t} \bigr\rangle\leq \mathbh{1}_{ [  0,n ]  } ( t ) \bigl
\llvert Y_{t}^{n}\bigr\rrvert \bigl\llvert \Phi ( t,0,0 )
\bigr\rrvert \,\mathrm{d}Q_{t}+\bigl\llvert Y_{t}^{n}
\bigr\rrvert ^{2}\,\mathrm{d} \tilde{V}_{t}^{n}+
\frac{1}{2a}%
\bigl\llvert Z_{t}^{n}\bigr
\rrvert ^{2}\,\mathrm{d}t.
\]
Equation (\ref{approxim eq 1}) can be written, for any $T\geq0$, in the form
\[
Y_{t}^{n}=Y_{T}^{n}+{\int
_{t}^{T}}\Phi_{n} \bigl(s,Y_{s}^{n},Z_{s}^{n}
\bigr)\,\mathrm{d}Q_{s}%
-{\int_{t}^{T}}Z_{s}^{n}
\,\mathrm{d}W_{s}, \qquad\mathbb{P}\mbox{-a.s., }\forall t\in [ 0,T ] .
\]
Applying Proposition \ref{prop 1 Appendix} (see the \hyperref[app]{Appendix}), we deduce that,
for all $q\in[2,p]$, there exists a constant $C=C (  a,q )  >0$
such that such that, $\mathbb{P}$-a.s., for all $0\leq t\leq T\leq n$,
\begin{eqnarray*}
&&\mathbb{E}^{\mathcal{F}_{t}}\sup_{s\in [  t,T ]  }\mathrm{e}^{q\tilde{V}_{s}%
}
\bigl\llvert Y_{s}^{n} \bigr\rrvert ^{q}+
\mathbb{E}^{\mathcal{F}_{t}%
} \biggl({\int_{t}^{T}}
\mathrm{e}^{2\tilde{V}_{s}}\bigl\llvert Z_{s}^{n}\bigr\rrvert
^{2}\,\mathrm{d}s \biggr)^{q/2}
\\
&&\quad\leq C \mathbb{E}^{\mathcal{F}_{t}} \biggl[\mathrm{e}^{q\tilde{V}_{T}}\bigl
\llvert Y_{T}^{n}%
\bigr\rrvert ^{q}+
\biggl({\int_{t}^{T}} \mathbh{1}_{ [  0,n ]  } (
s ) \mathrm{e}^{\tilde{V}_{s}}\bigl\llvert \Phi ( s,0,0 ) \bigr\rrvert \,
\mathrm{d}Q_{s} \biggr)^{q}%
\biggr]
\\
&&\quad\leq C \mathbb{E}^{\mathcal{F}_{t}} \biggl[\mathrm{e}^{q\sup_{s\in [  t,T ]
}\tilde{V}_{s}}\llvert
\xi_{T}\rrvert ^{q}+ \biggl({\int_{t}^{{\infty}}%
}\mathbh{1}_{ [  0,n ]  } ( s ) \mathrm{e}^{\tilde{V}_{s}}%
\bigl
\llvert \Phi ( s,0,0 ) \bigr\rrvert \,\mathrm{d}Q_{s}
\biggr)^{q} \biggr]
\\
&&\quad\leq C \mathbb{E}^{\mathcal{F}_{t}} \biggl[\mathrm{e}^{q\sup_{s\in [  t,T ]
}\tilde{V}_{s}}\llvert
\eta\rrvert ^{q}+ \biggl({\int_{t}^{{\infty}}%
}\mathbh{1}_{ [  0,n ]  } ( s ) \mathrm{e}^{\tilde{V}_{s}}%
\bigl
\llvert \Phi ( s,0,0 ) \bigr\rrvert \,\mathrm{d}Q_{s}
\biggr)^{q} \biggr],
\end{eqnarray*}
since by Jensen's inequality we have $\llvert \xi_{T}\rrvert ^{q}=\llvert \mathbb{E}^{\mathcal{F}%
_{T}}\eta\rrvert ^{q}\leq\mathbb{E}^{\mathcal{F}_{T}}\llvert  \eta\rrvert  ^{q}$.

Using (\ref{ineq_xi_zeta}), it can be proved that the above inequality holds
also for all $0\leq t\vee n\leq T$.

Passing to limit as $T\rightarrow\infty$ we infer, using Beppo Levi's theorem,
that $\mathbb{P}$-a.s.
%
\begin{eqnarray}
\label{boundedness for approxim}%
&&\mathbb{E}^{\mathcal{F}_{t}}\sup_{s\geq t}
\mathrm{e}^{q\tilde{V}_{s}}\bigl\llvert Y_{s}^{n}%
\bigr\rrvert ^{q}+\mathbb{E}^{\mathcal{F}_{t}} \biggl( {\int
_{t}%
^{{\infty}}}\mathrm{e}^{2\tilde{V}_{s}}
\bigl\llvert Z_{s}^{n}\bigr\rrvert ^{2}\,
\mathrm{d}s \biggr) ^{q/2}%
\nonumber
\\[-8pt]
\\[-8pt]
&&\quad\leq C \mathbb{E}^{\mathcal{F}_{t}} \biggl[ \bigl(\mathrm{e}^{q\sup_{s\in [  t,\tau ]
}\tilde{V}_{s}}
\llvert \eta\rrvert ^{q} \bigr)+ \biggl({\int_{t}^{{\infty}}%
}\mathbh{1}_{ [  0,n ]  } ( s ) \mathrm{e}^{\tilde{V}_{s}}%
\bigl
\llvert \Phi ( s,0,0 ) \bigr\rrvert \,\mathrm{d}Q_{s}
\biggr)^{q} \biggr].
\nonumber
\end{eqnarray}
In particular, for $q=2$, there exists another constant $C\geq1$ such that,
for all $t\geq0$,
%
\begin{eqnarray}
\label{boundedness for approxim 2}%
\mathrm{e}^{2\tilde{V}_{t}}\bigl\llvert
Y_{t}^{n}\bigr\rrvert ^{2}& \leq& C
\mathbb{E}^{\mathcal{F}_{t}%
} \biggl[ \bigl(\mathrm{e}^{2\sup_{s\in [  t,\tau ]  }\tilde{V}_{s}}\llvert \eta
\rrvert ^{2} \bigr)\nonumber\\
&&{}+ \biggl({\int_{t}^{{\infty}}}
\mathbh{1}_{ [
0,n ]  } ( s ) \mathrm{e}^{\tilde{V}_{s}}\bigl\llvert \Phi (
s,0,0 ) \bigr\rrvert \,\mathrm{d}Q_{s} \biggr)^{2} \biggr]
\\
&\leq&2CR_{0}^{2}=R_{0}^{\prime}, \qquad
\mathbb{P}\mbox{-a.s.,}
\nonumber
\end{eqnarray}
where $R_{0}$ is given by (\ref{Assumption A_10}).

\begin{remark}\label{remark_(A9)'}
We emphasis that we just use, for the first time,
assumption {(A$_{9}$)} and the obtained inequality
(\ref{boundedness for approxim 2}) will be essential in order to deduce, using
assumption {(A$_{8}$)} of locally boundedness for the generators, the
subsequent step $\mathrm{D}$ (i.e., the boundedness of the gradient of
$\Psi^{n}$) and what follows afterwards.
\end{remark}

\textrm{C.} \textit{Other boundedness results on }$Y^{n}$ \textit{and} $Z^{n}$.
Since for all $u\in H$, $\langle u-y,\nabla_{y}\Psi^{n} (  t,y )
\rangle\leq\Psi^{n} (  t,u )  -\Psi^{n} (  t,y )  $, we can
deduce (see inequality (\ref{a priori estim 2})) that, as signed measures on
$ [  0,n ]  $,
%
\begin{eqnarray}
\label{boundedness for approxim 3}%
&& \bigl\langle Y_{t}^{n}-
\xi_{t},\Phi_{n} \bigl(s,Y_{s}^{n},Z_{s}^{n}%
\bigr) \bigr\rangle\,\mathrm{d}Q_{t}
\nonumber
\\
&&\quad\leq \bigl[ \Psi^{n}(t,\xi_{t})-\Psi^{n}
\bigl(t,Y_{t}%
^{n} \bigr) \bigr]\,
\mathrm{d}Q_{t}
\\
&&\qquad{} +\bigl\llvert Y_{t}^{n}-\xi_{t}\bigr
\rrvert \bigl\llvert \Phi(t, \xi_{t},\zeta_{t})\bigr\rrvert
\,\mathrm{d}Q_{t}+\bigl\llvert Y_{t}^{n}%
\bigr\rrvert ^{2}\,\mathrm{d}\tilde{V}_{t}^{n}+
\frac{1}{2a}\bigl\llvert Z_{t}^{n}-
\zeta_{t}\bigr\rrvert ^{2}\,\mathrm{d}t.
\nonumber
\end{eqnarray}
But $0\leq\Psi^{n}(t,\xi_{t})\leq\Psi(t,\xi_{t})=\mathbh{1}_{ [
0,\tau (  \omega )   ]  } (  t )   [  \alpha
_{t} (  \omega )  \varphi(\xi_{t})+ (  1-\alpha_{t} (
\omega )   )  \psi(\xi_{t}) ]  $, therefore
(\ref{boundedness for approxim 3}) becomes
\begin{eqnarray*}
&& \Psi^{n} \bigl(t,Y_{t}^{n} \bigr)\,
\mathrm{d}Q_{t}+ \bigl\langle Y_{t}^{n}-
\xi_{t},\Phi_{n} \bigl(s,Y_{s}%
^{n},Z_{s}^{n} \bigr) \bigr\rangle\,\mathrm{d}Q_{t}
\\
&&\quad\leq\Psi(t,\xi_{t})\,\mathrm{d}Q_{t}+\bigl\llvert
Y_{t}^{n}- \xi_{t}\bigr\rrvert \bigl\llvert
\Phi(t,\xi_{t},\zeta _{t})\bigr\rrvert \,
\mathrm{d}Q_{t}+\bigl\llvert Y_{t}^{n}\bigr
\rrvert ^{2}\,\mathrm{d} \tilde{V}_{t}^{n}+
\frac{1}{2a}\bigl\llvert Z_{t}^{n}%
-
\zeta_{t}\bigr\rrvert ^{2}\,\mathrm{d}t.
\end{eqnarray*}
From (\ref{approxim eq 2}), we see that $(Y^{n},Z^{n})$ satisfies the equation
\[
Y_{t}^{n}-\xi_{t}={\int_{t}^{n}}
\Phi_{n} \bigl(s,Y_{s}^{n},Z_{s}^{n}
\bigr)\,\mathrm{d}Q_{s}%
-{\int_{t}^{n}
\bigl(}Z_{s}^{n}-\zeta_{s} \bigr)\,
\mathrm{d}W_{s}, \qquad\forall t\in [ 0,n ] ,
\]
since $\xi_{t}=\xi_{n}-{\int_{t}^{n}}\zeta_{s}\,\mathrm{d}W_{s}$, $\forall t\in [
0,n ]  $.

Applying again Proposition \ref{prop 1 Appendix}, there exists a constant
$C=C (  a,p )  >0$ such that, $\mathbb{P}$-a.s., for all
$t\in [  0,n ]  $,
\begin{eqnarray*}
&&\mathbb{E}^{\mathcal{F}_{t}}\sup_{s\in [  t,n ]  }\mathrm{e}^{p\tilde{V}_{s}%
}
\bigl\llvert Y_{s}^{n}-\xi_{s} \bigr\rrvert
^{p}+\mathbb{E}^{\mathcal{F}_{t}%
} \biggl({\int_{t}^{n}}
\mathrm{e}^{2\tilde{V}_{s}}\bigl\llvert Z_{s}^{n}-
\zeta_{s}%
\bigr\rrvert ^{2}\,\mathrm{d}s
\biggr)^{p/2}
\\
&&\qquad{}+\mathbb{E} \biggl({\int_{t}^{n}}
\mathrm{e}^{2\tilde{V}_{s}} \Psi ^{n} \bigl(s,Y_{s}^{n}
\bigr)\,\mathrm{d}Q_{s} \biggr)^{p/2}
\\
&&\quad\leq C \mathbb{E}^{\mathcal{F}_{t}} \biggl[ \biggl({\int
_{t}^{n}%
}\mathrm{e}^{2\tilde{V}_{s}}
\Psi(s, \xi_{s})\,\mathrm{d}Q_{s} \biggr)^{p/2}+
\biggl({\int_{t}^{n}}\mathrm{e}^{\tilde{V}_{s}}
\bigl\llvert \Phi(s, \xi_{s},\zeta_{s})\bigr\rrvert \,
\mathrm{d}Q_{s} \biggr)^{p} \biggr].
\end{eqnarray*}
Therefore
%
\begin{eqnarray}
\label{boundedness for approxim 4}%
&&\mathbb{E}^{\mathcal{F}_{t}}\sup_{s\geq t}
\mathrm{e}^{p\tilde{V}_{s}}\bigl\llvert Y_{s}^{n}- \xi
_{s}\bigr\rrvert ^{p}+\mathbb{E}^{\mathcal{F}_{t}} \biggl({
\int_{t}%
^{{\infty}}}\mathrm{e}^{2\tilde{V}_{s}}
\bigl\llvert Z_{s}^{n}- \zeta_{s}\bigr\rrvert
^{2}\,\mathrm{d}s \biggr)^{p/2}%
\nonumber
\\[-8pt]
\\[-8pt]
&&\quad\leq C \mathbb{E}^{\mathcal{F}_{t}} \biggl[ \biggl({\int
_{t}^{{\infty}}}\mathrm{e}^{2\tilde{V}_{s}}\Psi(s,
\xi_{s})\,\mathrm{d}Q_{s} \biggr)^{p/2}%
+ \biggl({ \int_{t}^{{\infty}}}\mathrm{e}^{\tilde{V}_{s}}
\bigl\llvert \Phi (s, \xi_{s},\zeta_{s})\bigr\rrvert \,
\mathrm{d}Q_{s} \biggr)^{p} \biggr]
\nonumber
\end{eqnarray}
since $(Y_{s}^{n},Z_{s}^{n})=(\xi_{s},\zeta_{s})$ for $s>n$.

D. \textit{Boundedness of }$\nabla\varphi_{1/n}(Y_{t}^{n}%
)$\textit{ and }$\nabla\psi_{1/n}(Y_{t}^{n})$.
In order to obtain the boundedness for $\llvert \nabla\varphi_{1/n}(Y_{s}^{n})\rrvert ^{2}$
it is essential to use the following stochastic subdifferential inequality
(see Proposition 11 in Maticiuc and R\u{a}\c{s}canu \cite{ma-ra/10}), written first for $\varphi_{1/n}$:
for all $0\leq t\leq s\leq n$
\[
\mathrm{e}^{2\tilde{V}_{s}}\varphi_{1/n} \bigl(Y_{s}^{n}
\bigr)\geq \mathrm{e}^{2\tilde{V}_{t}}\varphi _{1/n}
\bigl(Y_{t}^{n} \bigr)+{\int_{t}^{s}}
\varphi_{1/n} \bigl(Y_{r}^{n} \bigr)\,\mathrm{d}
\bigl(\mathrm{e}^{2\tilde{V}_{r}%
} \bigr)+{\int_{t}^{s}}
\mathrm{e}^{2\tilde{V}_{r}}\nabla \varphi_{1/n} \bigl(Y_{r}^{n}
\bigr)\,\mathrm{d}Y_{r}^{n}.
\]
Hence,
\[
\mathrm{e}^{2\tilde{V}_{s}}\varphi_{1/n} \bigl(Y_{s}^{n}
\bigr)\geq \mathrm{e}^{2\tilde{V}_{t}}\varphi _{1/n}
\bigl(Y_{t}^{n} \bigr)+2{\int_{t}^{s}}
\mathrm{e}^{2\tilde{V}_{r}} \varphi_{1/n} \bigl(Y_{r}%
^{n} \bigr)\,\mathrm{d} \tilde{V}_{r}+{\int
_{t}^{s}}\mathrm{e}^{2\tilde{V}_{r}} \nabla
\varphi_{1/n}%
\bigl(Y_{r}^{n} \bigr)\,
\mathrm{d}Y_{r}^{n}.
\]
It follows that, $\mathbb{P}$-a.s. for all $0\leq t\leq s\leq n$,
\begin{eqnarray*}
&&\mathrm{e}^{2\tilde{V}_{t}}\varphi_{1/n} \bigl(Y_{t}^{n}
\bigr)+{\int_{t}^{s}%
}
\mathrm{e}^{2\tilde{V}_{r}} \bigl\langle\nabla\varphi_{1/n}
\bigl(Y_{r}^{n} \bigr),\mathbh{1}%
_{ [  0,n ]  } ( r ) \nabla_{y} \Psi^{n} \bigl(
r,Y_{r}%
^{n} \bigr) \bigr\rangle\,\mathrm{d}Q_{r}
\\
&&\quad\leq \mathrm{e}^{2\tilde{V}_{s}}\varphi_{1/n}
\bigl(Y_{s}^{n} \bigr)+{\int_{t}^{s}%
}\mathrm{e}^{2\tilde{V}_{r}} \bigl\langle\nabla\varphi_{1/n}
\bigl(Y_{r}^{n} \bigr),\mathbh{1}%
_{ [  0,n ]  } ( r ) \Phi \bigl(r,Y_{r}^{n},Z_{r}^{n}%
\bigr)\,\mathrm{d}Q_{r} \bigr\rangle
\\
&&\qquad{}-{\int_{t}^{s}}\mathrm{e}^{2\tilde{V}_{r}}
\bigl\langle\nabla \varphi_{1/n} \bigl(Y_{r}^{n}
\bigr),Z_{r}^{n}\,\mathrm{d}W_{r} \bigr
\rangle-2{\int_{t}^{s}}\mathrm{e}^{2\tilde{V}_{r}}
\varphi_{1/n} \bigl(Y_{r}^{n} \bigr)\,\mathrm{d}
\tilde{V}_{r}%
\end{eqnarray*}
(and a similar inequality for $\psi_{1/n}$).

Since
%
\begin{equation}
\label{positivity of phi} \varphi_{1/n}(0)+\psi_{1/n}(0)=0\leq
\varphi_{1/n}(y)+\psi_{1/n}(y)\leq \varphi(y)+\psi(y), \qquad
\forall y\in H, %
\end{equation}
we infer that
%
\begin{eqnarray}
\label{boundedness for approxim 5} %
&&\mathrm{e}^{2\tilde{V}_{t}} \bigl[
\varphi_{1/n} \bigl(Y_{t}^{n} \bigr)+
\psi_{1/n} \bigl(Y_{t}%
^{n} \bigr)
\bigr] +{\int_{t}^{s}}\mathbh{1}_{ [  0,n\wedge
\tau ]  } (
r ) \mathrm{e}^{2\tilde{V}_{r}} \bigl[\alpha_{r}%
\bigl
\llvert \nabla \varphi_{1/n} \bigl(Y_{r}^{n}
\bigr)\bigr\rrvert ^{2}
\nonumber
\\
&&\qquad{}+ \bigl\langle\nabla\varphi_{1/n} \bigl(Y_{r}^{n}
\bigr),\nabla\psi_{1/n} \bigl(Y_{r}%
^{n} \bigr) \bigr\rangle+ ( 1-\alpha_{r} ) \bigl\llvert
\nabla \psi_{1/n} \bigl(Y_{r}%
^{n}
\bigr)\bigr\rrvert ^{2} \bigr]\,\mathrm{d}Q_{r}
\nonumber
\\
&&\quad\leq \mathrm{e}^{2\tilde{V}_{s}} \bigl[ \varphi \bigl(Y_{s}^{n}
\bigr)+ \psi \bigl(Y_{s}^{n} \bigr) \bigr] -2{\int
_{t}^{s}}\mathrm{e}^{2\tilde{V}_{r}} \bigl[
\varphi_{1/n} \bigl(Y_{r}%
^{n} \bigr)+
\psi_{1/n} \bigl(Y_{r}^{n} \bigr) \bigr]\,
\mathrm{d} \tilde{V}_{r}
\\
&&\qquad{}-{\int_{t}^{s}}\mathrm{e}^{2\tilde{V}_{r}}
\bigl\langle \nabla \varphi_{1/n} \bigl(Y_{r}^{n}
\bigr)+\nabla \psi_{1/n} \bigl(Y_{r}^{n}
\bigr),Z_{r}^{n}\,\mathrm{d}W_{r}%
\bigr\rangle
\nonumber
\\
&&\qquad{}+{\int_{t}^{s}}\mathrm{e}^{2\tilde{V}_{r}}
\bigl\langle \nabla \varphi_{1/n} \bigl(Y_{r}^{n}
\bigr)+\nabla \psi_{1/n} \bigl(Y_{r}^{n} \bigr),
\mathbh{1}_{ [
0,n ]  } ( r ) \Phi \bigl(r,Y_{r}^{n},Z_{r}^{n}
\bigr) \bigr\rangle\,\mathrm{d}Q_{r}.
\nonumber
\end{eqnarray}
Using definition of $\Phi$, compatibility assumptions
(\ref{compatib. assumption}) gives us
%
\begin{eqnarray}
\label{boundedness for approxim 6} %
&& \bigl\langle \nabla \varphi_{\varepsilon}(y),
\Phi(t,y,z) \bigr\rangle
\nonumber
\\
&&\quad=\mathbh{1}_{ [  0,\tau ]  } ( t ) \bigl\langle \nabla
\varphi_{\varepsilon}(y),\alpha_{t}F(t,y,z)+ ( 1-\alpha_{t}
) G(t,y) \bigr\rangle
\nonumber
\\[-8pt]
\\[-8pt]
&&\quad\leq\mathbh{1}_{ [  0,\tau ]  } ( t ) \bigl(\alpha _{t}\bigl
\llvert F(t,y,z)\bigr\rrvert \bigl\llvert \nabla\varphi_{\varepsilon}(y)\bigr
\rrvert + ( 1-\alpha_{t} ) \bigl\llvert G(t,y)\bigr\rrvert \bigl\llvert
\nabla \psi_{\varepsilon}(y)\bigr\rrvert
\nonumber
\\
&&\qquad{}+ ( 1-\alpha_{t} ) \nu_{t}^{-}\llvert
y\rrvert \bigl\llvert \nabla \psi_{\varepsilon
}(y)\bigr\rrvert - ( 1-
\alpha_{t} ) \nu_{t}^{-} \bigl\langle \nabla
\varphi_{\varepsilon}(y),y \bigr\rangle \bigr)
\nonumber
\end{eqnarray}
and respectively,
%
\begin{eqnarray}
\label{boundedness for approxim 6'} %
&& \bigl\langle \nabla
\psi_{\varepsilon}(y),\Phi(t,y,z) \bigr\rangle
\nonumber
\\
&&\quad =\mathbh{1}_{ [  0,\tau ]  } ( t ) \bigl\langle \nabla
\psi_{\varepsilon}(y),\alpha_{t}F(t,y,z)+ ( 1- \alpha_{t} )
G(t,y) \bigr\rangle
\nonumber
\\[-8pt]
\\[-8pt]
&&\quad\leq\mathbh{1}_{ [  0,\tau ]  } ( t ) \bigl(\alpha _{t}\bigl
\llvert F(t,y,z)\bigr\rrvert \bigl\llvert \nabla\varphi_{\varepsilon}(y)\bigr
\rrvert + ( 1-\alpha_{t} ) \bigl\llvert G(t,y)\bigr\rrvert \bigl\llvert
\nabla \psi_{\varepsilon}(y)\bigr\rrvert
\nonumber
\\
&&\qquad{} +\alpha_{t}\mu_{t}^{-}\llvert y
\rrvert \bigl\llvert \nabla \varphi_{\varepsilon}(y)\bigr\rrvert -
\alpha_{t}%
\mu_{t}^{-} \bigl\langle
\nabla\psi_{\varepsilon}(y),y \bigr\rangle \bigr).
\nonumber
\end{eqnarray}
From (\ref{compatib. assumption})(i), (\ref{boundedness for approxim 2}),
(\ref{boundedness for approxim 5})--(\ref{boundedness for approxim 6'}) and
inequality $2ab\leq\frac{1}{\alpha}a^{2}+\alpha b^{2}$ with $\alpha\in \{
2,4 \}  $, we obtain
%
\begin{eqnarray}
\label{boundedness for approxim 7} %
&&\mathrm{e}^{2\tilde{V}_{t}} \bigl[
\varphi_{1/n} \bigl(Y_{t}^{n} \bigr)+
\psi_{1/n} \bigl(Y_{t}%
^{n} \bigr)
\bigr] +\frac{1}{2}{\int_{t}^{s}}
\mathbh{1}_{ [
0,n\wedge\tau ]  } ( r ) \mathrm{e}^{2\tilde{V}_{r}} \bigl[\bigl\llvert
\nabla \varphi_{1/n} \bigl(Y_{r}^{n} \bigr)\bigr
\rrvert ^{2}\,\mathrm{d}r+\bigl\llvert \nabla\psi_{1/n}
\bigl(Y_{r}^{n} \bigr)\bigr\rrvert ^{2}\,
\mathrm{d}A_{r} \bigr]
\nonumber
\\
&&\quad\leq \mathrm{e}^{2\tilde{V}_{s}} \bigl[ \varphi \bigl(Y_{s}^{n}
\bigr)+ \psi \bigl(Y_{s}^{n} \bigr) \bigr] +{\int
_{t}^{s}} \mathbh{1}_{ [  0,n\wedge\tau ]  } ( r )
\mathrm{e}^{2\tilde{V}_{r}}\bigl\llvert Y_{r}^{n}\bigr\rrvert
^{2} \bigl[\bigl\llvert \mu_{r}^{-}\bigr\rrvert
^{2}\,\mathrm{d}r+\bigl\llvert \nu _{r}^{-}\bigr
\rrvert ^{2}\,\mathrm{d}A_{r} \bigr]
\nonumber
\\
&&\qquad{}+{\int_{t}^{s}}\mathbh{1}_{ [  0,n\wedge\tau ]
}
( r ) 4\mathrm{e}^{2\tilde{V}_{r}} \bigl[\bigl\llvert F \bigl(r,Y_{r}^{n},Z_{r}^{n}%
\bigr)\bigr\rrvert ^{2}\,\mathrm{d}r+\bigl\llvert G
\bigl(r,Y_{r}^{n} \bigr)\bigr\rrvert ^{2}\,
\mathrm{d}A_{r} \bigr]
\\
&&\qquad{}-{\int_{t}^{s}}\mathbh{1}_{ [  0,n\wedge\tau ]
}
( r ) \mathrm{e}^{2\tilde{V}_{r}} \bigl[ \mu_{r}^{-} \bigl
\langle\nabla \psi_{1/n} \bigl(Y_{r}^{n}
\bigr),Y_{r}^{n} \bigr\rangle \,\mathrm{d}r+\nu_{r}^{-}
\bigl\langle\nabla \varphi_{1/n} \bigl(Y_{r}^{n}
\bigr),Y_{r}^{n} \bigr\rangle\,\mathrm{d}A_{r} \bigr]
\nonumber
\\
&&\qquad{}-{\int_{t}^{s}}\mathrm{e}^{2\tilde{V}_{r}}
\bigl\langle \nabla \varphi_{1/n} \bigl(Y_{r}^{n}
\bigr)+\nabla \psi_{1/n} \bigl(Y_{r}^{n}
\bigr),Z_{r}^{n}\,\mathrm{d}W_{r}%
\bigr\rangle-2{\int_{t}^{s}}\mathrm{e}^{2\tilde{V}_{r}}
\bigl[ \varphi_{1/n} \bigl(Y_{r}^{n} \bigr)+
\psi_{1/n} \bigl(Y_{r}^{n} \bigr) \bigr]\,
\mathrm{d} \tilde{V}_{r}.
\nonumber
\end{eqnarray}
Using (\ref{positivity of phi}), the definition of $\tilde{V}$ and inequality
$0\leq\varphi_{1/n} (  y )  \leq \langle \nabla\varphi
_{1/n}(y),y \rangle , \forall y\in H$, we have
%
\begin{eqnarray}
\label{boundedness for approxim 7''} %
&&{-}\mu_{r}^{-}
\bigl\langle\nabla\psi_{1/n} \bigl(Y_{r}^{n}
\bigr),Y_{r}^{n} \bigr\rangle\, \mathrm{d}r-\nu _{r}^{-}
\bigl\langle\nabla\varphi_{1/n} \bigl(Y_{r}^{n}
\bigr),Y_{r}^{n} \bigr\rangle\,\mathrm{d}A_{r}%
-2 \bigl[ \varphi_{1/n} \bigl(Y_{r}^{n} \bigr)+
\psi_{1/n} \bigl(Y_{r}^{n} \bigr) \bigr]\,
\mathrm{d}\tilde {V}_{r}
\nonumber
\\
&&\quad\leq-\mu_{r}^{-} \bigl\langle\nabla
\psi_{1/n} \bigl(Y_{r}^{n} \bigr),Y_{r}^{n}%
\bigr\rangle\, \mathrm{d}r-\nu_{r}^{-} \bigl\langle\nabla
\varphi_{1/n} \bigl(Y_{r}^{n}
\bigr),Y_{r}%
^{n} \bigr\rangle\,\mathrm{d}A_{r}
\nonumber
\\[-8pt]
\\[-8pt]
&&\qquad{}+ \bigl[ \varphi_{1/n} \bigl(Y_{r}^{n}
\bigr)+ \psi_{1/n} \bigl(Y_{r}^{n} \bigr) \bigr]
\bigl[ \mu_{r}^{-}\,\mathrm{d}r+\nu_{r}^{-}
\,\mathrm{d}A_{r} \bigr]
\nonumber
\\
&&\quad\leq \bigl\langle\nabla\varphi_{1/n} \bigl(Y_{r}^{n}
\bigr),Y_{r}^{n} \bigr\rangle\mu_{r}%
^{-}\,\mathrm{d}r+ \bigl\langle\nabla\psi_{1/n}
\bigl(Y_{r}^{n} \bigr),Y_{r}^{n}
\bigr\rangle\nu_{r}%
^{-}\,\mathrm{d}A_{r},
\nonumber
\end{eqnarray}
since $2\mathrm{d}\tilde{V}_{r}\geq \mathrm{d}V_{r}\geq-\mu_{r}^{-}\,\mathrm{d}r-\nu_{r}^{-}\,\mathrm{d}A_{r}$ on
$ [  0,\tau ]  $.

Hence,
%
\begin{eqnarray}
\label{boundedness for approxim 7'} %
&&{-} {\int_{t}^{s}}
\mathbh{1}_{ [  0,n\wedge\tau ]  } ( r ) \mathrm{e}^{2\tilde{V}_{r}} \bigl[
\mu_{r}^{-} \bigl\langle\nabla\psi _{1/n}
\bigl(Y_{r}^{n} \bigr),Y_{r}^{n}
\bigr\rangle \,\mathrm{d}r+\nu_{r}^{-} \bigl\langle\nabla
\varphi_{1/n} \bigl(Y_{r}^{n}
\bigr),Y_{r}^{n} \bigr\rangle\,\mathrm{d}A_{r} \bigr]
\nonumber
\\
&&\qquad{-}\,2{\int_{t}^{s}}\mathrm{e}^{2\tilde{V}_{r}}
\bigl[ \varphi _{1/n} \bigl(Y_{r}^{n} \bigr)+
\psi_{1/n} \bigl(Y_{r}^{n} \bigr) \bigr]\,
\mathrm{d}\tilde{V}_{r}
\nonumber
\\[-8pt]
\\[-8pt]
&&\quad\leq{\int_{t}^{s}}\mathbh{1}_{ [  0,n\wedge\tau ]
}
( r ) \mathrm{e}^{2\tilde{V}_{r}} \biggl[\frac{1}{4}\bigl\llvert \nabla
\varphi _{1/n} \bigl(Y_{r}^{n} \bigr)\bigr\rrvert
^{2}\,\mathrm{d}r+\frac{1}{4}\bigl\llvert \nabla
\psi_{1/n} \bigl(Y_{r}^{n} \bigr)\bigr\rrvert
^{2}\,\mathrm{d}A_{r} \biggr]
\nonumber
\\
&&\qquad{}+{\int_{t}^{s}}\mathbh{1}_{ [  0,n\wedge\tau ]
}
( r ) \mathrm{e}^{2\tilde{V}_{r}}\bigl\llvert Y_{r}^{n}\bigr
\rrvert ^{2} \bigl[ \bigl\llvert \mu _{r}^{-}
\bigr\rrvert ^{2}\,\mathrm{d}r+ \bigl\llvert \nu_{r}^{-}
\bigr\rrvert ^{2}\,\mathrm{d}A_{r} \bigr].
\nonumber
\end{eqnarray}
Moreover, we see that $\mathbb{E}^{\mathcal{F}_{t}}\int_{t}%
^{s}\mathrm{e}^{2\tilde{V}_{r}} \langle \nabla\varphi_{1/n}(Y_{r}^{n})+\nabla
\psi_{1/n}(Y_{r}^{n}),Z_{r}^{n}\,\mathrm{d}W_{r} \rangle =0$, because
\begin{eqnarray*}
&&\mathbb{E} \biggl( {\int_{t}^{s}}
\mathrm{e}^{4\tilde{V}_{r}} \bigl( \bigl\llvert \nabla \varphi_{1/n}
\bigl(Y_{r}^{n} \bigr)\bigr\rrvert +\bigl\llvert \nabla
\psi_{1/n} \bigl(Y_{r}^{n} \bigr)\bigr\rrvert
\bigr) ^{2}\bigl\llvert Z_{r}^{n}\bigr\rrvert
^{2}\,\mathrm{d}r \biggr) ^{1/2}
\\
&&\quad\leq\mathbb{E} \biggl[ \sup_{r\in [  t,s ]  } \bigl(2n
\mathrm{e}^{\tilde{V}_{r}%
} \bigl\llvert Y_{r}^{n} \bigr
\rrvert \bigr) \biggl( {\int_{t}^{s}%
}
\mathrm{e}^{2\tilde{V}_{r}}\bigl\llvert Z_{r}^{n}\bigr\rrvert
^{2}\,\mathrm{d}r \biggr) ^{1/2} \biggr]
\\
&&\quad\leq2\sqrt{2}n\sqrt{R_{0}^{\prime}} \biggl[
\mathbb{E} \biggl( {\int_{t}^{s}}
\mathrm{e}^{2\tilde{V}_{r}}\bigl\llvert Z_{r}^{n}- \tilde{
\zeta}_{r}%
\bigr\rrvert ^{2}\,\mathrm{d}r \biggr)
^{1/2}+\mathbb{E} \biggl( {\int_{t}^{s}}%
\mathrm{e}^{2\tilde{V}_{r}}\llvert \tilde{\zeta}_{r}\rrvert
^{2}\,\mathrm{d}r \biggr) ^{1/2} \biggr]
\\
&&\quad <\infty.
\end{eqnarray*}
For $s=n$, Jensen's inequality yields
\[
\mathbb{E} \bigl[ \mathrm{e}^{2\tilde{V}_{s}}\varphi \bigl(Y_{s}^{n}
\bigr)+\psi \bigl(Y_{s}^{n} \bigr) \bigr] =\mathbb{E}
\bigl[ \mathrm{e}^{2\tilde{V}_{n}} \bigl( \varphi(\xi_{n})+\psi(\xi
_{n}) \bigr) \bigr] \leq\mathbb{E} \bigl[ \mathrm{e}^{2\tilde{V}_{n}}
\bigl( \varphi(\eta)+\psi( \eta) \bigr) \bigr] ,
\]
and using (\ref{boundedness for approxim 7'}), inequality
(\ref{boundedness for approxim 7}) becomes
\begin{eqnarray*}
&&\mathbb{E} \bigl[ \mathrm{e}^{2\tilde{V}_{t}} \bigl(
\varphi_{1/n} \bigl(Y_{t}^{n}%
\bigr)+
\psi_{1/n} \bigl(Y_{t}^{n} \bigr) \bigr) \bigr]
\\
&&\qquad{}+\frac{1}{4}\mathbb{E}%
{\int_{t}^{n}}
\mathbh{1}_{ [  0,n\wedge\tau ]  } ( r ) \mathrm{e}^{2\tilde{V}_{r}} \bigl[\bigl\llvert
\nabla \varphi_{1/n} \bigl(Y_{r}^{n} \bigr)\bigr
\rrvert ^{2}\,\mathrm{d}r+\bigl\llvert \nabla\psi_{1/n}
\bigl(Y_{r}^{n} \bigr)\bigr\rrvert ^{2}\,
\mathrm{d}A_{r} \bigr]
\\
&&\quad\leq\mathbb{E} \bigl[ \mathrm{e}^{2\tilde{V}_{n}} \bigl( \varphi(\eta)+\psi
(\eta) \bigr) \bigr] +4\mathbb{E} {\int_{t}^{n}}
\mathbh{1}%
_{ [  0,n\wedge\tau ]  } ( r ) \mathrm{e}^{2\tilde{V}_{r}} \bigl(
\bigl\llvert F \bigl(r,Y_{r}^{n},Z_{r}^{n}
\bigr)\bigr\rrvert ^{2}\,\mathrm{d}r+\bigl\llvert G
\bigl(r,Y_{r}^{n} \bigr)\bigr\rrvert ^{2}\,
\mathrm{d}A_{r} \bigr)
\\
&&\qquad{}+2\mathbb{E} {\int_{t}^{n}}
\mathbh{1}_{ [  0,n\wedge
\tau ]  } ( r ) \mathrm{e}^{2\tilde{V}_{r}}\bigl\llvert
Y_{r}^{n}\bigr\rrvert ^{2}%
\bigl[
\bigl\llvert \mu_{r}^{-} \bigr\rrvert ^{2}\,
\mathrm{d}r+ \bigl\llvert \nu_{r}%
^{-} \bigr
\rrvert ^{2}\,\mathrm{d}A_{r} \bigr].
\end{eqnarray*}
The right hand side in the above inequality is bounded since
\begin{eqnarray*}
\mathrm{e}^{2\tilde{V}_{r}}\bigl\llvert G
\bigl(r,Y_{r}^{n} \bigr)\bigr\rrvert ^{2}&\leq&
\sup_{\llvert  y\rrvert
\leq\sqrt{R_{0}^{\prime}}}\mathrm{e}^{2\tilde{V}_{r}}\bigl\llvert G \bigl(r,
\mathrm{e}^{-\tilde{V}_{r}}%
y \bigr)\bigr\rrvert ^{2},
\\
\mathrm{e}^{2\tilde{V}_{r}}\bigl\llvert F \bigl(r,Y_{r}^{n},Z_{r}^{n}
\bigr)\bigr\rrvert ^{2}&\leq&3\sup_{\llvert
y\rrvert  \leq\sqrt{R_{0}^{\prime}}}
\mathrm{e}^{2\tilde{V}_{r}}\bigl\llvert F \bigl(r,\mathrm{e}^{-\tilde
{V}_{r}}y,0
\bigr)\bigr\rrvert ^{2}+3\ell^{2}\mathrm{e}^{2\tilde{V}_{r}}
\bigl\llvert Z_{r}^{n}- \zeta_{r}\bigr\rrvert
^{2}%
+3\ell^{2}\mathrm{e}^{2\tilde{V}_{r}}\llvert
\zeta_{r}\rrvert ^{2}.
\end{eqnarray*}
Therefore
%
\begin{equation}
\label{boundedness for approxim 9} \mathbb{E} \bigl[ \mathrm{e}^{2\tilde{V}_{t}} \bigl(
\varphi_{1/n} \bigl(Y_{t}^{n}%
\bigr)+
\psi_{1/n} \bigl(Y_{t}^{n} \bigr) \bigr) \bigr]
\leq C, \qquad\mbox{for all }t \geq0 %
\end{equation}
and
%
\begin{equation}
\label{boundedness for approxim 10} \mathbb{E}\int_{0}^{{\infty}}
\mathbh{1}_{ [
0,n\wedge\tau ]  } ( r ) \bigl[\mathrm{e}^{2\tilde{V}_{r}}\bigl\llvert
\nabla \varphi_{1/n} \bigl(Y_{r}^{n} \bigr)\bigr
\rrvert ^{2}\,\mathrm{d}r+\mathrm{e}^{2\tilde{V}_{r}}\bigl\llvert \nabla
\psi_{1/n} \bigl(Y_{r}%
^{n} \bigr)
\bigr\rrvert ^{2}\,\mathrm{d}A_{r} \bigr]\leq C. %
\end{equation}
From (\ref{boundedness for approxim 9}) and (\ref{ineq Yosida})(a) we see
that, for all $t\geq0$,
%
\begin{equation}
\label{boundedness for approxim 11} \mathbb{E} \bigl[ \mathrm{e}^{2\tilde{V}_{t}} \bigl( \bigl\llvert
1/n\nabla\varphi _{1/n} \bigl(Y_{r}^{n} \bigr)
\bigr\rrvert ^{2}+ \bigl\llvert 1/n\nabla\psi_{1/n}
\bigl(Y_{r}%
^{n} \bigr) \bigr\rrvert
^{2} \bigr) \bigr] \leq2C/n %
\end{equation}
and
%
\begin{equation}
\label{boundedness for approxim 11'} \mathbb{E} \bigl[ \mathrm{e}^{2\tilde{V}_{t}} \bigl( \varphi \bigl(
Y_{t}^{n}%
-1/n\nabla\varphi_{1/n}
\bigl(Y_{r}^{n} \bigr) \bigr) +\psi \bigl(
Y_{t}^{n}%
-1/n\nabla \psi_{1/n}
\bigl(Y_{r}^{n} \bigr) \bigr) \bigr) \bigr] \leq C.
\end{equation}

E. \textit{Cauchy sequences and convergence.}
From (\ref{boundedness for approxim 4}), we have
%
\begin{eqnarray}
\label{boundedness for approxim 12} %
&&\mathbb{E}\sup_{s\geq n}
\mathrm{e}^{p\tilde{V}_{s}}\bigl\llvert Y_{s}^{n+l}-
\xi_{s}\bigr\rrvert ^{p}%
+\mathbb{E} \biggl( {
\int_{n}^{\infty}}\mathrm{e}^{2\tilde{V}_{s}}%
\bigl\llvert Z_{s}^{n+l}-\zeta_{s}\bigr\rrvert
^{2}\,\mathrm{d}s \biggr) ^{p/2}\nonumber
\\
&&\quad\leq C \mathbb{E} \biggl[ \biggl({\int_{n}%
^{\infty}}\mathrm{e}^{2\tilde{V}_{s}}\Psi(s,\xi_{s})\,
\mathrm{d}Q_{s} \biggr)^{p/2}\\
&&\qquad{}
+ \biggl({\int
_{n}^{\infty}}\mathrm{e}^{\tilde{V}_{s}}\bigl\llvert
\Phi(s,\xi_{s}%
,\zeta_{s})\bigr\rrvert \,
\mathrm{d}Q_{s} \biggr)^{p} \biggr]\rightarrow 0,\qquad n
\rightarrow {\infty}.
\nonumber
\end{eqnarray}
By uniqueness it follows that, for all $t\in [  0,n ]  $,
\[
Y_{t}^{n+l}-Y_{t}^{n}=Y_{n}^{n+l}-
\xi_{n}+\int_{t}^{n}\,
\mathrm{d}K_{s}^{n,l}- \int_{t}^{n}
\bigl(Z_{s}^{n+l}-Z_{s}^{n} \bigr)\,
\mathrm{d}W_{s},\qquad \mathrm{a.s.},
\]
where $\mathrm{d}K_{s}^{n,l}= (  \Phi(s,Y_{s}^{n+l},Z_{s}^{n+l})-\Phi(s,Y_{s}%
^{n},Z_{s}^{n})-[\nabla_{y}\Psi^{n+l}(Y_{s}^{n+l})-\nabla_{y}\Psi^{n}%
(Y_{s}^{n})] )\,\mathrm{d}Q_{s}$.

By (\ref{ineq Yosida})(d) with $\varepsilon=1/ (  n+l )  $ and
$\delta=1/n$
\begin{eqnarray*}
&&{-} \bigl\langle Y_{s}^{n+l}-Y_{s}^{n},
\bigl(\nabla_{y}\Psi^{n+l} \bigl(s,Y_{s}%
^{n+l} \bigr)-\nabla_{y}\Psi^{n}
\bigl(s,Y_{s}^{n} \bigr) \bigr)\,\mathrm{d}Q_{s}
\bigr\rangle
\\
&&\quad\leq(\varepsilon+\delta)\mathbh{1}_{ [  0,\tau ]  } ( s ) \bigl( \bigl
\langle \nabla\varphi_{\varepsilon} \bigl(Y_{s}^{n+l} \bigr),
\nabla \varphi _{\delta} \bigl(Y_{s}^{n} \bigr) \bigr
\rangle \,\mathrm{d}s+ \bigl\langle\nabla\psi_{\varepsilon}%
\bigl(Y_{s}^{n+l}
\bigr),\nabla\psi_{\delta} \bigl(Y_{s}^{n} \bigr)
\bigr\rangle\,\mathrm{d}A_{s} \bigr),
\end{eqnarray*}
and using (\ref{a priori estim 2}) we have on $ [  0,n ]  $
\begin{eqnarray*}
&& \bigl\langle Y_{s}^{n+l}-Y_{s}^{n},\mathrm{d}K_{s}^{n,l}
\bigr\rangle
\\
&&\quad\leq\frac
{\varepsilon+\delta}{2}\mathbh{1}_{ [  0,\tau ]  } ( s ) \bigl[ \bigl(\bigl
\llvert \nabla\varphi_{\varepsilon} \bigl(Y_{s}^{n} \bigr)
\bigr\rrvert ^{2}+\bigl\llvert \nabla\varphi_{\delta
}
\bigl(Y_{s}^{n+l} \bigr)\bigr\rrvert ^{2} \bigr)\,
\mathrm{d}s
\\
&&\qquad{}+ \bigl(\bigl\llvert \nabla\psi_{\varepsilon} \bigl(Y_{s}^{n}
\bigr)\bigr\rrvert ^{2}+\bigl\llvert \nabla\psi_{\delta
}
\bigl(Y_{s}^{n+l} \bigr)\bigr\rrvert ^{2} \bigr)\,
\mathrm{d}A_{s} \bigr]+\bigl\llvert Y_{s}^{n+l}-Y_{s}^{n}
\bigr\rrvert ^{2}\,\mathrm{d} \tilde{V}%
_{s}^{n}+
\frac{1}{2a}\bigl\llvert Z_{s}^{n+l}-Z_{s}^{n}
\bigr\rrvert ^{2}\,\mathrm{d}s.
\end{eqnarray*}
Proposition \ref{prop 1 Appendix} yields once again
\begin{eqnarray*}
&&\mathbb{E}\sup_{s\in [  0,n ]  }
\mathrm{e}^{2\tilde{V}_{s}} \bigl\llvert Y_{s}^{n+l}-Y_{s}^{n}
\bigr\rrvert ^{2}+\mathbb{E} {\int_{0}^{n}%
}\mathrm{e}^{2\tilde{V}_{s}}\bigl\llvert Z_{s}^{n+l}-Z_{s}^{n}
\bigr\rrvert ^{2}\,\mathrm{d}s
\\
&&\quad\leq C \mathbb{E}\mathrm{e}^{2\tilde{V}_{n}}\bigl\llvert
Y_{n}^{n+l}- \xi_{n}\bigr\rrvert ^{2}+(
\varepsilon +\delta) C \mathbb{E} {\int_{0}^{n\wedge\tau}}
\mathrm{e}^{2\tilde{V}_{s}%
} \bigl(\bigl\llvert \nabla \varphi_{\varepsilon}
\bigl(Y_{s}^{n} \bigr)\bigr\rrvert ^{2}+\bigl
\llvert \nabla\varphi_{\delta
} \bigl(Y_{s}^{n+l} \bigr)
\bigr\rrvert ^{2} \bigr)\,\mathrm{d}s
\\
&&\qquad{}+(\varepsilon+\delta) C \mathbb{E} {\int_{0}^{n\wedge\tau}%
}\mathrm{e}^{2\tilde{V}_{s}} \bigl(\bigl\llvert \nabla\psi_{\varepsilon}
\bigl(Y_{s}^{n} \bigr)\bigr\rrvert ^{2}+\bigl
\llvert \nabla \psi_{\delta} \bigl(Y_{s}^{n+l} \bigr)
\bigr\rrvert ^{2} \bigr)\,\mathrm{d}A_{s}.
\end{eqnarray*}
The estimates (\ref{boundedness for approxim 10}) and
(\ref{boundedness for approxim 12}) give us, for $n\rightarrow\infty$,
\begin{eqnarray*}
&&\mathbb{E}\sup_{s\in [  0,n ]  }\mathrm{e}^{2\tilde{V}_{s}} \bigl\llvert
Y_{s}^{n+l}-Y_{s}^{n} \bigr\rrvert
^{2}+\mathbb{E} {\int_{0}^{n}%
}\mathrm{e}^{2\tilde{V}_{s}}\bigl\llvert Z_{s}^{n+l}-Z_{s}^{n}
\bigr\rrvert ^{2}\,\mathrm{d}s \\
&&\quad\leq\mathbb{E}\sup_{s\geq
n}
\mathrm{e}^{2\tilde{V}_{s}}\bigl\llvert Y_{s}^{n+l}-
\xi_{s}\bigr\rrvert ^{2}+\frac{C}{n}%
\rightarrow0.
\end{eqnarray*}
Hence, for $n\rightarrow\infty$,
\[
\mathbb{E}\sup_{s\geq0}\mathrm{e}^{2\tilde{V}_{s}}\bigl\llvert
Y_{s}^{n+l}-Y_{s}^{n}\bigr\rrvert
^{2}%
\leq\mathbb{E}\sup_{s\in [  0,n ]  }
\mathrm{e}^{2\tilde{V}_{s}} \bigl\llvert Y_{s}^{n+l}-Y_{s}^{n}
\bigr\rrvert ^{2}+\mathbb{E}\sup_{s\geq n}
\mathrm{e}^{2\tilde
{V}_{s}}\bigl\llvert Y_{s}^{n+l}-
\xi_{s}\bigr\rrvert ^{2}\rightarrow 0
\]
and
\[
\mathbb{E} {\int_{0}^{{\infty}}}\mathrm{e}^{2\tilde{V}_{s}}
\bigl\llvert Z_{s}^{n+l}-Z_{s}%
^{n}\bigr\rrvert ^{2}\,\mathrm{d}s\leq\mathbb{E} {\int
_{0}^{n}}\mathrm{e}^{2\tilde{V}_{s}}\bigl\llvert
Z_{s}^{n+l}%
-Z_{s}^{n}\bigr
\rrvert ^{2}\,\mathrm{d}s+\mathbb{E} {\int_{n}^{{\infty}}}
\mathrm{e}^{2\tilde{V}_{s}%
}\bigl\llvert Z_{s}^{n+l}-
\zeta_{s}\bigr\rrvert ^{2}\,\mathrm{d}s\rightarrow 0.
\]

F. \textit{Passage to the limit.}
Consequently there exists $(Y,Z)\in\mathcal{S}^{2}\times\Lambda^{2}$ such that
\[
\mathbb{E}\sup_{s\geq0}\mathrm{e}^{2\tilde{V}_{s}}\bigl\llvert
Y_{s}^{n}-Y_{s}%
\bigr\rrvert
^{2}+\mathbb{E}\int_{0}^{\infty}
\mathrm{e}^{2\tilde{V}_{s}}\bigl\llvert Z_{s}^{n}-Z_{s}%
\bigr\rrvert ^{2}\,\mathrm{d}s\rightarrow0, \qquad\mbox{as }n
\rightarrow\infty.
\]
We have $(Y_{t},Z_{t})= (  \eta,0 )  $ for $t>\tau$, since $Y_{t}%
^{n}=\xi_{t}=\eta$ and $Z_{t}^{n}=\zeta_{t}=0$ for $t>\tau$.

Applying Fatou's lemma to (\ref{boundedness for approxim}) and
(\ref{boundedness for approxim 4}), we obtain (\ref{ineq of sol 3})(a), (b) and
taking the limit along a subsequence in (\ref{boundedness for approxim 2}), we
deduce that $\mathrm{e}^{2\tilde{V}_{t}}\llvert Y_{t}\rrvert ^{2}\leq R_{0}^{\prime}, \mathbb{P}%
$-a.s., for all $t\geq0$.

From (\ref{boundedness for approxim 10}), there exist two p.m.s.p. $U^{1}$ and
$U^{2}$, such that along a subsequence still indexed by $n$, $\mathbh{1}%
_{ [  0,\tau\wedge n ]  }\mathrm{e}^{2\tilde{V}}\nabla\varphi_{1/n}%
(Y^{n})\rightharpoonup\mathbh{1}_{ [  0,\tau ]  }U^{1}$, weakly in
$L^{2} (  \Omega\times\mathbb{R}_{+},\mathrm{d}\mathbb{P}\otimes \mathrm{d}t;H )  $ and
$\mathbh{1}_{ [  0,\tau\wedge n ]  }\mathrm{e}^{2\tilde{V}}\nabla\psi
_{1/n}(Y^{n})\rightharpoonup\mathbh{1}_{ [  0,\tau ]  }U^{2}$, weakly
in $L^{2} (  \Omega\times\mathbb{R}_{+},\mathrm{d}\mathbb{P}\otimes \mathrm{d}A_{t}%
;H )  $.

Applying Fatou's lemma to (\ref{boundedness for approxim 10}), we obtain
(\ref{ineq of sol 4}) and from (\ref{boundedness for approxim 11}) we deduce
that for all $t\geq0$ fixed, there exists a subsequence indexed also by $n$,
such that
\[
\frac{1}{n}\nabla\varphi_{1/n} \bigl(Y_{t}^{n}
\bigr)\xrightarrow{} {\mbox{{\fontsize{7.6}{9.6}\selectfont{a.s.}}}}0 \quad\mbox{and}\quad\frac{1}{n}\nabla
\psi_{1/n} \bigl(Y_{t}^{n}%
\bigr)
\xrightarrow{} {\mbox{{\fontsize{7.6}{9.6}\selectfont{a.s.}}}}0.
\]
We now apply Fatou's lemma to (\ref{boundedness for approxim 11'}) and we
conclude (\ref{ineq of sol 3})(d).

From (\ref{approxim eq 1}), we have for all $0\leq t\leq T\leq n$, $\mathbb{P}%
$-a.s.
\[
Y_{t}^{n}+\int_{t}^{T}
\nabla_{y}\Psi^{n} \bigl(s,Y_{s}^{n}
\bigr)\,\mathrm{d}Q_{s}=Y_{T}^{n}+\int
_{t}^{T} \Phi \bigl(s,Y_{s}^{n},Z_{s}^{n}
\bigr)\,\mathrm{d}Q_{s}-\int_{t}^{T}Z_{s}^{n}
\,\mathrm{d}W_{s},
\]
and passing to the limit we conclude that
\[
Y_{t}+\int_{t}^{T}U_{s}\,
\mathrm{d}Q_{s}=Y_{T}+ \int_{t}^{T}
\Phi(s,Y_{s},Z_{s})\,\mathrm{d}Q_{s}%
-
\int_{t}^{T}Z_{s}\,\mathrm{d}W_{s},
\qquad \mbox{a.s.}%
\]
with $U_{s}=\mathbh{1}_{ [  0,\tau ]  } (  s )   [
\alpha_{s}U_{s}^{1}+ (  1-\alpha_{s} )  U_{s}^{2} ]  $, for
$s\geq0$.

Since (\ref{ineq Yosida})(b), we see that, for all $E\in\mathcal{F}$, $0\leq
s\leq t$ and $X\in\mathcal{S}^{2} [  0,T ]  $,
\begin{eqnarray*}
&&\mathbb{E} {\int_{s}^{t}}\mathbh{1}_{E}
\bigl( \bigl\langle \mathrm{e}^{2\tilde{V}_{r}%
}\nabla\varphi_{1/n}
\bigl(Y_{r}^{n} \bigr),X_{r}-Y_{r}^{n}
\bigr\rangle+\mathrm{e}^{2\tilde{V}_{r}%
}\varphi \bigl(Y_{r}^{n}-1/n
\nabla\varphi_{1/n} \bigl(Y_{s}^{n} \bigr) \bigr)
\bigr)\,\mathrm{d}r\\
&&\quad\leq\mathbb{E}%
\int_{s}^{t}
\mathbh{1}_{E}\mathrm{e}^{2\tilde{V}_{r}}\varphi(X_{r})\,
\mathrm{d}r.
\end{eqnarray*}
Passing to $\liminf$ for $n\rightarrow{\infty}$ in the above inequality
we obtain $U_{s}^{1}\in\partial\varphi(Y_{s})$, $\mathrm{d}\mathbb{P}\otimes \mathrm{d}s$-a.e. and, with similar arguments, $U_{s}^{2}\in\partial\psi(Y_{s}%
),\mathrm{d}\mathbb{P}\otimes \mathrm{d}A_{s}$-a.e.

Summarizing the above conclusions we see that $(Y,Z,U)$ is solution of the
BSVI (\ref{GBSVI 4}) under assumptions {(A$_{1}$)--(A$_{9}$)}.
\end{pf*}

\section{Variational weak formulation}\label{weak form}

In this section, we generalize the notion of solution for (\ref{GBSVI 1}), or
(\ref{GBSVI 4}), in order to give up to the assumption {(A$_{9}$)}.
The existence and the uniqueness of a weak solution $ (  Y,Z )  $
will be given. We mention that without {(A$_{9}$)} we cannot prove the
existence of a process $K$ such that $\mathrm{d}K_{t}=U_{t}\,\mathrm{d}Q_{t}\in\partial_{y}%
\Psi (  t,Y_{t} )\,\mathrm{d}Q_{t}$ (see Remarks \ref{remark_(A9)} and
\ref{remark_(A9)'}); more precisely we cannot obtain the boundedness in
$L^{2}$ of the gradients, see (\ref{boundedness for approxim 10}), and
respectively, the existence of a process $U$ such that $K_{t}=\int_{0}^{t}%
U_{s}\,\mathrm{d}Q_{s}$. Therefore, we shall give the definition of a weak solution of the
BSVI (\ref{GBSVI 4}).

Let us define the space $\mathcal{M}$ of the semimartingales $M\in
\mathcal{S}^{1}$ of the form
\[
M_{t}=\gamma-{\int_{0}^{t}}%
N_{r}\,\mathrm{d}Q_{r}+{\int_{0}^{t}}R_{r}
\,\mathrm{d}W_{r},
\]
where $N$ and $R$ are two
p.m.s.p. such that
\[
{\int_{0}^{T}}\llvert N_{r}\rrvert\,\mathrm{d}Q_{r}+{\int_{0}^{T}}\llvert
R_{r}\rrvert ^{2}\,\mathrm{d}r<{\infty}\qquad \mbox{a.s., }
\forall T>0 \mbox{ and } \gamma\in L^{0} (\Omega;\mathcal{F}_{0},
\mathbb{P};H ) .
\]

For a intuitive introduction, let $M\in\mathcal{M}$ and $ (  Y,Z )  $
be strong a solution of (\ref{GBSVI 4}), in the sense of Definition
\ref{definition sol GBSVI}. By It\^{o}'s formula, we deduce inequality
%
\begin{eqnarray}
\label{definition sol 4'} 
&&{{\frac{1}{2}}}\llvert M_{t}-Y_{t}
\rrvert ^{2}+{{\frac{1}{2}%
}\int_{t}^{T}}
\llvert R_{r}-Z_{r} \rrvert ^{2}\,\mathrm{d}r+{
\int_{t}^{T} \Psi} ( r,Y_{r} )\,
\mathrm{d}Q_{r}
\nonumber
\\
&&\quad\leq {{\frac{1}{2}}}\llvert M_{T}-Y_{T}\rrvert
^{2}+{\int_{t}^{T} \Psi} (
r,M_{r} )\,\mathrm{d}Q_{r}
\\
&&\qquad{}+{\int_{t}^{T}} \bigl\langle
M_{r}-Y_{r},N_{r}- \Phi (
r,Y_{r},Z_{r} ) \bigr\rangle\,\mathrm{d}Q_{r}-{ \int
_{t}^{T}} \bigl\langle M_{r}-Y_{r},
( R_{r}-Z_{r} )\,\mathrm{d}W_{r} \bigr
\rangle,
\nonumber
\end{eqnarray}
since $\mathrm{d}K_{t}=U_{t}\,\mathrm{d}Q_{t}\in\partial_{y}\Psi (  t,Y_{t} )\,\mathrm{d}Q_{t}$
(see inequality (\ref{in the subdiff})).

Following the approach for the forward stochastic variational inequalities
from article R\u{a}\c{s}canu \cite{ra/81}, we propose the next weak formulation of Definition
\ref{definition sol GBSVI}:

\begin{definition}\label{definition sol GBSVI 2}
We call $ (  Y_{t},Z_{t} )  _{t\geq0}$
a variational weak solution of (\ref{GBSVI 4}) if $ (  Y,Z )
\in\mathcal{S}^{2}\times\Lambda^{2}$, $ (  Y_{t},Z_{t} )  = (
\xi_{t},\zeta_{t} )  = (  \eta,0 )  $ for $t>\tau$ and
%
\begin{equation}
\label{definition sol 4} \everymath{\displaystyle} %
\begin{array} [c]{r@{
\quad}l}%
\mbox{\emph{(i)}} & {\int_{0}^{T}}
\bigl( \bigl\llvert \Phi ( s,Y_{s},Z_{s} ) \bigr\rrvert +
\bigl\llvert \Psi ( s,Y_{s} ) \bigr\rrvert \bigr)\,
\mathrm{d}Q_{s}< \infty, \qquad\mathbb{P}\mbox{-a.s., for all
}%
T\geq0,
\\\noalign{\vspace*{6pt}}
\mbox{\emph{(ii)}} & {{\frac{1}{2}}}\llvert M_{t}%
-Y_{t}\rrvert ^{2}+{{\frac{1}{2}}\int
_{t}^{s}}\llvert R_{r}-Z_{r}
\rrvert ^{2}\,\mathrm{d}r+{\int_{t}^{s}
\Psi} ( r,Y_{r} )\,\mathrm{d}Q_{r}
\\\noalign{\vspace*{6pt}}
&\quad\leq{{\frac{1}{2}}}\llvert M_{s}-Y_{s}\rrvert
^{2}
\\\noalign{\vspace*{6pt}}
& \qquad{} +{\int_{t}^{s}\Psi} (
r,M_{r} )\,\mathrm{d}Q_{r}+{\int_{t}^{s}}
\bigl\langle M_{r}-Y_{r}%
,N_{r}-
\Phi ( r,Y_{r},Z_{r} ) \bigr\rangle\,\mathrm{d}Q_{r}
\\\noalign{\vspace*{6pt}}
& \qquad{}-{\int_{t}^{s}} \bigl\langle
M_{r}-Y_{r}, ( R_{r}-Z_{r} )\,
\mathrm{d}W_{r} \bigr\rangle ,
\\\noalign{\vspace*{6pt}}
& \qquad\forall 0\leq t\leq s, \forall ( N,R ) \in L^{2} \bigl( \Omega
\times[0,{\infty});H \bigr) \times\Lambda ^{2}, \forall M\in
\mathcal{M},
\\\noalign{\vspace*{6pt}}
\mbox{\emph{(iii)}} & \mathbb{E}\mathrm{e}^{2V_{T}}\llvert
Y_{T}-\xi_{T} \rrvert ^{2}+\mathbb{E} {\int
_{T}^{\infty}}\mathrm{e}^{2V_{s}} \llvert
Z_{s}-\zeta_{s}\rrvert ^{2}\,\mathrm{d}s
\rightarrow 0, \qquad\mbox{as }T \rightarrow\infty. \end{array}
\end{equation}
\end{definition}

\begin{remark}
It is obviously that a strong solution for (\ref{GBSVI 4}) is also a weak
solution (see the intuitive introduction for inequality
(\ref{definition sol 4'})).
\end{remark}

\begin{remark}
We highlight the connection between this definition and the Fitzpatrick
function approach for the multivalued BSDEs driven by a maximal monotone
operator or, in particular, by a subdifferential operator (see R\u{a}\c{s}canu and Rotenstein \cite{ra-ro/11}).
\end{remark}

\begin{theorem}\label{main result 2}
Let assumptions \emph{(A$_{1}$)--(A$_{8}$)} be satisfied.
Then the backward stochastic variational inequality (\ref{GBSVI 4}) has a
unique solution $ (  Y,Z )  $ in the sense of Definition
\ref{definition sol GBSVI 2} such that $\mathbb{E}\sup_{s\in [
0,T ]  }\mathrm{e}^{p\tilde{V}_{s}}\llvert  Y_{s}\rrvert  ^{p}<\infty$, for
all $T\geq0$. Moreover, inequalities (\ref{ineq of sol 3}) hold.
\end{theorem}

\begin{pf}
First, we shall approximate the data $\eta$ and $\Phi$ by $\eta^{n}$, respectively $\Phi^{n}$ which satisfy (\ref{Assumption A_10}).

Let
\begin{eqnarray*}
\eta^{n} ( \omega ) &=&\eta ( \omega )
\mathbh{1}_{ [
0,n ]  } \Bigl( \bigl\llvert \eta ( \omega ) \bigr\rrvert +\Bigl
\llvert \sup_{s\in [  0,\tau ]  }\tilde{V}_{s}\Bigr\rrvert \Bigr),
\\
\Phi^{n} ( t,y,z ) &=&\Phi ( t,y,z ) -\Phi ( t,0,0 ) +\Phi ( t,0,0 )
\mathbh{1}_{ [  0,n ]
} \bigl(t+ \bigl\llvert \Phi ( t,0,0 ) \bigr\rrvert +
\llvert \tilde{V}%
_{t}\rrvert \bigr).
\end{eqnarray*}
Obviously, as $n\rightarrow{\infty}$,
\[
\mathbb{E} \bigl(\mathrm{e}^{p\sup_{s\in [  0,\tau ]  }\tilde{V}_{s}} \bigl\llvert \eta^{n}-\eta
\bigr\rrvert ^{p} \bigr)+\mathbb{E} \biggl( {\int_{0}^{\tau}}
\mathrm{e}^{\tilde{V}_{s}} \bigl\llvert \Phi^{n} ( s,0,0 ) -\Phi ( s,0,0
) \bigr\rrvert\,\mathrm{d}Q_{s} \biggr) ^{p}%
\rightarrow
0.
\]
Theorem \ref{main result 1} shows that there exists a unique solution $ (
Y^{n},Z^{n},U^{n} )  $ of the BSVI (\ref{GBSVI 4}) corresponding to
$\eta^{n}$ and $\Phi^{n}$:
\[
\cases{\displaystyle Y_{t}^{n}+{\int_{t}^{\infty}}U_{s}^{n}
\,\mathrm{d}Q_{s}= \eta^{n}%
+{\int
_{t}^{\infty}} \Phi^{n} \bigl(
s,Y_{s}^{n},Z_{s}^{n} \bigr)\,
\mathrm{d}Q_{s}-{\int_{t}^{\infty}}Z_{s}^{n}
\,\mathrm{d}W_{s},\qquad \mbox{a.s.,}%
\cr
\displaystyle{U_{t}^{n}\in\partial_{y} \Psi
\bigl( t,Y_{t}^{n} \bigr) , \qquad\forall t
\geq0.}%
}
\]
This solution satisfies inequalities (\ref{ineq of sol 3}) and (\ref{ineq of sol 4}%
) with $Y,Z,U,\Phi,\eta,\xi,\zeta$ replaced respectively, with $Y^{n}$, $ Z^{n}$, $U^{n}$, $\Phi^{n}$, $\eta^{n}$, $\xi^{n}$,
$ \zeta^{n}$.

Since $\llvert  \eta^{n}\rrvert  \leq\llvert  \eta\rrvert  $ and
$\llvert  \Phi^{n} (  t,0,0 )  \rrvert  \leq\llvert
\Phi (  t,0,0 )  \rrvert  $,
\[
\mathrm{e}^{2\tilde{V}_{t}}\bigl\llvert Y_{t}^{n}\bigr
\rrvert ^{2}\leq C \mathbb{E}^{\mathcal{F}_{t}%
} \biggl[ \bigl(
\mathrm{e}^{2\sup_{s\in [  t,\tau ]  }\tilde{V}_{s}}\llvert \eta \rrvert ^{2} \bigr)+ \biggl({\int
_{t}^{{\infty}}}\mathrm{e}^{\tilde{V}_{s}}%
\bigl\llvert \Phi ( s,0,0 ) \bigr\rrvert \,\mathrm{d}Q_{s}
\biggr)^{2} \biggr],\qquad \forall t\geq0,\ \mathbb{P}%
\mbox{-a.s.}%
\]
Using (\ref{a priori estim 2}), we see that
\begin{eqnarray*}
&& \bigl\langle Y_{s}^{n}-Y_{s}^{m},
\Phi^{n} \bigl( s,Y_{s}^{n},Z_{s}^{n}
\bigr) -\Phi^{m} \bigl(s,Y_{s}^{m},Z_{s}^{m}
\bigr)- \bigl( U_{s}^{n}-U_{s}^{m}
\bigr) \bigr\rangle\,\mathrm{d}Q_{s}
\\
&&\quad\leq \bigl\langle Y_{s}^{n}-Y_{s}^{m},
\Phi \bigl( s,Y_{s}^{n},Z_{s}^{n}
\bigr) -\Phi \bigl(s,Y_{s}^{m},Z_{s}^{m}
\bigr) \bigr\rangle\,\mathrm{d}Q_{s}
\\
&&\qquad{}+ \bigl\langle Y_{s}^{n}-Y_{s}^{m},
\Phi ( s,0,0 ) \bigr\rangle ( \mathbh{1}_{ [  0,n ]  }-\mathbh{1}_{ [  0,m ]  }
) \bigl(t+ \bigl\llvert \Phi ( t,0,0 ) \bigr\rrvert +\llvert
\tilde{V}_{t}%
\rrvert \bigr)\,\mathrm{d}Q_{s}
\\
&&\quad\leq\bigl\llvert Y_{s}^{n}-Y_{s}^{m}
\bigr\rrvert \bigl\llvert \Phi ( s,0,0 ) \bigr\rrvert \bigl\llvert (
\mathbh{1}_{ [  0,n ]  }- \mathbh{1}_{ [
0,m ]  } ) \bigl(t+ \bigl\llvert
\Phi ( t,0,0 ) \bigr\rrvert +\llvert \tilde{V}_{t}\rrvert \bigr)
\bigr\rrvert \,\mathrm{d}Q_{s}
\\
&&\qquad{}+\bigl\llvert Y_{s}^{n}-Y_{s}^{m}
\bigr\rrvert ^{2}\,\mathrm{d} \tilde{V}_{s}+\frac{1}{2a}
\bigl\llvert Z_{s}^{n}%
-Z_{s}^{m}
\bigr\rrvert ^{2}\,\mathrm{d}s,
\end{eqnarray*}
since $ \langle Y_{s}^{n}-Y_{s}^{m},U_{s}^{n}-U_{s}^{m} \rangle
\geq0$, for $ U_{s}^{n}\in\partial_{y}\Psi(s,Y_{s}^{n})$ and $U_{s}^{m}%
\in\partial_{y}\Psi(s,Y_{s}^{m})$, and $\mathrm{d}V_{s}\leq \mathrm{d}\tilde{V}_{s}$ on $ [
0,\tau ]  $.

Applying Proposition \ref{prop 1 Appendix} (see the \hyperref[app]{Appendix}) for the equation
satisfied by $Y^{n}-Y^{m}$ on $ [  0,T ]  $, it follows that
\begin{eqnarray*}
&&\mathbb{E}\sup_{s\in [  0,T ]  }
\mathrm{e}^{p\tilde{V}_{s}} \bigl\llvert Y_{s}^{n}-Y_{s}^{m}
\bigr\rrvert ^{p}+\mathbb{E} \biggl( {\int_{0}^{T}}
\mathrm{e}^{2\tilde{V}_{s}}\bigl\llvert Z_{s}^{n}-Z_{s}^{m}
\bigr\rrvert ^{2}\,\mathrm{d}s \biggr) ^{p/2}%
\\
&&\quad{\leq C \mathbb{E} \biggl({\int_{0}^{T}}%
\mathrm{e}^{\tilde{V}_{s}} \bigl\llvert \Phi ( s,0,0 ) \bigr\rrvert \bigl\llvert (
\mathbh{1}_{ [  0,n ]  }-\mathbh{1}_{ [
0,m ]  } ) \bigl(t+ \bigl\llvert
\Phi ( t,0,0 ) \bigr\rrvert +\llvert \tilde{V}_{t}\rrvert \bigr)
\bigr\rrvert \,\mathrm{d}Q_{s} \biggr)^{p} }
\\
&&\qquad{} +C \mathbb{E}\mathrm{e}^{p\tilde{V}_{T}} \bigl(\bigl\llvert
Y_{T}^{n}- \xi_{T}%
^{n}\bigr
\rrvert ^{p}+\bigl\llvert \xi_{T}^{n}-
\xi_{T}^{m}\bigr\rrvert ^{p}+\bigl\llvert
\xi_{T}^{m}-Y_{T}^{m}\bigr\rrvert
^{p} \bigr).
\end{eqnarray*}
Therefore, passing to the limit for $T\rightarrow{\infty}$,
\begin{eqnarray*}
&&\mathbb{E}\sup_{s\geq0}
\mathrm{e}^{p\tilde{V}_{s}}\bigl\llvert Y_{s}^{n}-Y_{s}^{m}
\bigr\rrvert ^{p}%
+\mathbb{E} \biggl( {\int
_{0}^{{\infty}}}\mathrm{e}^{2\tilde{V}_{s}%
}\bigl\llvert
Z_{s}^{n}-Z_{s}^{m}\bigr\rrvert
^{2}\,\mathrm{d}s \biggr) ^{p/2}
\\
&&\quad {\leq C \mathbb{E} \biggl({\int_{0}%
^{{\infty}}}\mathrm{e}^{\tilde{V}_{s}} \bigl\llvert \Phi ( s,0,0 ) \bigr
\rrvert \bigl\llvert ( \mathbh{1}_{ [  0,n ]  }-\mathbh{1}%
_{ [  0,m ]  } ) \bigl(t+ \bigl\llvert \Phi ( t,0,0 ) \bigr\rrvert +
\llvert \tilde{V}_{t}\rrvert \bigr) \bigr\rrvert \,\mathrm{d}Q_{s}
\biggr)^{p} }
\\
&&\qquad{} +C \mathbb{E}\mathrm{e}^{p\sup_{s\geq0}\tilde{V}_{s}}\bigl\llvert
\eta^{n}%
-\eta^{m}\bigr\rrvert ^{p}
\xrightarrow{n,m\rightarrow \infty} {}0.%
\end{eqnarray*}
Consequently there exists $(Y,Z)\in\mathcal{S}^{0}\times\Lambda^{0}$ a
solution of the BSVI (\ref{GBSVI 4}) such that
\[
\mathbb{E}\sup_{s\geq0}\mathrm{e}^{2\tilde{V}_{s}}\bigl\llvert
Y_{s}^{n}-Y_{s}%
\bigr\rrvert
^{2}+\mathbb{E}\int_{0}^{\infty}
\mathrm{e}^{2\tilde{V}_{s}}\bigl\llvert Z_{s}^{n}-Z_{s}%
\bigr\rrvert ^{2}\,\mathrm{d}s\rightarrow 0,\qquad \mbox{as }n
\rightarrow\infty
\]
and $(Y_{t},Z_{t})= (  \eta,0 )  $ for $t>\tau$, since $Y_{t}^{n}%
=\xi_{t}^{n}=\eta^{n}$ and $Z_{t}^{n}=\zeta_{t}^{n}=0$ for $t>\tau$.

Let $M\in\mathcal{M}$ given by $M_{t}=\gamma-{\int_{0}^{t}}N_{r}\,\mathrm{d}Q_{r}%
+{\int_{0}^{t}}R_{r}\,\mathrm{d}W_{r} $. From the It\^{o}'s formula applying to
$\llvert  M_{t}-Y_{t}^{n}\rrvert  ^{2}$, we deduce that, for all $0\leq
t\leq s\leq\tau$,
\begin{eqnarray*}
&&{{\frac{1}{2}}} \bigl\llvert
M_{t}-Y_{t}^{n} \bigr\rrvert ^{2}+{{
\frac{1}{2}%
}\int_{t}^{s}} \bigl
\llvert R_{r}-Z_{r}^{n} \bigr\rrvert
^{2}\,\mathrm{d}r+{\int_{t}^{s}\Psi}
\bigl( r,Y_{r}^{n} \bigr)\,\mathrm{d}Q_{r}%
\\
&&\quad\leq{{ \frac{1}{2}}} \bigl\llvert M_{s}-Y_{s}^{n}
\bigr\rrvert ^{2}
\\
&&\qquad{}+{\int_{t}^{s}\Psi} (
r,M_{r} )\,\mathrm{d}Q_{r}%
+{\int
_{t}^{s}} \bigl\langle M_{r}-Y_{r}^{n},N_{r}-
\Phi \bigl( r,Y_{r}^{n},Z_{r}^{n}
\bigr) \bigr\rangle\,\mathrm{d}Q_{r}
\\
&&\qquad{}-{\int_{t}^{s}} \bigl\langle
M_{r}-Y_{r}^{n}, \bigl( R_{r}-Z_{r}%
^{n} \bigr)\,\mathrm{d}W_{r} \bigr\rangle .
\end{eqnarray*}
Since on a subsequence (still denoted by $n$)
\[
\sup_{s\in [  0,T ]  } \bigl\llvert Y_{s}^{n}-Y_{s}
\bigr\rrvert ^{2}+{\int_{0}^{T}} \bigl
\llvert Z_{s}^{n}-Z_{s} \bigr\rrvert
^{2}\,\mathrm{d}s %
\rightarrow 0,\qquad
\mbox{a.s.,}%
\]
it follows easily, passing to the $\lim\inf$, that the couple $ (
Y,Z )  $ satisfies inequality (\ref{definition sol 4})(ii).

In the same manner, inequalities (\ref{ineq of sol 3}) follow now from the
similar properties satisfied by the approximate solution $ (  Y^{n}%
,Z^{n} )  $.

In order to prove the uniqueness of the solution let $ (  Y^{1}%
,Z^{1} )  $ and $ (  Y^{2},Z^{2} )  $ be two solutions of
(\ref{GBSVI 4}) corresponding to $\eta^{1}$ and $\eta^{2}$, respectively. Hence,
\begin{eqnarray*}
&&{{\frac{1}{2}}} \bigl( \bigl\llvert M_{t}-Y_{t}^{1}
\bigr\rrvert ^{2}+ \bigl\llvert M_{t}-Y_{t}^{2}
\bigr\rrvert ^{2} \bigr)+{{\frac{1}{2}}\int_{t}^{s}%
} \bigl( \bigl\llvert R_{r}-Z_{r}^{1} \bigr
\rrvert ^{2}+ \bigl\llvert R_{r}-Z_{r}%
^{2} \bigr\rrvert ^{2} \bigr)\,\mathrm{d}r
\\
&&\qquad{}+{\int_{t}^{s} \bigl({\Psi} \bigl(
r,Y_{r}^{1} \bigr) +{\Psi } \bigl( r,Y_{r}^{2}
\bigr) \bigr)}\,\mathrm{d}Q_{r}
\\
&&\quad\leq{{\frac{1}{2}}} \bigl( \bigl\llvert M_{s}-Y_{s}^{1}
\bigr\rrvert ^{2}+ \bigl\llvert M_{s}-Y_{s}^{2}
\bigr\rrvert ^{2} \bigr)
\\
&&\qquad{}+{\int_{t}^{s}} \bigl( \bigl\langle
M_{r}-Y_{r}^{1},N_{r}%
-
\Phi \bigl( r,Y_{r}^{1},Z_{r}^{1}
\bigr) \bigr\rangle+ \bigl\langle M_{r}-Y_{r}%
^{2},N_{r}-\Phi \bigl( r,Y_{r}^{2},Z_{r}^{2}
\bigr) \bigr\rangle \bigr)\,\mathrm{d}Q_{r}%
\\
&&\qquad{}+2{\int_{t}^{s}\Psi} (
r,M_{r} )\,\mathrm{d}Q_{r}%
-{\int
_{t}^{s}} \bigl\langle M_{r}-Y_{r}^{1},
\bigl( R_{r}-Z_{r}%
^{1} \bigr)\,
\mathrm{d}W_{r} \bigr\rangle
\\
&&\qquad{}-{\int_{t}^{s}} \bigl\langle
M_{r}-Y_{r}^{2}, \bigl( R_{r}%
-Z_{r}^{2} \bigr)\,\mathrm{d}W_{r} \bigr
\rangle, \qquad\forall 0\leq t\leq s, \forall M\in\mathcal{M}.
\end{eqnarray*}
Let $Y=\frac{Y^{1}+Y^{2}}{2}$, $Z=\frac{Z^{1}+Z^{2}}{2}$ and $\Phi (
r )  =\frac{\Phi (  r,Y_{r}^{1},Z_{r}^{1} )  +\Phi (
r,Y_{r}^{2},Z_{r}^{2} )  }{2}$. From the convexity of ${{\Psi}}$ we see
that $2{{\Psi} (  r,Y_{r} )  \leq{\Psi} (  r,Y_{r}^{1} )
+{\Psi} (  r,Y_{r}^{2} )  }$ and, using the identity
\[
2 \biggl\langle\frac{y^{1}+y^{2}}{2},\frac{f^{1}+f^{2}}{2} \biggr\rangle+
\frac{1}%
{2} \bigl\langle y^{1}-y^{2},f^{1}-f^{2}
\bigr\rangle = \bigl\langle y^{1}%
,f^{1} \bigr
\rangle + \bigl\langle y^{2},f^{2} \bigr\rangle ,
\]
we obtain
\begin{eqnarray*}
&& \bigl\langle M_{r}-Y_{r}^{1},N_{r}-
\Phi \bigl( r,Y_{r}^{1},Z_{r}^{1}
\bigr) \bigr\rangle+ \bigl\langle M_{r}-Y_{r}^{2},N_{r}-
\Phi \bigl( r,Y_{r}^{2},Z_{r}%
^{2} \bigr) \bigr\rangle
\\
&&\quad{}-2 \bigl\langle M_{r}-Y_{r},N_{r}-
\Phi ( r ) \bigr\rangle -\frac{1}{2} \bigl\langle Y_{r}^{1}-Y_{r}^{2},
\Phi \bigl( r,Y_{r}^{1},Z_{r}%
^{1} \bigr) -\Phi \bigl( r,Y_{r}^{2},Z_{r}^{2}
\bigr) \bigr\rangle=0
\end{eqnarray*}
and
\[
\bigl\langle M_{r}-Y_{r}^{1},R_{r}-Z_{r}^{1}
\bigr\rangle+ \bigl\langle M_{r}-Y_{r}^{2}%
,R_{r}-Z_{r}^{2} \bigr\rangle-2 \langle
M_{r}-Y_{r},R_{r}-Z_{r} \rangle -
\frac{1}{2} \bigl\langle Y_{r}^{1}-Y_{r}^{2},Z_{r}^{1}-Z_{r}^{2}
\bigr\rangle=0.
\]
Therefore, since $\frac{1}{2} (\llvert m-y^{1}\rrvert ^{2}+\llvert m-y^{2}\rrvert ^{2}%
)= \llvert m-\frac{y^{1}+y^{2}}{2} \rrvert ^{2}+\frac{1}{4}\llvert y^{1}-y^{2}\rrvert ^{2}$,
we have
%
\begin{eqnarray}
\label{uniqueness_Gronwall} %
&& \bigl\llvert Y_{t}^{1}-Y_{t}^{2}
\bigr\rrvert ^{2}+{\int_{t}^{s}%
} \bigl\llvert Z_{r}^{1}-Z_{r}^{2}
\bigr\rrvert ^{2}\,\mathrm{d}r
\nonumber
\\
&&\quad\leq8B_{t,s} ( M ) + \bigl\llvert Y_{s}^{1}-Y_{s}^{2}
\bigr\rrvert ^{2}
\nonumber
\\
&&\qquad{}+{{2}\int_{t}^{s}} \bigl\langle
Y_{r}^{1}-Y_{r}^{2},\Phi \bigl(
r,Y_{r}^{1},Z_{r}^{1} \bigr) -\Phi
\bigl( r,Y_{r}^{2},Z_{r}^{2} \bigr)
\bigr\rangle\,\mathrm{d}Q_{r}
\nonumber
\\
&&\qquad{}-2{\int_{t}^{s}} \bigl\langle
Y_{r}^{1}-Y_{r}^{2}, \bigl(
Z_{r}^{1}-Z_{r}^{2} \bigr)\,
\mathrm{d}W_{r} \bigr\rangle, \qquad\forall 0\leq t\leq s, \forall M
\in\mathcal{M}, %
\end{eqnarray}
where
\begin{eqnarray*}
B_{t,s} ( M ) &=&{{ \frac{1}{2}}}\llvert M_{s}-Y_{s}
\rrvert ^{2}+{\int_{t}^{s}{\Psi} (
r,M_{r} ) }\,\mathrm{d}Q_{r}%
-{\int
_{t}^{s}{ \Psi} ( r,Y_{r} ) }\,
\mathrm{d}Q_{r}-{{\frac
{1}{2}}}\llvert M_{t}-Y_{t}
\rrvert ^{2}
\\
&&{}-{{\frac{1}{2}}\int_{t}^{s}}\llvert
R_{r}-Z_{r}%
\rrvert ^{2}\,
\mathrm{d}r+{\int_{t}^{s}} \bigl\langle
M_{r}-Y_{r},N_{r}%
-\Phi ( r )
\bigr\rangle\,\mathrm{d}Q_{r}
\\
&&{}-{\int_{t}^{s}} \bigl\langle
M_{r}-Y_{r}, ( R_{r}-Z_{r} )\,
\mathrm{d}W_{r} \bigr\rangle.
\end{eqnarray*}
Our next goal will be to prove that
%
\begin{equation}
\label{ineq 1 Lemma} \mbox{there exists }M^{\varepsilon}\in\mathcal{M}\mbox{ such that
}%
\lim_{\varepsilon\rightarrow0}B_{t,s} \bigl(
M^{\varepsilon} \bigr) =0,\qquad \mbox{a.s., }\forall 0\leq t\leq s. %
\end{equation}
Let $M_{t}^{\varepsilon}=\mathrm{e}^{-\fraca{Q_{t}}{Q_{\varepsilon}}} [Y_{0}+\frac
{1}{Q_{\varepsilon}}{\int_{0}^{t}}\mathrm{e}^{\fraca{Q_{r}}{Q_{\varepsilon}}}Y_{r}\,\mathrm{d}Q_{r} ]$.
Clearly, $M^{\varepsilon}\in\mathcal{M}$, since $M_{t}%
^{\varepsilon}=M_{0}^{\varepsilon}+{\int_{0}^{t}}\,\mathrm{d}M_{r}^{\varepsilon} $.

The next result it is necessary in order to obtain the limit in the Stieltjes
type integrals:

\begin{lemma}\label{limit in Stieltjes integral}
Let $a\dvtx  [  0,T ]  \rightarrow
\mathbb{R}$ be a strictly increasing continuous function such that $a (
0 )  =0$ and $f\dvtx  [  0,T ]  \rightarrow H$ be a measurable
function such that $\llvert  f (  t )  \rrvert  \leq C$ a.e.
$t\in [  0,T ]  $. Define, for $\varepsilon>0$,
\[
f_{\varepsilon} ( t ) =f ( 0 ) \mathrm{e}^{\fraca{-a (
t )  }{a (  \varepsilon )  }}+\frac{1}{a (  \varepsilon
)  }\int
_{0}^{t}\mathrm{e}^{\fracb{a (  r )  -a (  t )
}{a (  \varepsilon )  }}f ( r )\,
\mathrm{d}a ( r ) .
\]
Then as $\varepsilon\rightarrow0$, $f_{\varepsilon}(t)\rightarrow
f(t)$, a.e. $t\in[0,T]$ and, if $f$ is continuous, then $\sup_{s\in[0,T]}\llvert f_{\varepsilon
}(t)-f(t)\rrvert \rightarrow 0 $.
\end{lemma}

\begin{remark}\label{limit in Stieltjes integral 2}
The same conclusions are true if we
consider $f_{\varepsilon} (  t )  $ replaced by
\[
g_{\varepsilon} ( t ) =f ( T ) \mathrm{e}^{\fracb{a (
t )  -a (  T )  }{a (  \varepsilon )  }}+\frac
{1}{a (  \varepsilon )  }\int
_{t}^{T}\mathrm{e}^{\fracb{a (  t )
-a (  r )  }{a (  \varepsilon )  }}f ( r )\,
\mathrm{d}a ( r ) ,\qquad t\in [ 0,T ] .
\]

\begin{pf*}{Proof of  Lemma \ref{limit in Stieltjes integral}}
Obviously, we have
\begin{eqnarray*}
&&{\int_{0}^{t}}
\frac{1}{a (  \varepsilon )  }%
\mathrm{e}^{\fracb{a (  r )  -a (  t )  }{a (  \varepsilon )
}}f ( r )\,\mathrm{d}a ( r )
\\
&&\quad={\int_{\fraca{-a (
t )  }{a (  \varepsilon )  }}^{0}}\mathrm{e}^{u}f
\bigl( \bigl( a^{-1} \bigl( ua ( \varepsilon ) +a ( t ) \bigr) \bigr)
\bigr)\,\mathrm{d}u
\\
&&\quad={\int_{\fraca{-a (  t )  }{a (  \varepsilon )
}}^{0}}\mathrm{e}^{u}
\bigl[ f \bigl(a^{-1} \bigl( ua ( \varepsilon ) +a ( t ) \bigr)
\bigr)-f \bigl(a^{-1} \bigl( a ( t ) \bigr) \bigr) \bigr]\,
\mathrm{d}u+f ( t ) {\int_{\fraca
{-a (  t )  }{a (  \varepsilon )  }}^{0}}
\mathrm{e}^{u}\,\mathrm{d}u.
\end{eqnarray*}
But
\begin{eqnarray*}
&&\limsup_{\varepsilon\rightarrow0} \biggl\llvert { \int
_{\fraca
{-a (  t )  }{a (  \varepsilon )  }}^{0}}\mathrm{e}^{u} \bigl[ f
\bigl(a^{-1} \bigl( ua ( \varepsilon ) +a ( t ) \bigr) \bigr)-f
\bigl(a^{-1} \bigl( a ( t ) \bigr) \bigr) \bigr]\,\mathrm{d}u \biggr
\rrvert
\\
&&\quad\leq\limsup_{\varepsilon\rightarrow0}{\int_{-{\infty}}^{0}}
\mathrm{e}^{u} \bigl\llvert f \bigl(a^{-1} \bigl( \bigl( ua
( \varepsilon ) +a ( t ) \bigr) \vee0 \bigr) \bigr)-f \bigl(a^{-1}
\bigl( a ( t ) \bigr) \bigr) \bigr\rrvert\, \mathrm{d}u
\\
&&\quad\leq2C{\int_{-{\infty}}^{-n}}
\mathrm{e}^{u}\,\mathrm{d}u+{\int_{-n}^{0}}
\mathrm{e}^{u} \bigl\llvert f \bigl(a^{-1} \bigl( \bigl( ua
( \varepsilon ) +a ( t ) \bigr) \vee0 \bigr) \bigr)-f \bigl(a^{-1}
\bigl( a ( t ) \bigr) \bigr) \bigr\rrvert \,\mathrm{d}u
\\
&&\quad\leq2C\mathrm{e}^{-n}+\limsup_{\varepsilon\rightarrow0}{\int
_{-n}^{0}} \bigl\llvert f \bigl(a^{-1}
\bigl( \bigl( ua ( \varepsilon ) +a ( t ) \bigr) \vee0 \bigr) \bigr)-f
\bigl(a^{-1} \bigl( a ( t ) \bigr) \bigr) \bigr\rrvert \,\mathrm{d}u \leq2C
\mathrm{e}^{-n},
\end{eqnarray*}
for all $n$, since $\lim_{\delta\rightarrow0}\int_{\alpha}^{\beta}\llvert  f (a^{-1}%
(  s+\delta u )   )-f (a^{-1} (  s )
)\rrvert\,  \mathrm{d}u=0 $ a.e.

Therefore, there exists
\[
\lim_{\varepsilon\rightarrow0} \biggl\llvert {\int_{\fraca{-a (
t )  }{a (  \varepsilon )  }}^{0}}
\mathrm{e}^{u} \bigl[ f \bigl(a^{-1}%
\bigl( ua (
\varepsilon ) +a ( t ) \bigr) \bigr)-f ( t ) \bigr]\,\mathrm{d}u \biggr\rrvert
=0,
\]
and the first conclusion follows.

In the case of continuity for $f$ it is sufficient to write
\begin{eqnarray*}
f_{\varepsilon} ( t ) &=&f ( 0 ) \mathrm{e}^{\fraca{-a (
t )  }{a (  \varepsilon )  }}+\frac{1}{a (  \varepsilon
)  }{
\int_{0}^{t}}\mathrm{e}^{\fracb{a (  r )  -a (
t )  }{a (  \varepsilon )  }}f ( r )\,
\mathrm{d}a ( r )
\\
&=&f ( 0 ) \mathrm{e}^{\fraca{-a (  t )  }{a (  \varepsilon
)  }}+\frac{1}{a (  \varepsilon )  }{\int
_{0}^{t_{\varepsilon}}}\mathrm{e}^{\fracb{a (  r )  -a (  t )
}{a (  \varepsilon )  }}f ( r )\,
\mathrm{d}a ( r ) \\
&&{}+\frac{1}{a (  \varepsilon )  }{\int_{t_{\varepsilon}%
}^{t}}
\mathrm{e}^{\fracb{a (  r )  -a (  t )  }{a (
\varepsilon )  }}f ( r )\,\mathrm{d}a ( r ) ,
\end{eqnarray*}
where $t_{\varepsilon}:=a^{-1} (a (  t )  -\sqrt{a (
\varepsilon )  } )
\rightarrow
t$, as $\varepsilon\rightarrow0$, and $t_{\varepsilon}<t$.
\end{pf*}
\end{remark}

Applying the above lemma, we can conclude that
%
\begin{equation}
\label{ineq 3 Lemma} M_{t}^{\varepsilon} \rightarrow Y_{t} ,
\qquad\forall t\in [ 0,T ] . %
\end{equation}
Next, we shall prove that, for all $t\leq s$,
\[
{\int_{t}^{s}{\Psi} \bigl( r,M_{r}^{\varepsilon}
\bigr) }\,\mathrm{d}Q_{r} \rightarrow {\int_{t}^{s}{
\Psi} ( r,Y_{r} ) }\,\mathrm{d}Q_{r} .
\]
Using definition of $M^{\varepsilon}$ and the convexity of the functions
$\varphi$ and $\psi$ we deduce that
%
%
\begin{eqnarray}
\label{ineq Stieltjes phi} %
{\int_{t}^{s}} \varphi
\bigl( M_{r}^{\varepsilon} \bigr) \alpha_{r}\,
\mathrm{d}Q_{r} &\leq&{\int_{t}^{s}}
\mathrm{e}^{\fraca{-Q_{r}}{Q_{\varepsilon
}}}\varphi ( Y_{0} )\,\mathrm{d}r+{\int
_{t}^{s}} \biggl( { \int_{0}^{r}}
\frac{1}{Q_{\varepsilon}}\mathrm{e}^{\fracb{Q_{u}-Q_{r}%
}{Q_{\varepsilon}}} \varphi ( Y_{u} )\,
\mathrm{d}Q_{u} \biggr)\,\mathrm{d}r
\nonumber
\\
&=&\varphi ( Y_{0} ) {\int_{t}^{s}}
\mathrm{e}^{\fraca{-Q_{r}%
}{Q_{\varepsilon}}}\,\mathrm{d}r+{ \int_{0}^{s}}
\biggl( {\int_{0}^{s}} \frac{1}{Q_{\varepsilon}}
\mathrm{e}^{\fracb{Q_{u}-Q_{r}}{Q_{\varepsilon}}%
}\varphi ( Y_{u} ) \mathbh{1}_{ [  0,r ]  }
( u )\,\mathrm{d}Q_{u} \biggr)\,\mathrm{d}r
\nonumber
\\
&&{}-{\int_{0}^{t}} \biggl( {\int
_{0}^{t}}\frac
{1}{Q_{\varepsilon}}\mathrm{e}^{\fracb{Q_{u}-Q_{r}}{Q_{\varepsilon}}}
\varphi ( Y_{u} ) \mathbh{1}_{ [  0,r ]  } ( u )\,
\mathrm{d}Q_{u} \biggr)\,\mathrm{d}r
\\
&=&\varphi ( Y_{0} ) {\int_{t}^{s}}
\mathrm{e}^{\fraca{-Q_{r}%
}{Q_{\varepsilon}}}\,\mathrm{d}r+{ \int_{0}^{s}}
\biggl( \varphi ( Y_{u} ) {\int_{0}^{s}}
\frac{1}{Q_{\varepsilon}}%
\mathrm{e}^{\fracb{Q_{u}-Q_{r}}{Q_{\varepsilon}}}\mathbh{1}_{ [  u,s ]
}
( r )\,\mathrm{d}r \biggr)\,\mathrm{d}Q_{u}
\nonumber
\\
&&{}-{\int_{0}^{t}} \biggl( \varphi (
Y_{u} ) {\int_{0}^{t}}
\frac{1}{Q_{\varepsilon}}\mathrm{e}^{\fracb{Q_{u}-Q_{r}%
}{Q_{\varepsilon}}}\mathbh{1}_{ [  u,t ]  } ( r )\,
\mathrm{d}r \biggr)\,\mathrm{d}Q_{u}
\nonumber
\end{eqnarray} %
and
%
%
\begin{eqnarray}
\label{ineq Stieltjes psi} %
&&{\int_{t}^{s}} \psi
\bigl( M_{r}^{\varepsilon} \bigr) ( 1-\alpha_{r} )\,
\mathrm{d}Q_{r}\nonumber\\
&&\quad\leq{\int_{t}^{s}}
\mathrm{e}^{\fraca{-Q_{r}%
}{Q_{\varepsilon}}} \psi ( Y_{0} )\,\mathrm{d}A_{r}+{
\int_{t}%
^{s}} \biggl( {\int
_{0}^{r}} \frac{1}{Q_{\varepsilon}}%
\mathrm{e}^{\fracb{Q_{u}-Q_{r}}{Q_{\varepsilon}}}\psi ( Y_{u} )\,\mathrm{d}Q_{u}
\biggr)\,\mathrm{d}A_{r}
\nonumber
\\
&&\quad=\psi ( Y_{0} ) {\int_{t}^{s}}
\mathrm{e}^{\fraca{-Q_{r}%
}{Q_{\varepsilon}}}\,\mathrm{d}A_{r}+{ \int
_{0}^{s}} \biggl( {\int_{0}^{r}}
\frac{1}{Q_{\varepsilon}}\mathrm{e}^{\fracb{Q_{u}-Q_{r}%
}{Q_{\varepsilon}}}\psi ( Y_{u} )\,
\mathrm{d}Q_{u} \biggr)\,\mathrm{d}A_{r}
\\
&&\qquad{}-{\int_{0}^{t}} \biggl( {\int
_{0}^{r}}\frac
{1}{Q_{\varepsilon}}\mathrm{e}^{\fracb{Q_{u}-Q_{r}}{Q_{\varepsilon}}}
\psi ( Y_{u} )\,\mathrm{d}Q_{u} \biggr)\,
\mathrm{d}A_{r}\nonumber
\\
&&\quad=\psi ( Y_{0} ) {\int_{t}^{s}}
\mathrm{e}^{\fraca{-Q_{r}%
}{Q_{\varepsilon}}}\,\mathrm{d}A_{r}+{ \int
_{0}^{s}} \biggl( \psi ( Y_{u} ) {
\int_{0}^{s}}\frac{1}{Q_{\varepsilon}}%
\mathrm{e}^{\fracb{Q_{u}-Q_{r}}{Q_{\varepsilon}}}\mathbh{1}_{ [  u,s ]
} ( r )\,\mathrm{d}A_{r}
\biggr)\,\mathrm{d}Q_{u}
\nonumber
\\
&&\qquad{}-{\int_{0}^{t}} \biggl( \psi (
Y_{u} ) {\int_{0}^{t}}
\frac{1}{Q_{\varepsilon}}\mathrm{e}^{\fracb{Q_{u}-Q_{r}%
}{Q_{\varepsilon}}} \mathbh{1}_{ [  u,t ]  } ( r )\,
\mathrm{d}A_{r} \biggr)\,\mathrm{d}Q_{u} .
\nonumber
\end{eqnarray}
On the other hand, using Remark \ref{limit in Stieltjes integral 2} and
Lebesgue's dominated convergence theorem, we conclude that
%
\begin{equation}
\label{ineq 5 Lemma} \lim_{\varepsilon\rightarrow0}\int_{u}^{s}
\frac{1}{Q_{\varepsilon}%
}\mathrm{e}^{\fracb{Q_{u}-Q_{r}}{Q_{\varepsilon}}}\alpha_{r}\,
\mathrm{d}Q_{r}=\lim_{\varepsilon\rightarrow0}\int_{u}^{t}
\frac{1}{Q_{\varepsilon}}%
\mathrm{e}^{\fracb{Q_{u}-Q_{r}}{Q_{\varepsilon}}}\alpha_{r}\,
\mathrm{d}Q_{r}= \alpha_{u}%
, \qquad
\mbox{a.e.}%
\end{equation}
and respectively,
%
\begin{eqnarray}
\label{ineq 6 Lemma}
 \lim_{\varepsilon\rightarrow0}\int_{u}^{s}
\frac{1}{Q_{\varepsilon}%
}\mathrm{e}^{\fracb{Q_{u}-Q_{r}}{Q_{\varepsilon}}} ( 1-\alpha_{r} )\,
\mathrm{d}Q_{r}&=&\lim_{\varepsilon\rightarrow0}\int_{u}^{t}
\frac{1}%
{Q_{\varepsilon}}\mathrm{e}^{\fracb{Q_{u}-Q_{r}}{Q_{\varepsilon}}} ( 1-\alpha _{r} )\,
\mathrm{d}Q_{r}\nonumber\\[-8pt]\\[-8pt]
&=&\alpha_{u}, \qquad\mbox{a.e.}\nonumber %
\end{eqnarray}
From inequalities (\ref{ineq Stieltjes phi}) and (\ref{ineq Stieltjes psi}), we
obtain
\begin{eqnarray*}
&&{\int_{t}^{s}{ \Psi} (
r,Y_{r} ) }\,\mathrm{d}Q_{r}\\
&&\quad\leq {\int
_{t}^{s}{ \Psi} \bigl( r,M_{r}^{\varepsilon}
\bigr) }\,\mathrm{d}Q_{r}
\\
&&\quad\leq\varphi ( Y_{0} ) Q_{\varepsilon}\mathrm{e}^{\fraca{-Q_{t}}%
{Q_{\varepsilon}}}{
\int_{t}^{s}}\frac{1}{Q_{\varepsilon}}%
\mathrm{e}^{\fracb{Q_{t}-Q_{r}}{Q_{\varepsilon}}}\,\mathrm{d}r+\psi ( Y_{0} ) Q_{\varepsilon}
\mathrm{e}^{\fraca{-Q_{t}}{Q_{\varepsilon}}}{ \int_{t}^{s}%
}\frac{1}{Q_{\varepsilon}}\mathrm{e}^{\fracb{Q_{t}-Q_{r}}{Q_{\varepsilon}}}\,\mathrm{d}A_{r}
\\
&&\qquad{}+{\int_{0}^{s}} \biggl( \varphi (
Y_{u} ) {\int_{u}^{s}}
\frac{1}{Q_{\varepsilon}}\mathrm{e}^{\fracb{Q_{u}-Q_{r}%
}{Q_{\varepsilon}}}\,\mathrm{d}r+\psi ( Y_{u}
) {\int_{u}^{s}%
}\frac{1}{Q_{\varepsilon}}
\mathrm{e}^{\fracb{Q_{u}-Q_{r}}{Q_{\varepsilon}}}\,\mathrm{d}A_{r} \biggr)\,
\mathrm{d}Q_{u}
\\
&&\qquad{}-{\int_{0}^{t}} \biggl( \varphi (
Y_{u} ) {\int_{u}^{t}}
\frac{1}{Q_{\varepsilon}}\mathrm{e}^{\fracb{Q_{u}-Q_{r}%
}{Q_{\varepsilon}}}\,\mathrm{d}r+\psi ( Y_{u}
) {\int_{u}^{t}%
}\frac{1}{Q_{\varepsilon}}
\mathrm{e}^{\fracb{Q_{u}-Q_{r}}{Q_{\varepsilon}}}\,\mathrm{d}A_{r} \biggr)\,
\mathrm{d}Q_{u},
\end{eqnarray*}
and applying the limits (\ref{ineq 5 Lemma}) and (\ref{ineq 6 Lemma}), we deduce
%
\begin{eqnarray}
\label{ineq 4 Lemma} {\int_{t}^{s}{\Psi} \bigl(
r,M_{r}^{\varepsilon} \bigr) }\,\mathrm{d}Q_{r} &
\xrightarrow{} {\varepsilon\rightarrow0}&%
{\int_{t}^{s}}
\bigl(\varphi ( Y_{u} ) \alpha_{u}+\psi (
Y_{u} ) ( 1-\alpha_{u} ) \bigr)\,\mathrm{d}Q_{u}
\nonumber
\\[-8pt]
\\[-8pt]
&=&{\int_{t}^{s}{\Psi } ( u,Y_{u} )
}\,\mathrm{d}Q_{u} .
\nonumber
\end{eqnarray}
Therefore (\ref{ineq 1 Lemma}) follows immediately, since we have
(\ref{ineq 3 Lemma}) and (\ref{ineq 4 Lemma}).

Now returning to inequality (\ref{uniqueness_Gronwall}), for all $0\leq t\leq s$,
%
\begin{eqnarray}
\label{uniqueness_Gronwall 2} %
&& \bigl\llvert Y_{t}^{1}-Y_{t}^{2}
\bigr\rrvert ^{2}+{\int_{t}^{s} } \bigl
\llvert Z_{r}^{1}-Z_{r}^{2} \bigr
\rrvert ^{2}\,\mathrm{d}r
\nonumber
\\
&&\quad\leq \bigl\llvert Y_{s} ^{1}-Y_{s}^{2}
\bigr\rrvert ^{2}+{{2}\int_{t}^{s}}
\bigl\langle Y_{r} ^{1}-Y_{r}^{2},
\Phi \bigl( r,Y_{r}^{1},Z_{r}^{1}
\bigr) -\Phi \bigl( r,Y_{r}^{2},Z_{r}^{2}
\bigr) \bigr\rangle\,\mathrm{d}Q_{r}
\\
&&\qquad{}-2{\int_{t}^{s}} \bigl\langle
Y_{r}^{1}-Y_{r}^{2}, \bigl(
Z_{r}^{1}-Z_{r}^{2} \bigr)\,
\mathrm{d}W_{r} \bigr\rangle.
\nonumber
\end{eqnarray}
From (\ref{a priori estim 2}),
\[
\bigl\langle Y_{r}^{1}-Y_{r}^{2},
\Phi \bigl( r,Y_{r}^{1},Z_{r}^{1}
\bigr) -\Phi \bigl( r,Y_{r}^{2},Z_{r}^{2}
\bigr) \bigr\rangle\,\mathrm{d}Q_{r}\leq\bigl\llvert Y_{r}^{1}%
-Y_{r}^{2}\bigr\rrvert ^{2}\,\mathrm{d}
\tilde{V}_{r}+ \frac{1}{2a}\bigl\llvert Z_{r}-Z_{r}^{2}
\bigr\rrvert ^{2}\,\mathrm{d}r,
\]
and therefore inequality (\ref{uniqueness_Gronwall 2}) becomes
\begin{eqnarray*}
&& \bigl\llvert Y_{t}^{1}-Y_{t}^{2}
\bigr\rrvert ^{2}+ ( 1-1/a ) {\int_{t}^{s}}
\bigl\llvert Z_{r}^{1}-Z_{r}^{2}
\bigr\rrvert ^{2}\,\mathrm{d}r
\\
&&\quad\leq \bigl\llvert Y_{s}^{1}-Y_{s}^{2}
\bigr\rrvert ^{2}+{{2}\int_{t}^{s}}
\bigl\llvert Y_{r}^{1}-Y_{r}^{2}
\bigr\rrvert ^{2}\,\mathrm{d}\tilde{V}%
_{r}-2{
\int_{t}^{s}} \bigl\langle Y_{r}^{1}-Y_{r}^{2},
\bigl( Z_{r}%
^{1}-Z_{r}^{2}
\bigr)\,\mathrm{d}W_{r} \bigr\rangle.
\end{eqnarray*}
Applying a Gronwall's type stochastic inequality (see Lemma 12 from the
Appendix of Maticiuc and R\u{a}\c{s}canu \cite{ma-ra/07}) we conclude that, for all $0\leq t\leq s$,
$\mathbb{P}$-a.s.
\[
\mathrm{e}^{2\tilde{V}_{t}} \bigl\llvert Y_{t}^{1}-Y_{t}^{2}
\bigr\rrvert ^{2}\leq \mathrm{e}^{2\tilde{V}_{s}} \bigl\llvert
Y_{s}^{1}-Y_{s}^{2} \bigr\rrvert
^{2}-2{\int_{t}^{s}}
\mathrm{e}^{2\tilde{V}_{r}} \bigl\langle Y_{r}^{1}-Y_{r}^{2},
\bigl( Z_{r}^{1}%
-Z_{r}^{2}
\bigr)\,\mathrm{d}W_{r} \bigr\rangle.
\]
Therefore, using also the condition (\ref{definition sol 4})(iii) form the
definition of weak variational solution, the uniqueness follows.\hfill
\end{pf}

\section{Examples}

Let $\mathcal{D}\subset\mathbb{R}^{d}$ be an open bounded subset with boundary
$\Bd (  \mathcal{D} )  $ sufficiently smooth. In what follows
$H^{m} (  \mathcal{D} )  $ and $H_{0}^{m} (  \mathcal{D} )
$ stand for the usual Sobolev spaces. Let $(\Omega,\mathcal{F}, (
\mathcal{F}_{t} )  _{t\geq0},\mathbb{P})$ be a complete probability
space, $ \{  W_{s}\dvtx 0\leq s\leq t \}  $ a real Wiener process and set
$H=H_{1}:=L^{2} (  \mathcal{D} )  $. We notice that the space of
Hilbert--Schmidt operators from $L^{2} (  \mathcal{D} )  $ to
$L^{2} (  \mathcal{D} )$ can be identified with $L^{2} (
\mathcal{D}\times\mathcal{D} )  $.

Let $j\dvtx \mathbb{R\rightarrow(\infty},{\infty]}$ be a proper convex
l.s.c. function, for which we assume that $j (  u )  \geq j (
0 )  =0$, $\forall u\in\mathbb{R}$.

Our aim is to obtain, via Theorem \ref{main result 1}, the existence and
uniqueness of the solution for some backward stochastic partial differential
equations (SPDE) suggested in Pardoux and R\u{a}\c{s}canu \cite{pa-ra/99}. We recall assumptions
{(A$_{1}$)--(A$_{5}$)}, {(A$_{8}$)}(\ref{Assumption A_9}), condition $\mathbb{E} (  Q_{T}^{p} )  <\infty$, $\forall T>0$, and
definitions of $\Phi$ and $V$ from Section \ref{sec2.2}.

\begin{example}
First we consider the backward SPDE with Dirichlet boundary condition
%
\begin{equation}
\label{BSPDE 1} \cases{\displaystyle -dY ( t,x ) +\partial j \bigl( Y ( t,x )
\bigr)\,\mathrm{d}Q_{t}\ni\Delta Y ( t,x )\,\mathrm{d}Q_{t}+
\Phi \bigl( t,Y ( t,x ) ,Z ( t,x ) \bigr)\,\mathrm{d}Q_{t}\vspace*{2pt}
\cr
\displaystyle\hspace*{127pt}{} -Z ( t,x )\,\mathrm{d}W_{t},\qquad
\mbox{in }\Omega\times [ 0,\tau ] \times \mathcal{D},\vspace*{2pt}
\cr
\displaystyle{Y ( \omega,t,x ) =0 \qquad\mbox{on }\Omega \times [ 0,\tau ] \times
\Bd ( \mathcal{D} ) , }\vspace*{2pt}
\cr
\displaystyle{\mathrm{e}^{2V_{T}}
\bigl\llVert Y ( T ) -\xi_{T} \bigr\rrVert ^{2}+{\int
_{T}^{\infty}}\mathrm{e}^{2V_{s}} \bigl\llVert Z (
s ) -\zeta_{s} \bigr\rrVert ^{2}\,\mathrm{d}s \xrightarrow{T
\rightarrow\infty} {\mbox{{\fontsize{7.6}{9.6}\selectfont{prob.}}}} 0,}} %
\end{equation}
where $\llVert  f\rrVert  ^{2}:={\int_{\mathcal{D}}}\llvert  f (
x )  \rrvert  ^{2}\,\mathrm{d}x $.
\end{example}

Let us apply Theorem \ref{main result 1}, with $\Psi=\varphi=\psi$ (in which
case the compatibility assumptions (\ref{compatib. assumption}) are
satisfied), where $\varphi\dvtx L^{2} (  \mathcal{D} )  \rightarrow
(-\infty,\infty]$ is given by
\[
\varphi ( u ) = \cases{\displaystyle \frac{1}{2}\int_{\mathcal{D}}
\bigl\llvert \nabla u ( x ) \bigr\rrvert ^{2}\,\mathrm{d}x+\int
_{\mathcal{D}}j \bigl( u ( x ) \bigr)\,\mathrm{d}x, &\quad\mbox{if }
$u\in H_{0}^{1} ( \mathcal{D} ) , j ( u ) \in
L^{1} ( \mathcal{D} )$,\vspace*{2pt}
\cr
\displaystyle +\infty, &\quad
\mbox{otherwise.}%
}
\]
Proposition 2.8 from Barbu \cite{ba/10}, Chapter II, shows that the following
properties hold:
\begin{enumerate}[(b)]
\item[(a)] function $\mathit{\varphi}$ is proper, convex and l.s.c.,

\item[(b)]  $\partial\varphi ( u ) =  \{ u^{\ast}\in L^{2} (
\mathcal{D} ) \dvtx  u^{\ast} ( x ) \in\partial j  ( u ( x )  ) -\Delta
u ( x ) \mbox{ a.e. on }\mathcal{D}  \} , \forall u\in\Dom ( \partial\varphi )$,

\item[(c)]  $ \Dom ( \partial \varphi ) =  \{ u\in
H_{0}^{1} ( \mathcal{D} ) \cap H^{2} (
\mathcal{D} ) \dvtx u ( x ) \in\Dom ( \partial j ) \mbox{ a.e. on }\mathcal{D}  \}$.
\end{enumerate}
Moreover, there exists a positive constant $C$ such that
\begin{enumerate}[(d)]

\item[(d)] $\llVert u\rrVert
_{H_{0}^{1} (  \mathcal{D}%
)  \cap H^{2} (  \mathcal{D} )  }\leq C \llVert u^{\ast
} \rrVert
_{L^{2} (  \mathcal{D} )  }, \forall  ( u,u^{\ast
}  ) \in\partial\varphi$.
\end{enumerate}

Let $\eta$ be a $H_{0}^{1} (  \mathcal{D} )  $-valued random
variable, $\mathcal{F}_{\tau}$-measurable such that {(A$_{9}$)} is
satisfied and
\[
\mathbb{E} \bigl[\mathrm{e}^{p\sup_{s\in [  0,\tau ]  }\tilde{V}_{s}}\llvert \eta\rrvert ^{p}
\bigr]< \infty,\qquad \mathrm{e}^{2\sup_{s\in [  0,\tau ]
}\tilde{V}_{s}}j ( \eta ) \in L^{1} (
\Omega\times \mathcal{D} ) ,
\]
and the stochastic processes $\xi,\zeta$, associated to $\eta$ by the
martingale representation theorem, such that
%
\begin{equation}
\label{Assumption A_8'} \mathbb{E} \biggl( {\int_{0}^{\tau}}
\mathrm{e}^{2\tilde{V}_{s}} \varphi ( \xi _{s} )\,\mathrm{d}Q_{s}
\biggr) ^{p/2}+ \mathbb{E} \biggl( {\int_{0}^{\tau}%
}\mathrm{e}^{\tilde{V}_{s}} \bigl\llvert \Phi ( s,\xi_{s},
\zeta_{s} ) \bigr\rrvert\,\mathrm{d}Q_{s} \biggr) ^{p}<
\infty, %
\end{equation}
where $\tilde{V}$ is defined by (\ref{def V tilde}).

Applying now Theorem \ref{main result 1}, we deduce that, under the above
assumptions, backward SPDE (\ref{BSPDE 1}) has a unique solution $ (
Y,Z,U )  \in\mathcal{S}_{L^{2} (  \mathcal{D} )  }^{0}%
\times\Lambda_{L^{2} (  \mathcal{D}\times\mathcal{D} )  }^{0}%
\times\Lambda_{L^{2} (  \mathcal{D} )  }^{0}$ such that $ (
Y (  t )  ,Z (  t )   )  = (  \xi_{t},\zeta
_{t} )  = (  \eta,0 )  $, for $t\geq\tau$, and
\begin{enumerate}[(iii)] %
\item[(i)]  $Y ( t,x ) +{\int_{t}^{T}}\!U ( s,x )\,\mathrm{d}Q_{s}=Y ( T,x ) +{
\int_{t}^{T}}\!\Delta Y ( s,x )\,\mathrm{d}Q_{s}
+{\int_{t}^{T}}\!\Phi  ( s,Y ( s,x ) ,  Z ( s, x
)  )\,\mathrm{d}Q_{s}-{\int_{t}^{T}}\!Z ( s,x )\,\mathrm{d}W_{s} , \mbox{in } [ 0,T ] \times\mathcal{D}$, a.s.,

\item[(ii)]  $Y ( t ) \in H_{0}^{1} ( \mathcal{D})
\cap H^{2} ( \mathcal{D} ) , \mathrm{d}\mathbb{P}\times \mathrm{d}t$ a.e.,

\item[(iii)] $Y ( t,x ) \in\Dom ( \partial j ) , \mathrm{d}\mathbb{P}\times \mathrm{d}Q_{t}
\times \mathrm{d}x$ a.e.,

\item[(iv)] $U ( t,x ) \in\partial j  ( Y ( t,x )  ) , \mathrm{d}\mathbb{P}\times
\mathrm{d}Q_{t}\times \mathrm{d}x$ a.e.,

\item[(v)] $\mathrm{e}^{2\tilde{V}}Y\in L^{\infty}  ( 0,T;L^{2}  (
\Omega;H_{0}^{1} ( \mathcal{D} )  )  )$  and
$\mathrm{e}^{2\tilde{V}}j(Y)\in L^{\infty}  ( 0,T;L^{1}(\Omega\times
\mathcal{D})  ), \forall T>0$,

\item[(vi)] $\mathbb{E} {\int_{0}^{\tau}}\mathrm{e}^{2\tilde{V}%
_{s}}
\llVert Y ( s )  \rrVert _{H_{0}^{1} (
\mathcal{D} )  \cap H^{2} (  \mathcal{D} )  }^{2}\,\mathrm{d}Q_{s}%
<{\infty}$.
\end{enumerate}

\begin{remark}
If we renounce at assumption {(A$_{9}$)}, then it follows that
backward SPDE (\ref{BSPDE 1}) admits a variational weak solution. More
precisely, Theorem \ref{main result 2} shows that there exists a unique
solution $ (  Y,Z )  \in\mathcal{S}_{L^{2} (  \mathcal{D} )
}^{0}\times\Lambda_{L^{2} (  \mathcal{D}\times\mathcal{D} )  }^{0}$
such that $ (  Y (  t )  ,Z (  t )   )  = (
\xi_{t},\zeta_{t} )  = (  \eta,0 )  $, for $t\geq\tau$, and for
all $0\leq t\leq s$
\begin{enumerate}[(iii)]
\item[(i)]
$ {{\frac{1}{2}}}
\llVert M ( t ) -Y ( t )  \rrVert ^{2}+{{\frac{1}{2}}\int_{t}^{s}} \llvert R ( r ) -Z ( r )
\rrvert ^{2}\,\mathrm{d}r+{\int_{t}^{s}\int_{\mathcal{D}}j}  ( Y ( r,x )  )\,\mathrm{d}x\,\mathrm{d}Q_{r}
\leq{{\frac{1}{2}}} \llVert M ( s ) -Y ( s )  \rrVert
^{2}+{\int_{t}^{s}\int_{\mathcal{D}}j}  ( M ( r,x )  )\,\mathrm{d}x\,\mathrm{d}Q_{r}
+{\int_{t}^{s}} \langle  \langle M ( r
) -Y ( r ) ,N ( r ) -\Delta Y ( r ) -\Phi  ( r,Y ( r ) ,\allowbreak Z ( r )  )
\rangle  \rangle\,\mathrm{d}Q_{r}
-\int_{t}^{s} \langle  \langle M ( r
) -Y ( r ) ,  ( R ( r ) -Z ( r )  )\,\mathrm{d}W_{r}  \rangle
\rangle ,
\forall ( N,R ) \in L^{2}  ( \Omega
\times[0,{\infty} );\allowbreak L^{2} ( \mathcal{D} ) ) \times
\Lambda_{L^{2} (  \mathcal{D}\times\mathcal{D} )
}^{2}, \forall M\in\mathcal{M}$,

\item[(ii)]  $Y ( t ) \in H_{0}^{1} ( \mathcal{D}), \mathrm{d}\mathbb{P}\times \mathrm{d}t$ a.e.,

\item[(iii)]  $Y ( t,x ) \in\Dom ( j ) , \mathrm{d}\mathbb{P}\times \mathrm{d}Q_{t}\times \mathrm{d}x$ a.e.,
\end{enumerate}
where $\mathcal{M}$ is defined at the beginning of the Section \ref{weak form}
and
\[
\llVert f\rrVert ^{2}:={\int_{\mathcal{D}}} \bigl\llvert f
( x ) \bigr\rrvert ^{2}\,\mathrm{d}x\quad\mbox{and}\quad \bigl\langle
\langle f,g \rangle \bigr\rangle:={\int_{\mathcal{D}}}f ( x ) g ( x )
\,\mathrm{d}x.
\]
\end{remark}

\begin{example}
As a second example we consider the backward SPDE with Neumann boundary
condition
%
\begin{equation}
\label{BSPDE 2} \cases{\displaystyle -\mathrm{d}Y ( t,x ) =\Delta Y ( t,x )\,
\mathrm{d}Q_{t}\vspace*{2pt}
\cr
\displaystyle\hspace*{53pt} {}+\Phi
\bigl( t,Y ( t,x ) ,Z ( t,x ) \bigr)\,\mathrm{d}Q_{t}-Z_{t}
\,\mathrm{d}W_{t},\qquad\mbox{in }\Omega\times [ 0,\tau ] \times
\mathcal{D}%
,\vspace*{2pt}
\cr
\displaystyle {-\frac{\partial Y (  \omega,t,x )  }{\partial
n}\in
\partial j \bigl( Y ( \omega,t,x ) \bigr) ,\qquad \mbox{on }\Omega\times [ 0,\tau
] \times\Bd ( \mathcal{D}) , }\vspace*{2pt}
\cr
\displaystyle{
\mathrm{e}^{2V_{T}} \bigl\llVert Y ( T ) -\xi_{T} \bigr\rrVert
^{2}+{\int_{T}^{\infty}}\mathrm{e}^{2V_{s}}
\bigl\llVert Z ( s ) -\zeta_{s} \bigr\rrVert ^{2}\,
\mathrm{d}s \xrightarrow{T\rightarrow\infty} {\mbox{{\fontsize{7.6}{9.6}\selectfont{prob.}}}} 0.}}%
\end{equation}
\end{example}

We apply again Theorem \ref{main result 1}, with $\Psi=\varphi=\psi$, where
$\varphi\dvtx L^{2} (  \mathcal{D} )  \rightarrow(-\infty,\infty]$ is
given by
\[
\varphi ( u ) = \cases{\displaystyle%
\frac{1}{2}\int
_{\mathcal{D}} \bigl\llvert \nabla u ( x ) \bigr\rrvert ^{2}
\,\mathrm{d}x+\int_{\Bd (  \mathcal{D}%
)  }j \bigl( u ( x ) \bigr)\,\mathrm{d}x,&
\quad \mbox{if }$u\in H^{1} ( \mathcal{D} )$ \mbox{ and }$k ( u ) \in
L^{1} \bigl( \Bd ( \mathcal{D} ) \bigr)$,
\cr
\displaystyle+\infty,&
\quad\mbox{otherwise.}%
}
\]
Proposition 2.9 from Barbu \cite{ba/10}, Chapter II, shows that:
\begin{enumerate}[(b)]
\item[(a)] function $\varphi$ is proper, convex and l.s.c.,

\item[(b)] $\partial\varphi ( u ) =-\Delta u ( x ) , \forall u\in\Dom ( \partial
\varphi )$,

\item[(c)]  $\Dom ( \partial\varphi ) = \{u\in H^{2} ( \mathcal{D} )
\dvtx -\frac{\partial u (  x )
}{\partial n}\in\partial j  ( u ( x )  )\mbox{  a.e. on }\Bd (\mathcal{D} )  \}$,
\end{enumerate}
where $\frac{\partial u}{\partial n}$ is the outward normal
derivative to the boundary.

Moreover, there are some positive constants $C_{1}$, $C_{2}$ such that
\begin{enumerate}[(d)]
\item[(d)] $\llVert u\rrVert
_{H^{2} (  \mathcal{D} )
}\leq C_{1}\llVert u-\Delta u\rrVert
_{L^{2} (  \mathcal{D} )
}+C_{2} , \forall u\in\Dom ( \partial\varphi ) $.
\end{enumerate}

Let $\eta$ be a $H^{1} (  \mathcal{D} )  $-valued random variable,
$\mathcal{F}_{\tau}$-measurable such that {(A$_{9}$)} is satisfied
and
\[
\mathbb{E} \bigl[\mathrm{e}^{p\sup_{s\in [  0,\tau ]  }\tilde{V}_{s}}\llvert \eta\rrvert ^{p}
\bigr]< \infty,\qquad \mathrm{e}^{2\sup_{s\in [  0,\tau ]
}\tilde{V}_{s}}j ( \eta ) \in L^{1}
\bigl( \Omega\times \Bd ( \mathcal{D} ) \bigr) ,
\]
and the stochastic processes $\xi$ and $\zeta$ (from the martingale
representation theorem) be such that (\ref{Assumption A_8'}) holds.

Applying Theorem \ref{main result 1} we conclude that, under the above
assumptions, backward SPDE (\ref{BSPDE 2}) has a unique solution $ (
Y,Z )  \in\mathcal{S}_{L^{2} (  \mathcal{D} )  }^{0}%
\times\Lambda_{L^{2} (  \mathcal{D}\times\mathcal{D} )  }^{0}$ such
that $ (  Y (  t )  ,Z (  t )   )  = (  \xi
_{t},\zeta_{t} )  = (  \eta,0 )  $, for $t\geq\tau$, and
\begin{enumerate}[(iii)]
\item[(i)]  $Y ( t,x ) =Y ( T,x ) +{\int_{t}^{T}}\Delta Y ( s,x )\,\mathrm{d}Q_{s} +{\int_{t}^{T}}\Phi  ( s,Y ( s,x ) ,Z ( s,x
)  )\,\mathrm{d}Q_{s}-{\int_{t}^{T}}Z ( s,x )\,\mathrm{d}W_{s}$, in  $[ 0,T ] \times\mathcal{D}%
$,  a.s.,

\item[(ii)] $Y ( t ) \in H^{2} ( \mathcal{D} ) , \mathrm{d}\mathbb{P}\times \mathrm{d}t$
a.e.,

\item[(iii)] $-\frac{\partial Y (  t,x )  }{\partial n}%
\in\Dom ( \partial j ) , \mathrm{d}\mathbb{P}
\times \mathrm{d}Q_{t}\times \mathrm{d}x$ a.e.,

\item[(iv)] $\mathrm{e}^{2\tilde{V}}Y\in L^{\infty }  (0,T;L^{2} (
\Omega ;H^{1}(\mathcal{D}) )  )\mbox{ and }\mathrm{e}^{2\tilde{V}}j(Y)
\in L^{\infty }  (0,T;L^{1} (\Omega \times \Bd(
\mathcal{D}) )  ), \forall T>0$,

\item[(v)] $\mathbb{E} {\int_{0}^{\tau}}\mathrm{e}^{2\tilde{V}_{s}%
}
\llVert Y ( s )  \rrVert _{H^{2} (  \mathcal{D} )
}^{2}\,\mathrm{d}Q_{s}<{\infty}$.
\end{enumerate}

\begin{example}
The third example is the backward stochastic porous media equation
%
\begin{equation}
\label{BSPDE 3} \cases{\displaystyle -\mathrm{d}Y ( t,x ) =\Delta ( \partial j ) \bigl( Y (
t,x ) \bigr)\,\mathrm{d}Q_{t}+\Phi \bigl( t,Y ( t,x ) ,Z ( t,x )
\bigr)\,\mathrm{d}Q_{t}
\cr
\displaystyle\hspace*{56pt} {}-Z ( t,x )\,
\mathrm{d}W_{t} , \qquad\mbox{in }\Omega\times [ 0,\tau ] \times
\mathcal{D},
\cr
\displaystyle{\partial j \bigl( Y ( \omega,t,x ) \bigr) \ni0
\qquad \mbox{on } \Omega\times [ 0,\tau ] \times\Bd ( \mathcal{D} ) , }
\cr
\displaystyle {\mathrm{e}^{2V_{T}} \bigl\llVert Y ( T ) -\xi_{T}
\bigr\rrVert ^{2}+{\int_{T}^{\infty}}
\mathrm{e}^{2V_{s}} \bigl\llVert Z ( s ) -\zeta_{s} \bigr\rrVert
^{2}\,\mathrm{d}s %
\xrightarrow{T\rightarrow\infty}
{\mbox{{\fontsize{7.6}{9.6}\selectfont{prob.}}}}0.}} %
\end{equation}
\end{example}

In Theorem \ref{main result 1}, let $H=H^{-1} (  \mathcal{D} )  $
(the dual of $H_{0}^{1} (  \mathcal{D} )  $), $H_{1}=\mathbb{R}^{d}$
and $\Psi=\varphi=\psi$, where $\varphi\dvtx H^{-1} (  \mathcal{D} )
\rightarrow(-\infty,\infty]$ is given by
\[
\varphi ( u ) = \cases{\displaystyle \int_{\mathcal{D}}j \bigl( u (
x ) \bigr)\,\mathrm{d}x,& if $u\in L^{1} ( \mathcal{D} ) , j ( u ) \in
L^{1} ( \mathcal{D} ) $,
\cr
\displaystyle+\infty,&
otherwise,}%
\]
and $j\dvtx \mathbb{R}\rightarrow\mathbb{R}_{+}$ is suppose, moreover, to be
continuous with $\lim_{r\rightarrow{\infty}}j (  r )
/r={\infty}$.

Proposition 2.10 from Barbu \cite{ba/10}, Chapter II, shows that:
\begin{enumerate}[(b)]
\item[(a)] function $\varphi$
is proper, convex and l.s.c.,

\item[(b)]  $\partial\varphi ( u ) = \{u^{\ast}
\in H^{-1} ( \mathcal{D} ) \dvtx u^{\ast} ( x ) =-\Delta v  ( u
( x )  ) \mbox{, }v\in H_{0}^{1} ( \mathcal{D} ),
v ( x ) \in\partial j  ( u ( x
)  ) \mbox{ a.e. on }\mathcal{D}%
\},\allowbreak \forall u\in\Dom (
\partial\varphi ) $,

\item[(c)] $\Dom ( \partial\varphi ) =
\{ u\in H^{-1} ( \mathcal{D} ) \cap L^{1} ( \mathcal{D}
) \dvtx u ( x ) \in\Dom ( \partial j ) \mbox{ a.e. on }\mathcal{D}  \} $.%
\end{enumerate}

Let $\eta$ be a $H^{-1} (  \mathcal{D} )  $-valued random variable,
$\mathcal{F}_{\tau}$-measurable such that {(A$_{9}$)} is satisfied
and
\[
\mathbb{E} \bigl[\mathrm{e}^{p\sup_{s\in [  0,\tau ]  }\tilde{V}_{s}}\llvert \eta\rrvert ^{p}
\bigr]< \infty,\qquad\eta\in L^{1} ( \Omega \times\mathcal{D} ) ,\qquad
\mathrm{e}^{2\sup_{s\in [  0,\tau ]  }%
\tilde{V}_{s}}j ( \eta ) \in L^{1} ( \Omega\times
\mathcal{D} ) ,
\]
and the stochastic processes $\xi$ and $\zeta$ (from the martingale
representation theorem) be such that (\ref{Assumption A_8'}) holds.

From Theorem \ref{main result 1} it follows that, under the above assumptions,
backward SPDE (\ref{BSPDE 3}) has a unique solution $ (  Y,Z )
\in\mathcal{S}_{H^{-1} (  \mathcal{D} )  }^{0}\times\Lambda_{ (
H^{-1} (  \mathcal{D} )   )  ^{d}}^{0}$ such that $ (
Y (  t )  ,Z (  t )   )  = (  \xi_{t},\zeta
_{t} )  = (  \eta,0 )  $, for $t\geq\tau$, and
\begin{enumerate}[(iii)]
\item[(i)]  $Y ( t,x ) +{\int_{t}^{T}}\Delta U ( s,x )\,\mathrm{d}Q_{s}=Y ( T,x
)
+{\int_{t}^{T}}\Phi  ( s,Y ( s,x ) ,Z ( s,x
)  )\,\mathrm{d}Q_{s}-{\int_{t}^{T}}Z ( s,x )\,\mathrm{d}W_{s}$, in  $[ 0,T ] \times\mathcal{D}$, a.s.,

\item[(ii)] $Y ( t,x ) \in\Dom ( \partial j ) , \mathrm{d}\mathbb{P}\times \mathrm{d}t\times \mathrm{d}x$
a.e.,

\item[(iii)] $U ( t,x ) \in\partial j  ( Y ( t,x )  ) , \mathrm{d}\mathbb{P}\times \mathrm{d}t
\times \mathrm{d}x$ a.e.,

\item[(iv)] $\mathrm{e}^{2\tilde{V}}j ( Y ) \in L^{\infty}  ( 0,T;L^{1}
( \Omega\times\mathcal{D} )  ) , \forall T>0$,

\item[(v)] $\mathbb{E} {\int_{0}^{\tau}}\mathrm{e}^{2\tilde{V}_{s}%
}
\llVert U ( s )  \rrVert _{H_{0}^{1} (  \mathcal{D}%
)  }^{2}\,\mathrm{d}Q_{s}<{\infty}$.
\end{enumerate}

\begin{appendix}
\section*{Appendix}\label{app}

In this section, we first present some useful and general estimates on $ (
Y,Z )  \in\mathcal{S}^{0} [  0,T ]  \times\Lambda^{0} (
0,T )  $ satisfying an identity of type
\[
Y_{t}=Y_{T}+\int_{t}^{T}\,
\mathrm{d}K_{s}- \int_{t}^{T}Z_{s}
\,\mathrm{d}W_{s},\qquad t\in [ 0,T ] , \ \mathbb{P}\mbox{-a.s.,}%
\]
where $K\in\mathcal{S}^{0} [  0,T ]  $ and $t\longmapsto K_{t} (
\omega )  $ is a bounded variation function, $\mathbb{P}$-a.s.

The following results are proved in monograph Pardoux and R\u{a}\c{s}canu \cite{pa-ra/13}, Annex C.

\setcounter{equation}{0}

Assume there exist: three progressively measurable increasing continuous
stochastic processes $D,R,N$ such that $D_{0}=R_{0}=N_{0}=0$, a progressively
measurable bounded variation continuous stochastic process $V$ with $V_{0}=0$,
some constants $a,p>1$ such that, as signed measures on $ [  0,T ]$:
%
\begin{equation}
\label{ineq D,R,N} \mathrm{d}D_{t}+ \langle Y_{t},\mathrm{d}K_{t}
\rangle \leq \bigl( \mathbh{1}_{p\geq
2}\,\mathrm{d}R_{t}+
\llvert Y_{t}\rrvert \,\mathrm{d}N_{t}+\llvert
Y_{t}\rrvert ^{2}\,\mathrm{d}V_{t} \bigr) +
\frac{n_{p}}{2a}\llVert Z_{t}\rrVert ^{2}\,\mathrm{d}t,
\end{equation}
where $n_{p}= (  p-1 )  \wedge1$.

Let $\llVert  \mathrm{e}^{V}Y\rrVert  _{ [  t,T ]  }:=\sup_{s\in [
t,T ]  }\llvert  \mathrm{e}^{V_{s}}Y_{s}\rrvert  $.

\begin{proposition}\label{prop 1 Appendix}
Assume \textup{(\ref{ineq D,R,N})} and that
\[
\mathbb{E} \bigl\llVert Y\mathrm{e}^{V} \bigr\rrVert
_{ [  0,T ]  }^{p}%
+\mathbb{E} \biggl(\int
_{0}^{T}\mathrm{e}^{2V_{s}}
\mathbh{1}_{p\geq2}\,\mathrm{d}R_{s} \biggr)^{p/2}%
+ \mathbb{E} \biggl(\int_{0}^{T}
\mathrm{e}^{V_{s}}\,\mathrm{d}N_{s} \biggr)^{p}<
\infty.
\]
Then there exists a positive constant $C=C (  a,p )  $ such that,
$\mathbb{P}$-a.s., for all $t\in [  0,T ]  $:
\begin{eqnarray*}
&&\mathbb{E}^{\mathcal{F}_{t}} \biggl[\bigl\|
\mathrm{e}^{V}Y\bigr\|_{ [  t,T ]
}^{p}+ \biggl(\int
_{t}^{T}\mathrm{e}^{2V_{s}}\,
\mathrm{d}D_{s} \biggr)^{p/2}+ \biggl(\int
_{t}^{T}\mathrm{e}^{2V_{s}%
} \llVert
Z_{s}\rrVert ^{2}\,\mathrm{d}s \biggr)^{p/2}
\biggr]
\\
&&\qquad{}+\mathbb{E}^{\mathcal{F}_{t}} \biggl[\int_{t}^{T}
\mathrm{e}^{pV_{s}%
} \llvert Y_{s}\rrvert ^{p-2}
\mathbh{1}_{Y_{s}\neq0}\,\mathrm{d}D_{s}+ \int_{t}%
^{T}\mathrm{e}^{pV_{s}}\llvert Y_{s}\rrvert
^{p-2}\mathbh{1}_{Y_{s}\neq
0}\llVert Z_{s}\rrVert
^{2}\,\mathrm{d}s \biggr]
\\
&&\quad\leq C \mathbb{E}^{\mathcal{F}_{t}} \biggl[ \bigl\llvert
\mathrm{e}^{V_{T}%
}Y_{T} \bigr\rrvert ^{p}+ \biggl(
\int_{t}^{T}\mathrm{e}^{2V_{s}}
\mathbh{1}_{p\geq2}\,\mathrm{d}R_{s} \biggr)^{p/2}+
\biggl(\int_{t}^{T}\mathrm{e}^{V_{s}}\,
\mathrm{d}N_{s} \biggr)^{p} \biggr].
\end{eqnarray*}
In particular for all $t\in [  0,T ]  $:
\[
\llvert Y_{t}\rrvert ^{p}\leq C \mathbb{E}^{\mathcal{F}_{t}}
\bigl[ \bigl( \llvert Y_{T}\rrvert ^{p}+\mathbh{1}_{p\geq2}R_{T}^{p}%
+N_{T}^{p} \bigr) \mathrm{e}^{p\| (  V_{\cdot}-V_{t} )  ^{+}\|_{ [
t,T ]  }} \bigr] ,\qquad
\mathbb{P}\mbox{-a.s.}%
\]
\end{proposition}

As a simple consequence we can deduce, from the above proposition, an estimate
for the stochastic processes $ (  \xi,\zeta )  $ associated to $\eta$
as in Proposition \ref{mart repr th}:

\begin{corollary}
Let $ (  V_{t} )  _{t\geq0}$ be a bounded variation and continuous
p.m.s.p. with $V_{0}=0$, $\eta\dvtx \Omega\rightarrow H$ a random variable such that $\mathbb{E} (\mathrm{e}^{p\sup_{s\in [  0,T ]  }V_{s}}\llvert
\eta\rrvert  ^{p} )<\infty$ and $ (  \xi,\zeta )
\in\mathcal{S}^{0}\times\Lambda^{0} (  0,{\infty} )  $ the
unique solution of the following equation (see the martingale representation
formula ({\ref{mart repr th 1}})): $\xi_{s}=\mathbb{E}^{\mathcal{F}%
_{T}}\eta-{\int_{s}^{T}}\zeta_{r}\,\mathrm{d}W_{r}$, $s\in [  0,T ]  $, a.s.
Therefore, there exists $C=C (  p )  >0$ such that for all
$t\in [  0,T ]  $,
%
\begin{equation}
\label{ineq_xi_zeta} \mathbb{E}^{\mathcal{F}_{t}}\sup_{s\in [  t,T ]  }
\mathrm{e}^{pV_{s}%
}\llvert \xi_{s}\rrvert ^{p}+
\mathbb{E}^{\mathcal{F}_{t}} \biggl({\int_{t}^{T}}
\mathrm{e}^{2V_{s}}\llvert \zeta_{s}\rrvert ^{2}\,
\mathrm{d}s \biggr)^{p/2}\leq C \mathbb{E}^{\mathcal{F}_{t}} \bigl(
\mathrm{e}^{p\sup_{s\in [  t,T ]  }V_{s}%
}\llvert \eta\rrvert ^{p} \bigr). %
\end{equation}
\end{corollary}

\begin{pf}
We see at once that the stochastic pair $ (  \xi,\zeta )  $ satisfy
equation $\xi_{s}=\xi_{T}-\int_{s}^{T}\zeta_{r}\,\mathrm{d}W_{r}$, $s\in[0,T]$ a.s. For any fixed $t\in [  0,T ]  $ let $ (
\bar{V}_{s}^{t} )  _{s\in [  0,T ]  }$ be the increasing
continuous p.m.s.p. defined by $\bar{V}_{s}^{t}=V_{t}, s<t$, and $\bar{V}%
_{s}^{t}=\sup_{r\in [  t,s ]  }V_{r}, s\geq t$. Applying Jensen's
inequality and Proposition \ref{prop 1 Appendix} for $ (  \xi
,\zeta )  $ (which satisfies an inequality of type (\ref{ineq D,R,N}),
with $K=0$ and $R=N=0$), we deduce that for all $p>1$, there exists
$C=C (  p )  >0$ such that
\begin{eqnarray*}
&&\mathbb{E}^{\mathcal{F}_{t}}\sup_{s\in [  t,T ]  }\mathrm{e}^{pV_{s}%
}
\llvert \xi_{s}\rrvert ^{p}+\mathbb{E}^{\mathcal{F}_{t}}%
\biggl({\int_{t}^{T}}\mathrm{e}^{2V_{s}}
\llvert \zeta_{s}\rrvert ^{2}\,\mathrm{d}s
\biggr)^{p/2}
\\
&&\quad\leq \mathbb{E}^{\mathcal{F}_{t}}\sup_{s\in [  t,T ]
}
\mathrm{e}^{p\bar{V}_{s}^{t}}\llvert \xi_{s}\rrvert ^{p} +
\mathbb{E}^{\mathcal{F}_{t}} \biggl({\int_{t}^{T}}
\mathrm{e}^{2\bar
{V}_{s}^{t}} \llvert \zeta_{s}\rrvert ^{2}\,
\mathrm{d}s \biggr)^{p/2}
\\
&&\quad\leq C \mathbb{E}^{\mathcal{F}_{t}} \bigl(\mathrm{e}^{p\bar{V}_{T}^{t}}\llvert
\xi _{T}\rrvert ^{p} \bigr)\leq C \mathbb{E}^{\mathcal{F}_{t}}
\bigl(\mathrm{e}^{p\sup
_{s\in [  t,T ]  }\llvert  V_{s}\rrvert  } \llvert \eta\rrvert ^{p} \bigr).
\end{eqnarray*}
\upqed
\end{pf}

Let us now discuss the existence and uniqueness of a solution for the backward
stochastic equation of the form
%
\begin{equation}
\label{GBSDE} Y_{t}=\eta+\int_{t}^{T}
\Phi ( s,Y_{s},Z_{s} )\,\mathrm{d}Q_{s}-\int
_{t}%
^{T}Z_{s}\,
\mathrm{d}B_{s},\qquad \mbox{a.s., } \forall t\in [ 0,T ] . %
\end{equation}

We will need the following basic assumptions:

\begin{enumerate}[(A$_{5}^{\prime}$)]
\item[(A$_{3}^{\prime}$)] \textit{the process }$ \{
Q_{t}\dvtx t\geq0 \}  $\textit{ is a progressively measurable increasing
continuous stochastic process such that }$Q_{0}=0$, \textit{and }$ \{
\alpha_{t}\dvtx t\geq0 \}  $\textit{ is a real positive p.m.s.p. such that
}$\alpha\in [  0,1 ]  $\textit{ and }$\mathrm{d}t=\alpha_{t}\,\mathrm{d}Q_{t} $;

\item[(A$_{4}^{\prime}$)] \textit{the function }$\Phi
\dvtx \Omega\times[0,\infty)\times H\times L_{2} (  H_{1},H )
\rightarrow H$ \textit{is such that}
\[
\cases{\displaystyle \Phi ( \cdot,\cdot,y,z ) \mbox{ \textit{is p.m.s.p.,}}
\qquad \forall ( y,z ) \in H\times L_{2} ( H_{1},H ) \mbox{,}
\cr
\displaystyle\Phi ( \omega,t,\cdot,\cdot ) \mbox{ \textit{is continuous
function,}}\qquad\mathrm{d} \mathbb{P}\otimes \mathrm{d}t\mbox{\textit{-a.e.,}}%
}
\]
\textit{and }$\mathbb{P}$\textit{-a.s., }$\int_{0}^{T}\Phi_{\rho}^{\#} (
s )\,\mathrm{d}Q_{s}<\infty$, $\forall \rho\geq0$, \textit{where }$\Phi_{\rho
}^{\#} (  \omega,s )  :=\sup_{\llvert  u\rrvert
\leq\rho}\llvert  \Phi (  \omega,s,u,0 )  \rrvert  $;

\item[(A$_{5}^{\prime}$)] \textit{there exist a p.m.s.p.
}$\mu\dvtx \Omega\times[0,\infty)\rightarrow\mathbb{R}$\textit{ and a
function }$\ell\dvtx [0,\infty)\rightarrow{}[0,\infty)$\textit{ such that
}$\int_{0}^{T}\llvert \mu_{t}\rrvert \,\mathrm{d}Q_{t}+\int_{0}^{T}\ell^{2} (  t )\,\mathrm{d}t<\infty
$, $\mathbb{P}$\textit{-a.s. and, for all }$y,y^{\prime}\in H$, $z,z^{\prime
}\in L_{2} (  H_{1},H )  $,
\begin{eqnarray*}
\bigl\langle y^{\prime}-y,\Phi \bigl(t,y^{\prime},z \bigr)-\Phi(t,y,z)
\bigr\rangle &\leq&\mu_{t} \bigl\llvert y^{\prime}-y \bigr\rrvert
^{2},
\\
\bigl\llvert \Phi \bigl(t,y,z^{\prime} \bigr)-\Phi(t,y,z) \bigr\rrvert &
\leq& \ell ( t ) \alpha_{t} \bigl\llvert z^{\prime}-z \bigr\rrvert
.
\end{eqnarray*}
\end{enumerate}

Let $a>1$ and $V_{t}=\int_{0}^{t} (\mu_{s}+\frac{a}{2n_{p}%
} \ell^{2} (  s )  \alpha_{s} )\,\mathrm{d}Q_{s} $.

\begin{proposition}\label{prop 3 Appendix}
Let $p>1$ and $\eta\dvtx \Omega\rightarrow H$ be a random variable measurable with respect to $\sigma (   \{
\mathcal{F}_{t}\dvtx t\geq0 \}   )  $. Under the hypotheses
\emph{(A$_{3}^{\prime}$)--(A$_{5}^{\prime}$)}, if moreover,
%
\begin{equation}
\label{ineq 4 appendix} \mathbb{E} \bigl( \mathrm{e}^{pV_{T}}\llvert \eta\rrvert
^{p} \bigr) + \mathbb{E} \biggl(\int_{0}^{T}
\sup_{\llvert  y\rrvert  \leq\rho
} \bigl\llvert \mathrm{e}^{V_{t}}\Phi \bigl(
t,\mathrm{e}^{-V_{t}}y,0 \bigr) -\mu_{s}y \bigr\rrvert\,\mathrm{d}Q_{s} \biggr)^{p}<\infty,\qquad\forall\rho\geq0,
\end{equation}
there exists a unique pair $ (  Y_{t},Z_{t} )  _{t\geq0}%
\in\mathcal{S}^{0}\times\Lambda^{0}$ solution of the BSDE \emph{(\ref{GBSDE})} in the
sense that
\begin{enumerate}[(jj)] %
\item[(j)] $Y_{t}=\eta+\int_{t}^{T}\Phi ( s,Y_{s},Z_{s}
)\,\mathrm{d}Q_{s}-\int_{t}^{T}Z_{s}\,\mathrm{d}B_{s}%
$, a.s., $\forall t\in [ 0,T ] $,

\item[(jj)] $\mathbb{E}\|\mathrm{e}^{V}Y\|_{ [  0,T ]  }%
^{p}+\mathbb{E}  (\int_{0}^{T}\mathrm{e}^{2V_{s}}
\llvert Z_{s}\rrvert ^{2}\,\mathrm{d}s  )^{p/2}<\infty$.
\end{enumerate}
\end{proposition}

\setcounter{remark}{0}

\begin{remark}
If $ (  V_{t} )  _{t\geq0}$ is a deterministic process, then
assumption (\ref{ineq 4 appendix}) is equivalent to
\[
\mathbb{E} \bigl( \llvert \eta\rrvert ^{p} \bigr) +
\mathbb{E}%
\biggl(\int_{0}^{T}
\Phi_{\rho}^{\#}(s)\,\mathrm{d}Q_{s}
\biggr)^{p}< \infty.
\]
\end{remark}

\end{appendix}

\section*{Acknowledgements}
The authors would like to thank the referees for the attention in reading this paper and for their helpful suggestions and comments that
have led to the improvement of the paper.

The work was supported by grant POSDRU/89/1.5/S/49944 and respectively by IDEAS project, \textquotedblleft Deterministic and stochastic systems with
state constraints\textquotedblright, code 241/05.10.2011.


\printhistory

\end{document}